\pdfoutput=1
\documentclass[11pt]{article}

\usepackage{color}
\usepackage[math, proba]{package_rf}

\usepackage[top=4cm, bottom=3cm, left=2.5cm, right=2.5cm]{geometry}
\usepackage{abstract}
\usepackage{hyperref}
\usepackage{url}
\usepackage{shortcuts}

\title{A central limit theorem for the \slfv with selection}
\author{Rapha\"el Forien\footnote{CMAP - \'Ecole Polytechnique, Route de Saclay, 91128 Palaiseau Cedex, France, e-mail: \protect\url{raphael.forien@cmap.polytechnique.fr}}\: and Sarah Penington\footnote{Department of Statistics, University of Oxford, 1 South Parks Road, Oxford OX1 3TG, UK, e-mail: \protect\url{sarah.penington@sjc.ox.ac.uk}}}

\begin{document}

\maketitle

\begin{abstract}
We study the evolution of gene frequencies in a population living in $\R^d$, modelled by the \slfv with natural selection (\cite{barton_new_2010}, \cite{etheridge_rescaling_2014}).
We suppose that the population is divided into two genetic types, $a$ and $A$, and consider the proportion of the population which is of type $a$ at each spatial location.
If we let both the selection intensity and the fraction of individuals replaced during reproduction events tend to zero, the process can be rescaled so as to converge to the solution to a reaction-diffusion equation (typically the Fisher-KPP equation, \cite{etheridge_rescaling_2014}).
We show that the rescaled fluctuations converge in distribution to the solution to a linear \spde. 
Depending on whether offspring dispersal is only local or if large scale extinction-recolonization events are allowed to take place, the limiting equation is either the stochastic heat equation with a linear drift term driven by space-time white noise or the corresponding fractional heat equation driven by a coloured noise which is white in time.
If individuals are diploid (\textit{i.e.} either $AA$, $Aa$ or $aa$) and if natural selection favours heterozygous ($Aa$) individuals, a stable intermediate gene frequency is maintained in the population.
We give estimates for the asymptotic effect of random fluctuations around the equilibrium frequency on the local average fitness in the population.
In particular, we find that the size of this effect - known as the \textit{drift load} - depends crucially on the dimension $d$ of the space in which the population evolves, and is reduced relative to the case without spatial structure.

\end{abstract}

\section*{Introduction}

Consider a population distributed across a geographical space (typically of dimension one or two). 
Suppose that each individual carries one of several possible versions (or \textit{alleles}) of a gene. 
How do the different allele frequencies evolve with time and how are they shaped by the main evolutionary forces, such as natural selection and migration?
To answer this question, early models from population genetics were adapted by G. Mal\'ecot  \cite{malecot_mathematiques_1948}, S. Wright \cite{wright_isolation_1943} and M. Kimura \cite{kimura_stepping-stone_1953} to include spatial structure.
These spatial models either considered subdivided populations reproducing locally and exchanging migrants at each generation or made inconsistent assumptions about the distribution of individuals across space.

In this work, we focus on a mathematical model for populations evolving in a spatial continuum, the \slfv (SLFV for short), originally proposed in \cite{etheridge_drift_2008}.
The main feature of this model is that instead of each individual carrying exponential clocks determining its reproduction and death times, reproduction times are specified by a \ppp of extinction-recolonization events.
At each of these events, some proportion - often denoted $u$ - of the individuals present in the region affected by the event is replaced by the offspring of an individual (the \textit{parent}) chosen within this region. 
(The proportion $u$ which is replaced is called the \textit{impact parameter}.)
We shall only consider cases where the region affected is a ($d$-dimensional) ball, and the \ppp specifies the time, centre and radius of reproduction events.
(Since we consider scaling limits, minor changes to this assumption would not change our results.) 
Natural selection can be included in the SLFV by introducing an independent \ppp of selective events which give an advantage to a particular type.
Multiple \textit{potential parents} are chosen in the region affected by the event and one is chosen to be the parent and have offspring in a biased way depending on their types.
The \textit{selection parameter} determines the rate of this \ppp.
A comprehensive survey of recent developments related to the SLFV can be found in \cite{barton_modeling_2013}.

Several works have focussed on characterising the behaviour of this model over large space and time scales, in the special case where only two types (or two alleles) $a$ and $A$ are present in the population.
In this case the state of the process is given by a map $q_t : \R^d\to[0,1]$ defined Lebesgue almost everywhere, where $q_t(x)$ denotes the proportion of type $a$ at location $x$ and at time $t$. 
We shall first consider the simplest form of selection when individuals are \textit{haploid}, i.e. each individual has one copy of the gene.
At selective events, two potential parents are chosen and if their types are different, the parent is the one which has type $A$.
In \cite{etheridge_rescaling_2014}, rescaling limits of this form of the \slfv with selection (SLFVS) have been obtained when both the impact parameter and the selection parameter tend to zero.
Earlier results on the large scale behaviour of the SLFV had already been established in \cite{berestycki_large_2013} in the neutral case (\textit{i.e.} without selection), but keeping the impact parameter macroscopic. 
The behaviour of the SLFVS in the corresponding regime is studied in \cite{etheridge_brownian_2015} and \cite{etheridge_branching_2015}.

The limiting process obtained by \cite{etheridge_rescaling_2014} turns out to be deterministic as soon as $d\geq 2$, and, when the reproduction events have bounded radius, it is given by the celebrated Fisher-KPP equation,
\begin{equation} \label{fisher_kpp_intro}
\deriv*{f_t}{t} = \frac{1}{2}\Delta f_t - s f_t(1-f_t).
\end{equation}
This result fits the original interpretation of this equation proposed by R. A. Fisher as a model for the spread of advantageous genes in a spatially distributed population \cite{fisher_wave_1937}. 
The \slfv with selection (SLFVS) can thus be thought of as a refinement of the Fisher-KPP equation, combining spatial structure and a random sampling effect at each generation - what biologists call \textit{genetic drift}.

In the present work we prove a slightly stronger form of convergence to this deterministic rescaling limit. We also study the fluctuations of the allele frequency about (an approximation of) $(f_t)_{t\geq 0}$. We find that if the impact parameter is sufficiently small compared to the selection parameter and the fluctuations are rescaled in the right way then in the limit they solve the following \spde,
\begin{equation} \label{eq:fluctuations}
\ld z_t = \left[ \frac{1}{2}\Delta z_t - s(1-2f_t)z_t \right]\ld t + \sqrt{f_t(1-f_t)}\,\ld W_t,
\end{equation}
where W is space-time white noise, and $f$ is the solution of \eqref{fisher_kpp_intro}. More detailed statements with the precise conditions on the parameters of the SLFVS are given in Section~\ref{sec:resu}.

A very similar result was proved by F. Norman in the non-spatial setting \cite{norman_approximation_1975} (see also \cite{norman_central_1974}, \cite{norman_ergodicity_1977} and \cite{norman_limit_1975}). 
Norman considered the Wright-Fisher model for a population of size $N$ under natural selection (see \cite{etheridge_mathematical_2009} for an introduction to such models). 
Let $p^N_n$ denote the proportion of individuals not carrying the favoured allele at generation $n$, and suppose that the selection parameter is given by $s_N=\varepsilon_N s$, with $\varepsilon_N\to 0$ and $\varepsilon_N N\to\infty$ as $N\to \infty$. 
(At each generation, individuals choose a parent of the favoured type with probability $\frac{(1+s_N)(1-p^N_n)}{1+s_N(1-p^N_n)}$.)
Norman showed that, as $N\to\infty$, $p^N_{\floor{t/\varepsilon_N}}$ converges to $g_t$, which satisfies
\begin{equation*}
\deriv{g_t}{t} = - s g_t(1-g_t).
\end{equation*}
(In the weak selection regime - \textit{i.e.} when $Ns_N=\bigO{1}$ - one recovers the classical Wright-Fisher diffusion.)
Furthermore, the fluctuations of $p^N_{t/\varepsilon_N}$ around $g_t$ are of order $(N\varepsilon_N)^{-1/2}$. 
More precisely, for $t=n\varepsilon_N$, $n\in\N$, set
\begin{equation*}
Z^N(t) = (N\varepsilon_N)^{1/2}\left( p^N_{t/\varepsilon_N}-g_t \right),
\end{equation*}
and define $Z^N(t)$ for all $t\geq 0$ by linear interpolation.
Theorem 2 in \cite{norman_approximation_1975} states that, as $N\to\infty$, $\proc{Z^N(t)}{t\geq 0}$ converges to the solution of the following stochastic differential equation,
\begin{equation*}
\ld z_t = -s(1-2g_t)z_t\ld t + \sqrt{g_t(1-g_t)}\ld B_t,
\end{equation*}
where $\proc{B}$ is a standard Brownian motion; note that $\proc{z}$ is a Gaussian diffusion.
A similar regime in the case of a neutral model with mutations was already studied by W.~Feller in \citep[Section 9]{feller_diffusion_1951}, who identified the limiting diffusion for the fluctuations around the equilibrium frequency. 

Norman's result can be extended to other classical models from population genetics, and in particular to continuous-time processes such as the Moran model and the (non-spatial) $\Lambda$-Fleming-Viot process (introduced in \cite{bertoin_stochastic_2003}). 
The necessary tools can be found mainly in \citep[Chapter 11]{ethier_markov_1986} (see also Chapter 6 of the same book) and in \cite{kurtz_limit_1971}.
In this paper we adapt these methods to the setting of the \slfv, with the necessary tools for \spde{s} taken from \cite{walsh_introduction_1986} (see also \cite{mueller_stochastic_1995} and \cite{de_masi_reaction-diffusion_1986}).

We also consider a second regime for the SLFVS to allow large scale extinction-recolonization events; we let the radius of reproduction events follow an $\alpha$-stable distribution truncated at zero. 
For this regime, as in \cite{etheridge_rescaling_2014}, we find the Fisher-KPP equation with non-local diffusion as a rescaling limit (\textit{i.e.} with a fractional Laplacian instead of the usual Laplacian). 
The Laplacian is also replaced by a fractional Laplacian in \eqref{eq:fluctuations}, the equation satisfied by the limiting fluctuations, and the noise $W$ becomes a coloured noise with spatial correlations of order $\abs{x-y}^{-\alpha}$ (see Subsection~\ref{subsec:resu:stable_radii}).

These results are valid for a general class of selection mechanisms, with modified versions of~\eqref{fisher_kpp_intro} and~\eqref{eq:fluctuations} (and our proof will cover the general case). As an application of our results on the fluctuations, we turn to a particular kind of selection mechanism.
Suppose a given gene is present in two different forms - denoted $A_1$ and $A_2$ - within a population.
Suppose also that each individual carries two copies of this gene (each inherited from one of two parents).
We say that individuals are \textit{diploid}, and \textit{homozygous} individuals are those who carry two copies of the same type ($A_1A_1$ or $A_2A_2$) while \textit{heterozygous} individuals carry one copy of each type ($A_1A_2$). 
\textit{Overdominance} occurs when the relative fitnesses of the three possible genotypes are as follows,
\[\begin{array}{ccc}
A_{1}A_{1} & A_{1}A_{2} & A_{2}A_{2} \\
1-s_{1} & 1 & 1-s_{2},
\end{array}\]
where $s_1$, $s_2>0$.
In words, heterozygous individuals produce more offspring than both types of homozygous individuals. 
In this setting, in an infinite population a stable intermediate allele frequency is expected to be maintained, preventing either type from disappearing.
If $q$ is the frequency of type $A_1$ and $p=1-q$ that of type $A_2$ and if mating is random, the respective proportions of the three genotypes will be $q^2$, $2qp$, $p^2$, hence the population cannot remain composed exclusively of heterozygous individuals. 
As a consequence, even when the stable equilibrium is reached, the mean fitness of the population will not be as high as the highest possible individual fitness (\textit{i.e.} that of heterozygous individuals). 
This fitness reduction is referred to as the \textit{segregation load}.

In finite populations, because of finite sample size, the allele frequency is never exactly at its optimum. 
This was the subject of a work by A. Robertson \cite{robertson_reduction_1970} who considered this specific configuration of the relative fitnesses. 
He argued that the mean fitness in a \textit{panmictic} population (\textit{i.e.} one with no spatial structure) with finite but relatively large size $N$ is reduced by a term of order $(4N)^{-1}$, irrespective of the strength of selection. 
This is due to a trade-off between genetic drift and natural selection. 
The stronger selection is, the quicker the allele frequency is pushed back to the equilibrium, but at the same time even a small step away from the optimal frequency is very costly in terms of mean fitness. 
On the other hand, if natural selection is relatively weak, the allele frequency can wander off more easily, but the  mean fitness of the population decreases more slowly. 
This reduction in the mean fitness due to genetic drift - which is added to the reduction from the segregation load - is called the \textit{drift load}.

Robertson's result can be made rigorous using tools found in \cite{norman_central_1974} and \cite{norman_markovian_1974}. 
We adapt these to our setting and study the same effect in spatially structured populations. 
We find that the spatial structure significantly reduces the drift load, in a way that depends crucially on dimension. 
It turns out that migration prevents the allele frequencies from straying too far away from the equilibrium frequency, because incoming migrants are on average close to this equilibrium.

The paper is laid out as follows. 
We define the \slfv for a haploid model with general frequency dependent selection and for a diploid model of overdominance in Section~\ref{sec:model}. 
In Section~\ref{sec:resu} we state the main convergence results for the SLFVS in the bounded radius and stable radius regimes and we present our estimate of the drift load in spatially structured populations.
In Section~\ref{sec:martingale_pbs}, we present the main ingredient of the proof: a martingale problem satisfied by the SLFVS.
At the end of Subsections~\ref{subsec:mart_pb:fixed} and~\ref{subsec:mart_pb:stable}, we state more general results on solutions to these martingale problems which imply our convergence results for the SLFVS.
Most of the remainder of the paper is dedicated to the proofs of these results. 
The central limit theorem in the bounded radius case is proved in Section~\ref{sec:proof_brownian}, while the stable regime is dealt with in Section~\ref{sec:proof_stable} (the two proofs share the same structure, but differ in the details of the estimates). 
Finally, the asymptotics of the drift load are derived in Section~\ref{sec:proof_drift_load}.

\section*{Acknowledgements}

The authors would like to thank Alison Etheridge for many helpful suggestions and comments and Amandine V\'eber for helpful feedback on an early draft of this paper.

\section{Definition of the model} \label{sec:model}

\subsection{The state space of the \slfv with selection}\label{subsec:state_space}

We now turn to a precise definition of the underlying model, the spatial $\Lambda$-Fleming-Viot process with selection on $\R^d$, starting with the state space of the process. At each time $t\geq 0$, $\{ q_t(x):x\in \R^d\}$ is a random function such that
$$q_t(x):=\text{proportion of type }a\text{ alleles at spatial position }x\text{ at time }t,$$
which is in fact defined up to a Lebesgue null set of $\R^d$. More precisely, let $\Xi$ be the quotient of the space of Lebesgue-measurable maps $f:\R^d\rightarrow [0,1]$ by the equivalence relation
$$ f\sim f' \iff Leb (\{x\in \R^d:f(x)\neq f'(x)\})=0.$$
We endow $\Xi$ with the topology of vague convergence: letting $\dual{f}=\intrd f(x) \phi(x) \ld x$, a sequence $\proc{f_n}{n}$ converges vaguely to $f\in\Xi$ if and only if $\dual{f_n}\cvgas{n}\dual{f}$ for any continuous and compactly supported function $\phi:\R^d\to\R$. A convenient metric for this topology is given by choosing a separating family $(\phi_n)_{n\geq 1}$ of smooth, compactly supported functions which are uniformly bounded in $L^1(\R^d)$. Then for $f,g\in \Xi$, 
\begin{equation} \label{d_defn_family}
d_\Xi(f,g)=\sum_{n\geq 1}\frac{1}{2^n}\left| \dual{f}{\phi_n}-\dual{g}{\phi_n}\right|
\end{equation}
defines a metric for the topology of vague convergence on $\Xi$. The SLFVS up to time $T$ is then going to be a $\sko{\Xi}$-valued random variable: a $\Xi$-valued process with \cadlag paths.
\begin{definition}\label{def_metric_xi}
For $T>0$, let $f,g\in\sko{\Xi}$ be a pair of \cadlag maps $\proc{f_t}{0\leq t\leq T}$, $\proc{g_t}{0\leq t\leq T}$ from $[0,T]$ to $\Xi$. Then
\begin{equation*}
d\left( f,g \right) = \sup_{t\in[0,T]}d_\Xi(f_t,g_t)
\end{equation*}
is a metric for the topology of uniform convergence on $\sko{\Xi}$.
\end{definition}
For more details, see Section 2.2 of \cite{veber_spatial_2015}.

\subsection{The \slfv with selection}\label{subsec:model}

Let us now define the dynamics of the process. Let $u\in (0,1]$ and $s\in [0,1]$, and let $\mu(dr)$ be a finite measure on $(0,\infty)$ satisfying
\begin{equation} \label{slfv_existence}
\int_0^\infty r^d \mu(\ld r)<\infty.
\end{equation} 
For $m\in \N$ and $w\in [0,1]$, let $\vec{B}^m_w$ be a vector of $m$ independent random variables taking the value $a$ with probability $w$ and $A$ otherwise.
Then let $F:[0,1]\rightarrow \R$ be a polynomial such that for some $m\in \N$ and $p:\{a,A\}^m\to [0,1]$, for each $w\in [0,1]$, 
\begin{equation} \label{selection_fn_cond}
w-F(w)=\E{p(\vec{B}^m_w)}.
\end{equation}
\begin{definition}[SLFVS, haploid case with general frequency dependent selection] \label{definition_slfv_selection_haploid}
Let $\Pi$ and $\Pi^S$ be two independent Poisson point processes on $\R_+\times \R^d\times (0,\infty)$ with intensity measures $(1-s)\,dt\otimes dx\otimes \mu(dr)$ and $s\,dt\otimes dx\otimes \mu(dr)$ respectively. 
The spatial $\Lambda$-Fleming-Viot process with selection for a haploid population with impact parameter $u$, radius of reproduction events given by $\mu(dr)$, selection parameter $s$ and selection function $F$ is defined as follows.
If $(t,x,r)\in \Pi$, a neutral event occurs at time $t$ within the ball $B(x,r)$:
\begin{enumerate}
\item Choose a location $y$ uniformly at random in $B(x,r)$ and sample a parental type $k\in \{a,A\}$ according to $q_{t^-}(y)$ (i.e. $k=a$ with probability $q_{t^-}(y)$).
\item Update $q$ as follows:
\begin{equation}\label{update_w}
\forall z \in \R^d, q_t(z)=q_{t^-}(z)+u\1{|x-z|<r}(\1{k=a}-q_{t^-}(z)). 
\end{equation}
\end{enumerate}
Similarly, if $(t,x,r)\in \Pi^S$, a selective event occurs at time $t$ inside $B(x,r)$:
\begin{enumerate}
\item Choose $m$ locations $y_1,\ldots,y_m$ independently uniformly at random in $B(x,r)$, sample a type $k_i$ at each location $y_i$ according to $q_{t^-}(y_i)$ and then let $k=a$ with probability $p(k_1,\ldots, k_m)$ and $k=A$ otherwise.
\item Update $q$ as in \eqref{update_w}.
\end{enumerate}
\end{definition}

Note that if we let $w=|B(x,r)|^{-1}\int_{B(x,r)}q_{t^-}(z)\,dz$, then at a neutral reproduction event, $\P{k=a}=w$ and at a selective event, $\P{k=a}=w-F(w)$. 

\begin{remark} \label{exist_uniq}
The existence of a unique $\Xi$-valued process following these dynamics under condition \eqref{slfv_existence} is proved in \citep[Theorem 1.2]{etheridge_rescaling_2014} in the special case $F(w)=w(1-w)$ (in the neutral case $s=0$, this was done in \cite{barton_new_2010}). In our general case, the condition on $w-F(w)$ allows us to define a dual process and hence prove existence and uniqueness in the same way as in \cite{etheridge_rescaling_2014}.
\end{remark}
We shall consider two different distributions $\mu$ for the radii of events,
\begin{enumerate}[label=\roman*)]
\item the fixed radius case~: $\mu(\ld r) = \delta_R(\ld r)$ for some $R>0$,
\item the stable radius case~: $\mu(\ld r)=\frac{\1{r\geq 1}}{r^{d+\alpha+1}}\ld r$ for a fixed $\alpha\in (0,2\wedge d)$.
\end{enumerate} 
In each case, \eqref{slfv_existence} is clearly satisfied.

We now give two variants of this definition corresponding to the two selection mechanisms discussed in the introduction. We begin with a model for a selective advantage for $A$ alleles in haploid reproduction.
\begin{definition}[SLFVS, haploid model, genic selection] \label{definition_slfv_selection}
The spatial $\Lambda$-Fleming-Viot process with genic selection with impact parameter $u$, radius of reproduction events given by $\mu(dr)$ and selection parameter $s$ is defined as in Definition~\ref{definition_slfv_selection_haploid} with $F(w)=w(1-w)$. In this case, $m=2$ and the function $p$ equals simply
\begin{equation*}
p(k_1, k_2) = \1{k_1 = k_2 = a}.
\end{equation*}
In other words, during selective reproduction events, two types are sampled in $B(x,r)$ and $k=a$ if and only if both types are $a$.
\end{definition}

We now define a variant of the SLFVS to model overdominance. 
Individuals are diploid and we study a gene which is present in two different forms within the population, denoted $A_{1}$ and $A_{2}$. 
For $t\geq 0$ and $x\in \R^d$, let
$$ q_{t}(x):= \text{ the proportion of the allele type }A_{1}\text{ at location }x\text{ at time }t.$$
(If $p_1$ is the proportion of $A_1 A_1$ indiduals and $p_H$ is the proportion of $A_1 A_2$ heterozygous individuals, then $q = p_1 + \frac{1}{2} p_H$.)
We assume that the relative fitnesses of the different genotypes are as follows:
\[\begin{array}{ccc}
A_{1}A_{1} & A_{1}A_{2} & A_{2}A_{2} \\
1-s_{1} & 1 & 1-s_{2}.
\end{array}\]
In other words, for an event $(t,x,r)$ in the SLFVS with $w=|B(x,r)|^{-1}\int_{B(x,r)}q_{t^-}(z)\,dz$, we want to choose parental types $(k_1, k_2)\in \{A_1,A_2\}^2$ at random with 
\begin{equation} \label{wright-fisher-selection}
\begin{aligned}
&\P{\{k_1,k_2\}=\{A_1, A_1\}}=\tfrac{(1-s_1)w^2}{1-s_1w^2-s_2(1-w)^2},
\\
&\P{\{k_1,k_2\}=\{A_1, A_2\}}=\tfrac{2w(1-w)}{1-s_1w^2-s_2(1-w)^2},
\\
&\P{\{k_1,k_2\}=\{A_2, A_2\}}=\tfrac{(1-s_2)(1-w)^2}{1-s_1w^2-s_2(1-w)^2}.
\end{aligned}
\end{equation}
Further, we suppose that, with probability $\nu_{1}$, the type $A_{1}$ alleles produced mutate to type $A_{2}$, and that, with probability $\nu_{2}$, the type $A_{2}$ mutate to type $A_{1}$ (this is a technical assumption to ensure that $q_t(x)\notin \{ 0,1\}$; we shall assume that $\nu_1$ and $\nu_2$ are small).

We are going to be interested in small values of $\nu_1$, $\nu_2$, $s_1$ and $s_2$.
We thus define the following model, which is an approximation of the one described by \eqref{wright-fisher-selection} to the first order in $s_i$ and $\nu_i$.

\begin{definition}[SLFVS, overdominance] \label{definition_slfv_diploid}
Suppose that $\nu_1+\nu_2+s_1+s_2 < 1$.
Let $\Pi$, $\Pi^{S_i}$ and $\Pi^{\nu_i}$, $i=1,2$ be five independent \ppp{s} on $\R_+\times \R^d\times (0,\infty)$ with respective intensity measures $(1-s_1-s_2-\nu_1-\nu_2) \, dt\otimes dx\otimes \mu(dr)$, $s_i \,dt\otimes dx\otimes \mu(dr)$ and $\nu_i \,dt\otimes dx\otimes \mu(dr)$.
The spatial $\Lambda$-Fleming-Viot process with overdominance with impact parameter $u$, radius of reproduction events given by $\mu$, selection parameters $s_1$, $s_2$ and mutation parameters $\nu_1$, $\nu_2$ is defined as follows. 
If $(t,x,r)\in\Pi$, a neutral event occurs at time $t$ in $B(x,r)$: 
\begin{enumerate}
\item Pick two locations $y_1$ and $y_2$ uniformly at random within $B(x,r)$ and sample one parental type $k_i \in \left\lbrace A_1, A_2 \right\rbrace$ at each location according to $q_{t^-}(y_i)$, independently of each other.
\item Update $q$ as follows:
\begin{equation} \label{eq:overdom_update}
\forall z\in \R^d, \quad q_{t}(z) = q_{t^{-}}(z) + u\inball{x}{z}\left(\tfrac{1}{2}(\1{k_1=A_{1}}+\1{k_2=A_{1}}) - q_{t^{-}}(z)\right).
\end{equation}
\end{enumerate}
If $(t,x,r)\in\Pi^{S_i}$, a selective event occurs at time $t$ in $B(x,r)$: 
\begin{enumerate}
\item Pick four locations uniformly at random within $B(x,r)$ and sample one type at each location, forming two pairs of types. If one pair is $\left\lbrace A_i, A_i \right\rbrace$, let $\left\lbrace k_1, k_2 \right\rbrace$ be the other pair; otherwise pick one pair at random, each with probability $1/2$. (If the two sampled pairs are $\left\lbrace A_i, A_i \right\rbrace$, then $\left\lbrace k_1, k_2 \right\rbrace = \left\lbrace A_i, A_i \right\rbrace$.)
\item Update $q$ as in \eqref{eq:overdom_update}.
\end{enumerate}
If $(t,x,r)\in\Pi^{\nu_i}$, a mutation event occurs at time $t$ in $B(x,r)$: 
\begin{enumerate}
\item Set $\left\lbrace k_1, k_2 \right\rbrace = \left\lbrace A_{3-i}, A_{3-i} \right\rbrace$, irrespective of the state of $q_{t^-}$. (In other words we suppose that the $A_i$ genes of the offspring mutate to type $A_{3-i}$.)
\item Update $q$ as in \eqref{eq:overdom_update}.
\end{enumerate}
\end{definition}
\begin{remark}
Similarly to the haploid case, existence and uniqueness for this process can be proved as in \cite{etheridge_rescaling_2014} using a dual process.
\end{remark}
We shall see in Section \ref{sec:martingale_pbs} that this process satisfies essentially the same martingale problem as the general haploid process in Definition~\ref{definition_slfv_selection_haploid} with $$ F(w)=w(1-w)(w-\frac{s_2}{s_1 +s_2})+\frac{\nu_1}{s_1+s_2}w -\frac{\nu_2}{s_1+s_2}(1-w). $$

\section{Statement of the results}\label{sec:resu}

In this section, we present our main results.
We consider the SLFVS as in Definitions~\ref{definition_slfv_selection_haploid} and~\ref{definition_slfv_diploid}, and we let the impact parameter and the selection and mutation parameters tend to zero.
On a suitable space and time scale (depending on the regime of the radii of reproduction events) the process $\proc{q^N}$ converges to a deterministic process.
We also characterise the limiting fluctuations of $\proc{q^N}$ about an approximation to this deterministic process as the solution to a \spde.

\subsection{Fixed radius of reproduction events}\label{subsec:resu:fixed_radius}

We begin by considering the regime in which the radii of the regions affected by reproduction events are bounded. 
We shall only give the proof in the case of fixed radius events; the proof for bounded radius events is the same but notationally awkward. 
Fix $u, s \in(0,1]$ and $R>0$, and choose $w_{0}:\R^{d}\to[0,1]$ with uniformly bounded spatial derivatives of up to the fourth order. 
Take two sequences $\proc{\varepsilon_{N}}{N\geq 1}$, $\proc{\delta_{N}}{N\geq 1}$ of positive real numbers in $(0,1]$ decreasing to zero, and set 
\begin{align*}
s_{N}=\delta_{N}^{2}s, && u_N=\varepsilon_{N}u, && r_N = \delta_N R, && q_{0}^{N}(x) = w_{0}(\delta_{N}x).
\end{align*}
Let $\mu(dr)=\delta_R$, and let $F:\R\to\R$ be a smooth, bounded function with bounded first and second derivatives such that $F|_{[0,1]}$ satisfies \eqref{selection_fn_cond} for some $m\in \N$ and $p:\{a,A\}^m\to [0,1]$.
Then for $N\geq 1$, 
let $\proc{q^{N}}$ be the \slfv with selection following the dynamics of Definition~\ref{definition_slfv_selection_haploid} with impact parameter $u_N$, radius of reproduction events $R$, selection parameter $s_{N}$ and selection function $F$ started from the initial condition $q_{0}^{N}$.


Define the rescaled process $\proc{\mathbf{q}^N}$ by setting:
\begin{equation}\label{definition_rescaled_process}
\forall x\in\R^{d},t\geq 0, \:\mathbf{q}^{N}_{t}(x) = q^{N}_{t/(\varepsilon_{N}\delta_{N}^{2})}(x/\delta_{N}).
\end{equation}

We justify this scaling as follows.
Consider an individual sitting at location $x$ at time $t$. 
It finds itself within a region affected by a reproduction event at rate $|B(0,R)|$.
The probability that it dies and is replaced by a new individual is $u_N = \varepsilon_N u$, so, if we rescale time by $1/\varepsilon_N$, this will happen at rate $\bigO{1}$.
Also, we are going to see later (see Section~\ref{subsec:mart_pb:fixed}) that the reproduction events act like a discrete heat flow on the allele frequencies.
We rescale time further by $1/\delta_N^2$ and space by $1/\delta_N$, which corresponds to the diffusive scaling of this discrete heat flow.
Since selective events also take place at rate $\bigO{\delta_N^2}$, this is the right scaling to consider in order to observe the effects of both migration and selection in the limit.
(Due to this diffusive scaling we shall refer to this regime as the Brownian case.)

We need to introduce some notation. Let $L^{1,\infty}(\R^d)$ denote the space of bounded and integrable real-valued functions on $\R^d$. For $r>0$, we set $V_{r} = \abs{B(0,r)}$ and, for $x,y\in\R^{d}$, 
\begin{equation} \label{ball_intersect}
V_{r}(x,y)=\abs{B(x,r)\cap B(y,r)}.
\end{equation}
For $\phi:\R^{d}\to\R$ and $x\in\R^{d}$, set
\begin{equation*}
\avg{\phi}(x,r)=\frac{1}{V_{r}}\int_{B(x,r)}\phi(y)\ld y.
\end{equation*}
When there is no ambiguity, we shall not specify the radius $r$ and simply write $\avg{\phi}(x)$. 
This notation will be used throughout this paper and formulae will routinely involve averages of averages, \textit{etc}. 
For example we also write
\begin{equation} \label{defn_double_avg}
\dbavg{\phi}(x,r) = \frac{1}{V_r^2}\intbr{x}\intbr{y}\phi(z)\ld z\ld y.
\end{equation}
Let us define a linear operator $\L$ by setting
\begin{equation}\label{definition_calL}
\L\phi(x) = \frac{d+2}{2r^{2}}\left(\dbavg{\phi}(x,r)-\phi(x)\right).
\end{equation}
Finally let $\calS(\R^d)$ denote the Schwartz space of rapidly decreasing smooth functions on $\R^d$, whose derivatives of all orders are also rapidly decreasing. 
Accordingly, let $\calS'(\R^d)$ denote the space of tempered distributions.

Let $f^{N} : \R_+ \times \R^d \to \R$ be a deterministic function defined as the solution to
\begin{equation}\label{definition:centering_term_fixed_radius}
\left\lbrace
\begin{aligned}
& \deriv*{\ft}{t} = uV_{R}\left[ \frac{2R^{2}}{d+2}\L{r_N}\ft - s\avg{F(\avg{\ft})}{r_N}\right], \\
& f^N_{0} = w_{0}.
\end{aligned}
\right.
\end{equation}
One can check that this defines a unique function by a Picard iteration.

As stated in the introduction, the \slfv with genic selection with fixed radius of reproduction events converges, under what can be considered a diffusive scaling, to the solution of the Fisher-KPP equation (as in \cite{etheridge_rescaling_2014} for $d\geq 2$) while the limiting fluctuations are given by the solution to a \spde which generalises the result obtained in \cite{norman_approximation_1975}. 
We can now give a precise statement of this result for general frequency dependent selection.
The same result holds for radius distributions given by a finite measure $\mu$ on a bounded interval.
 
\begin{thm}[Central Limit Theorem for the SLFVS with fixed radius of reproduction events]\label{thm:clt_slfv_fixed_radius}
Suppose that $\varepsilon_{N}=\littleO{\delta_{N}^{d+2}}$, then the process $\proc{\mathbf{q}^{N}}$ converges in $L^{1}$ and in probability (for the metric $d$ of Definition~\ref{def_metric_xi}) to the deterministic solution of the following PDE,
\begin{equation*}
\left\lbrace
\begin{aligned}
& \deriv*{f_{t}}{t} = uV_{R}\left[\frac{R^{2}}{d+2}\Delta f_{t} - s F(f_{t})\right], \\
& f_{0} = w_{0}.
\end{aligned}
\right.
\end{equation*}
In addition, 
\begin{equation*}
\zt = (\varepsilon_{N}\delta_{N}^{d-2})^{-1/2}(\mathbf{q}^{N}_{t}-\ft)
\end{equation*}
defines a sequence of distribution-valued processes converging in distribution in $\sko{\calS'(\R^{d})}$ to the solution of the following \spde,
\begin{equation*}
\left\lbrace
\begin{aligned}
&\ld z_{t} = uV_{R} \left[ \frac{R^{2}}{d+2} \Delta z_{t} - s F'(f_{t} )z_{t} \right]\ld t + uV_{R} \sqrt{f_{t}(1-f_{t})} \ld W_{t},\\
&z_0=0,
\end{aligned}
\right.
\end{equation*}
where $W$ is a space-time white noise.
\end{thm}
\begin{remark}
The impact parameter $u_N$ is inversely proportional to the neighbourhood size - \textit{i.e.} the probability that two individuals have a common parent in the previous generation (see Section 3.6 of \cite{barton_modeling_2013} for details). Hence, letting $u_N$ tend to zero corresponds to letting the neighbourhood size grow to infinity.
\end{remark}
We shall show in Section \ref{sec:martingale_pbs} that Theorem~\ref{thm:clt_slfv_fixed_radius} is a consequence of Theorem~\ref{thm:clt_brownian}. The latter is a result on sequences of solutions to a martingale problem and is proved in Section~\ref{sec:proof_brownian}. 

\begin{remark}
It would have been more natural to consider the fluctuations directly around the deterministic limit $\proc{f}$, but in fact the difference between $f_N$ and $f$ is too large (of order $\delta _N^2$, see Proposition~\ref{prop:convergence_centering_term}).
We have that $|Z^N_t-(\varepsilon_N \delta_N^{d-2})^{-1/2}(\mathbf q ^N_t -f_t)|= \mathcal O(\delta_N^2(\varepsilon_N \delta_N^{d-2})^{-1/2})$ but if $\varepsilon_{N}=\littleO{\delta_{N}^{d+2}}$ then $(\varepsilon_N\delta_N^{d-2})^{-1/2}\delta _N^2\to \infty$ as $N\to \infty$ as soon as $d > 2$.
\end{remark}

\subsection{Stable radii of reproduction events}\label{subsec:resu:stable_radii}

In the previous subsection, we assumed that the radius of dispersion of the offspring produced at reproduction events was small.
We now wish to allow large scale extinction-recolonization events to take place to illustrate the fact that "catastrophic" extinction events can occur, followed by a quick replacement of the dead individuals by the offspring of a small subset (here only one individual) of the survivors.
To do so, we suppose that the intensity measure for the radius of reproduction events $\mu(\ld r)$ has a power law behaviour, following the work in \cite{etheridge_rescaling_2014}.
The corresponding limiting behaviour is described by reaction-diffusion equations with non-local diffusion, studied for example in \cite{chmaj_existence_2013,achleitner_traveling_2013}. 
Suppose that the measure $\mu(\ld r)$ for the radius of reproduction events is given by
\begin{equation} \label{defn_stable_radii}
\mu(\ld r) = \frac{\1{r\geq 1}}{r^{d+\alpha+1}}\ld r,
\end{equation}
for some $\alpha\in(0,2\wedge d)$.
Fix $u, s \in(0,1]$ and choose $w_{0}:\R^{d}\to[0,1]$ with uniformly bounded spatial derivatives of up to the second order. 
Again, take $\proc{\varepsilon_N}{N\geq 1}$ and $\proc{\delta_N}{N\geq 1}$ two sequences in $(0,1]$ decreasing to zero, and set
\begin{align*}
s_N = \delta_N^\alpha s, && u_N = \varepsilon_N u, && q_0^N(x) = w_0(\delta_N x).
\end{align*}
Let $F:\R\to\R$ be a smooth, bounded function with bounded first and second derivatives such that $F|_{[0,1]}$ satisfies \eqref{selection_fn_cond} for some $m\in\N$ and $p:\{a,A\}^m\to [0,1]$.
Then for $N\geq 1$, 
let $\proc{q^{N}}$ be the \slfv with selection following the dynamics of Definition~\ref{definition_slfv_selection_haploid} with impact parameter $u_N$, radius of reproduction events given by $\mu(\ld r)$ in \eqref{defn_stable_radii}, selection parameter $s_{N}$ and selection function $F$ started from the initial condition $q_{0}^{N}$.


The main difference with the setting of Subsection~\ref{subsec:resu:fixed_radius} is that the flow resulting from the reproduction events is the $\alpha$-stable version of the heat flow (see Section~\ref{subsec:mart_pb:stable}). 
Thus we apply a stable scaling of time by $1/\delta_N^\alpha$ and space by $1/\delta_N$ (after rescaling time by $1/\varepsilon_N$ as previously).
Since we have chosen $s_N = \delta_N^\alpha s$, this is the right scaling to consider in order to observe both selection and migration in the limit.
For all $x\in\R^d$ and $t\geq 0$, set
\begin{equation} \label{definition_rescaled_stable}
\mathbf{q}^N_t(x) = q^N_{t/(\varepsilon_N \delta_N^\alpha)}(x/\delta_N).
\end{equation}

We need some more notation; recall the notation for double averages in \eqref{defn_double_avg}.
The following will take up the role played by $\avg{F(\avg{w})}$ in the fixed radius case.
For $H:[0,1]\to\R$, $\delta>0$, and $f\in\Xi$, set
\begin{equation}\label{definition_F_delta}
H^{(\delta)}(f):x\mapsto \alpha\intg{1}{\infty}\avg{H(\avg{f})}(x,\delta r)\frac{\ld r}{r^{\alpha+1}}.
\end{equation}
Recalling the notation in \eqref{ball_intersect}, set, for $x,y\in\R^d$,
\begin{align*}
\Phi(|x-y|)=\int_{\frac{|x-y|}{2}}^\infty \frac{V_r(x,y)}{V_r}\frac{\ld r}{r^{d+\alpha+1}}, && \Phi^{(\delta)}(|x-y|)=\int_{\frac{|x-y|}{2}\vee \delta}^\infty \frac{V_r(x,y)}{V_r}\frac{\ld r}{r^{d+\alpha+1}}.
\end{align*}
For $\phi:\R^d\to\R$,
\begin{align}\label{definition_fractional_laplacian}
\D\phi(x) = \intrd\Phi(\abs{x-y})(\phi(y)-\phi(x))\ld y, && \L*\phi(x) = \intrd\Phi^{(\delta)}(\abs{x-y})(\phi(y)-\phi(x))\ld y.
\end{align}
\begin{remark}
Up to a multiplicative constant, depending on $d$ and $\alpha$, $\D$ is the fractional Laplacian (this can be seen via the Fourier transform, see \cite{samko_fractional_1993}). 
\end{remark}

We can now formulate our result for the stable radii regime.
The main difference from Theorem~\ref{thm:clt_slfv_fixed_radius} is that the Laplacian has to be replaced by the operator $\D$ and that the noise driving the fluctuations is replaced by a coloured noise which is white in time and has spatial correlations which decay like $K_{\alpha}(z_1,z_2)$ as $|z_1 - z_2|\rightarrow \infty$, where
\begin{equation} \label{def_K_alpha}
K_{\alpha}(z_1,z_2) = \intg{\frac{\abs{z_1-z_2}}{2}}{\infty}V_r(z_1,z_2)\frac{\ld r}{r^{d+\alpha+1}} = \frac{C_{d,\alpha}}{\abs{z_1-z_2}^{\alpha}}.
\end{equation}
We also set the following notation: for $f\in\Xi$,
\begin{equation}\label{definition_alpha_average}
[f]_\alpha(z_1,z_2) = \frac{\intg{\frac{\abs{z_1-z_2}}{2}}{\infty}\frac{\ld r}{r^{d+\alpha+1}}\int_{B(z_1,r)\cap B(z_2,r)}\avg{f}(x,r)\ld x}{\intg{\frac{\abs{z_1-z_2}}{2}}{\infty}V_r(z_1,z_2)\frac{\ld r}{r^{d+\alpha+1}}}.
\end{equation}
Note that if $f$ denotes the frequency of type $a$ in $\mathbf q^N_t$ immediately before a (neutral) reproduction event which hits both $z_1$ and $z_2$ with $\abs{z_1-z_2}\geq 2\delta_N$, then $[f]_\alpha(z_1,z_2)$ is the probability that the offspring produced in this event are of type $a$.

Now define $\ft$ as the solution to
\begin{equation} \label{definition_centering_term_stable}
\left\lbrace
\begin{aligned}
& \deriv*{\ft}{t} = u\left[ \L*{\delta_N}\ft - \frac{V_1 s}{\alpha}F^{(\delta_N)}(\ft)\right], \\
& \ft{0} = w_0.
\end{aligned}
\right.
\end{equation}
\begin{thm}[Central Limit Theorem for the SLFVS with stable radii of reproduction events]\label{thm:clt_slfv_stable_radius}
Suppose that $\varepsilon_N = \littleO{\delta_N^{2\alpha}}$; then $\proc{\mathbf{q}^N}$ converges in $L^1$ and in probability (for the metric $d$ of Definition~\ref{def_metric_xi}) to the deterministic solution of the following PDE,
\begin{equation*}
\left\lbrace
\begin{aligned}
& \deriv*{f_t}{t} = u\left[ \D f_t - \frac{s V_1}{\alpha}F(f_t)\right], \\
& f_0 = w_0.
\end{aligned}
\right.
\end{equation*}
In addition, $$ \zt = \varepsilon_N^{-1/2}(\mathbf{q}^N_t - \ft) $$
defines a sequence of distribution-valued processes, converging in distribution in $\sko{\calS'(\R^d)}$ to the solution of the following \spde,
\begin{equation*}
\left\lbrace
\begin{aligned}
& \ld z_{t} = u\left[\D z_{t} - \frac{s V_1}{\alpha}F'(f_t) z_{t}\right]\ld t + u\,\ld W^\alpha_{t} \\
& z_{0} = 0,
\end{aligned}
\right.
\end{equation*}
where $W^\alpha$ is a coloured noise with covariation measure given by
\begin{equation}\label{stable_noise_correlations}
Q^\alpha(\ld z_1\ld z_2\ld s) = K_\alpha(z_1,z_2) \left( [f_s]_\alpha(z_1,z_2) (1-f_s(z_1)-f_s(z_2)) + f_s(z_1)f_s(z_2) \right)\ld z_1 \ld z_2 \ld s.
\end{equation}
\end{thm}

\begin{remark}
The fact that the correlations in the noise decay as $\abs{z_1-z_2}^{-\alpha}$ can be expected from the results in \cite{bierme_about_2006}. The authors prove that, if $N$ is a \ppp on $\R^d \times \R_+$ whose intensity measure is of the form $\ld x f(r) \ld r$ with $f(r) \sim \frac{C}{ r^{1 + \alpha + d} }$, one can define a generalized random field $X$ on the space of signed measures on $\R^d$ with finite total variation by
\begin{equation*}
\dual{X}{\mu} = \int_{\R^d \times \R_+} \mu( B(x,r) ) N(\ld x, \ld r).
\end{equation*}
Under a suitable scaling of the radius and of the intensity measure, it is shown that the fluctuations of $X$ converge (in the sense of finite dimensional distributions) to a centred Gaussian random linear functional $W^\alpha$ with
\begin{equation*}
\E{ W^\alpha(\mu) W^\alpha(\nu) } = \int_{\R^d} \int_{\R^d} \abs{ z_1 - z_2 }^{-\alpha} \mu(\ld z_1) \nu(\ld z_2).
\end{equation*}
(The notation has been changed so as to fit that of our setting; in \cite{bierme_about_2006}, $\beta = \alpha + d$.)
\end{remark}

We shall show in Section \ref{sec:martingale_pbs} that Theorem~\ref{thm:clt_slfv_stable_radius} is a consequence of Theorem~\ref{thm:clt_stable}. The latter is a result on sequences of solutions to a martingale problem and is proved in Section~\ref{sec:proof_stable}.

\subsection{Drift load for a spatially structured population}\label{subsec:driftload}

We shall illustrate the application of our results by studying the drift load in the SLFVS with overdominance as in Definition~\ref{definition_slfv_diploid}, in the case of bounded radii.

As in Section \ref{subsec:resu:fixed_radius}, fix $u$, $s_1$, $s_2$, $\nu_1$, $\nu_2$ in $(0,1]$ and $R>0$ such that $s_1 + s_2 + \nu_1 + \nu_2 < 1$, take two sequences $\proc{\varepsilon_{N}}{N\geq 1}$, $\proc{\delta_{N}}{N\geq 1}$ of positive real numbers in $(0,1]$ decreasing to zero, and set 
\begin{align} \label{driftload_params}
&& u_N=\varepsilon_{N}u, && r_N = \delta_N R &&
s_{i,N} = \delta_N^2s_i, && \nu_{i,N} = \delta_N^2\nu_i
\end{align}
for $i=1,2$. 
Then for $N\geq 1$, let $\proc{q^{N}}$ be the SLFVS following the dynamics of Definition~\ref{definition_slfv_diploid} with impact parameter $u_N$, radius of reproduction events $R$, selection parameters $s_{i,N}$ and mutation parameters $\nu_{i,N}$, started from some initial condition $q_{0}^{N}$.

One thing to note is that for our results to hold, we need to make sure that the allele frequencies do not get "stuck" - even locally - at the boundaries (\textit{i.e.} upon reaching 0 or 1), which could significantly slow down the convergence to the equilibrium frequency. 
For this reason we choose to assume that during some mutation reproduction events the type of the offspring can differ from that of its parent.
This will not affect the results in any other way provided that the mutation parameters are negligible compared to the selection parameters. 
In the remainder of this section, we assume that $\nu_1,\nu_2\ll s_1,s_2$.

Now let
\begin{equation} \label{definition_F_driftload}
F(w)=w(1-w)(w-\frac{s_2}{s_1 +s_2})+\frac{\nu_1}{s_1+s_2}w -\frac{\nu_2}{s_1+s_2}(1-w).
\end{equation}
We shall see in Section \ref{sec:martingale_pbs} that this function plays the same role as $F$ in the haploid case.

Note that $F$ satisfies the following conditions:
\begin{equation}\label{F_stable_equilibrium_1}
\exists \lambda\in[0,1] : F(\lambda)=0;
\end{equation}
furthermore there is only one such $\lambda$ and it satisfies
\begin{equation}\label{F_stable_equilibrium_2}
0<\lambda<1 \: \text{ and } \: F'(\lambda)>0.
\end{equation} 
For the function $F$ given in \eqref{definition_F_driftload}, $\lambda$ is given by
\begin{equation*}
\lambda = \frac{s_2}{s_1+s_2} + \bigO{\frac{\nu_1+\nu_2}{s_1+s_2}}.
\end{equation*}

Let us define $K^{N}(t,x)$, the local mean fitness at a point $x\in\R^{d}$, as the expected fitness of an individual formed by fusing two gametes chosen uniformly at random from $B(x,R)$ at time $t\geq 0$.
In other words, its two copies of the gene are sampled independently by selecting two parental locations $y_1$ and $y_2$ uniformly at random in $B(x,R)$ and then types according to $q_t(y_1)$ and $q_t(y_2)$. 
Then, for $\nu_{i}\ll s_{i}$, (see \cite{robertson_reduction_1970})
\begin{align*}
K^{N}(t,x) &= \E{ (1-s_1^N) \avg{q^N_t}(x,R)^2 + 2\avg{q^N_t}(x,R)(1-\avg{q^N_t}(x,R)) + (1-s^N_2)(1-\avg{q^N_t}(x,R))^2 } \\
&= 1 - \E{ s_1^N \avg{q^N_t}(x,R)^2 + s_2^N (1-\avg{q^N_t}(x,R))^2 } \\
&= 1-\frac{s_{1}^Ns_{2}^N}{s_{1}^N+s_{2}^N} - (s_{1}^N+s_{2}^N)\E{\left(\avg{q^N_t}(x,R)-\lambda\right)^{2}} + \bigO{(s_1^N+s_2^N)(\nu_1+\nu_2) }.
\end{align*}
The first term $\frac{s_{1}^Ns_{2}^N}{s_{1}^N+s_{2}^N}$ is the segregation load mentioned in the introduction, and it is of order $\delta_N^2$. 
The remaining term is then the local drift load, which we aim to estimate at large times for large $N$. 
Let us set 
\begin{equation} \label{defn_driftload}
\Delta^N(t,x) = (s_1^N+s_2^N)\E{\left(\avg{q^N_t}(x,R)-\lambda\right)^{2}}.
\end{equation}
The following theorem is proved in Section~\ref{sec:proof_drift_load} using some of the intermediate results used to prove Theorem \ref{thm:clt_slfv_fixed_radius}.

\begin{thm}\label{thm:driftload_slfv}
Suppose that $q^N_0(x)=\lambda$ for all $x$ and assume that $\varepsilon_N=\littleO{\delta_N^4}$. There exists a constant $C>0$, depending only on the dimension $d$, such that, for all $x\in\R^d$, as $N,t\to\infty$, if $t$ grows fast enough that  $\varepsilon _N t \to\infty$ if $d\geq 3$ and $\varepsilon _N\delta_N^2 t\to\infty$ if $d\leq 2$,
\begin{equation*}
\Delta^N(t,x) \underset{N,t\to\infty}{\sim} C\varepsilon_N \delta_N^2 c_N,
\end{equation*}
where 
\begin{equation}\label{definition_cn}
c_N = \left\lbrace
\begin{aligned}
& 1 && \mathrm{if\:} d\geq 3, \\
& \abs{\log \delta_N^2} && \mathrm{if\:} d=2, \\
& \delta_N^{-1} && \mathrm{if\:} d=1.
\end{aligned}
\right.
\end{equation}
\end{thm}

Assumption \eqref{F_stable_equilibrium_1}-\eqref{F_stable_equilibrium_2} is crucial in \cite{norman_central_1974}, which serves as a basis for this result. 
In fact this condition ensures that $\lambda$ is the only equilibrium point for the allele frequency, and that it is stable. 
\begin{remark}
We chose to start the process from the equilibrium frequency $\lambda$ - \textit{i.e.} very near stationarity - but we need not do so. 
The same result can be obtained starting from an arbitrary initial condition, provided we let $t$ grow sufficiently fast that the process reaches stationarity quickly enough.
The corresponding centering term $f^N$ is then defined as in~\eqref{definition:centering_term_fixed_radius}, and~\eqref{F_stable_equilibrium_1}-\eqref{F_stable_equilibrium_2} ensures that it converges to $\lambda$ exponentially quickly. 
Starting from $\lambda$ simplifies the proof as in this case, for all $t\geq 0$, $f^N_t=\lambda$.
\end{remark}

In the non-spatial setting of the $\Lambda$-Fleming Viot process, a simplified version of the proof of Theorem \ref{thm:driftload_slfv} shows that the drift load is asymptotically proportional to $u_N$. 
We can see $u_N$ as being inversely proportional to the neighbourhood size, in other words the probability that two individuals had a common parent in the previous generation (see \cite{barton_modeling_2013} for details). 
This agrees with Robertson's estimate \cite{robertson_reduction_1970} of $(4N)^{-1}$, where $N$ is the total population size in a panmictic population. 
Note that this estimate is independent of the strength of selection.
This can be seen as the result of a trade off between selection and genetic drift: if selection is weak, the allele frequency can be far from the equilibrium whereas if selection is stronger, the allele frequency stays nearer to the equilibrium and in both cases the mean fitness of the population is the same.

For spatially structured populations, however, Theorem \ref{thm:driftload_slfv} shows that the local drift load is significantly smaller than in the non-spatial setting and does depend on the strength of natural selection.
For example, if a population lives in a geographical space of dimension 2, the corresponding drift load will be of order $u_N s_N \abs{\log s_N}$.
Moreover, we see a strong effect of dimension on this estimate.
Populations living in a space with a higher dimension have a reduced drift load compared to populations evolving in smaller dimensions.
This result illustrates the fact that, in a higher dimension, migration is more efficient at preventing the allele frequencies from being locally far from the equilibrium frequency.
It turns out from the proof that this is linked to the recurrence properties of Brownian motion.

\begin{remark}[Drift load in the stable case]
If one considers instead the SLFVS with stable radii of reproduction events, under similar conditions to those in Theorem~\ref{thm:clt_slfv_stable_radius}, one finds that for all $d\geq 1$ and $\alpha \in (0, 2 \wedge d)$, $\Delta^N(t,x)$ is asymptotically equivalent to a constant times $u_N s_N\abs{\log s_N}$.
\end{remark}

\section{Martingale problems for the SLFVS}\label{sec:martingale_pbs}

This section provides the basic ingredients for the proofs of Theorems~\ref{thm:clt_slfv_fixed_radius} and~\ref{thm:clt_slfv_stable_radius}. 
In Subsection~\ref{subsec:martingale_pb}, we prove that the SLFVS satisfies a martingale problem. In Subsections~\ref{subsec:mart_pb:fixed} and~\ref{subsec:mart_pb:stable}, we study the martingale problem for the rescaled version of this process, in the fixed radius case and in the stable radii case, and state general convergence results for processes satisfying these martingale problems. 
Theorems~\ref{thm:clt_slfv_fixed_radius} and~\ref{thm:clt_slfv_stable_radius} are a direct consequence of these results.

\subsection{The martingale problem for the SLFVS}\label{subsec:martingale_pb}

Let $\proc{q^N}$ be defined as in Sections~\ref{subsec:resu:fixed_radius} and~\ref{subsec:resu:stable_radii} as the SLFVS as in Definition~\ref{definition_slfv_selection_haploid} with impact parameter $u_N$, distribution of reproduction event radii given by $\mu(dr)$, selection parameter $s_N$ and selection function $F$. 
Let $\proc{\F}$ denote the natural filtration of this process.

\begin{proposition}\label{prop:martingale_pb_slfv}
Suppose that $\int_{0}^{\infty} V_r^2 \mu(\ld r) < \infty$. 
For any $\phi:\R^{d}\to\R$ in $L^{1,\infty}(\R^d)$,
\begin{equation}\label{martingale_w}
\dual{q^N_{t}} - \dual{q^N_{0}} - \intg{0}{t} \intg{0}{\infty} u_N V_{r} \braced{ \dual{q^N_{s}}{\dbavg{\phi}{r}-\phi} - s_N \dual{\avg{F(\avg{q^N_{s}})}{r}} } \mu(\ld r) \ld s
\end{equation}
defines a (mean zero) square integrable $\F_{t}$-martingale with (predictable) variation process
\begin{equation}\label{variation_w}
\intg{0}{t} \intg{0}{\infty} u_N^{2} V_{r}^{2} \intrd{2} \phi(z_{1}) \phi(z_{2}) \sigma^{(r)}_{z_{1},z_{2}}(q^N_{s}) \ld z_{1} \ld z_{2} \mu(\ld r) \ld s +\bigO{ t u_N^2 s_N \norm[2]{\phi}^2 },
\end{equation}
where
\begin{multline} \label{def_sigma_r}
\sigma^{(r)}_{z_{1},z_{2}}(q) = \frac{1}{V_{r}^{2}}\int_{B(z_{1},r)\cap B(z_{2},r)}[\avg{q}(x,r)(1-q(z_{1}))(1-q(z_{2}))  +(1-\avg{q}(x,r))q(z_{1})q(z_{2})]\ld x.
\end{multline}
\end{proposition}

Proposition~\ref{prop:martingale_pb_slfv} can be seen as a way to write $q_t$ as the sum of the effects of the different evolutionary forces at play in this model. 
The term $\dbavg{\phi}-\phi$ represents migration, while the term involving the function $F$ in~\eqref{martingale_w} accounts for the bias introduced during selective events. 
As for the martingale term, it corresponds to the stochasticity at each reproduction event, which is called \textit{genetic drift}.

\begin{proof}[Proof of Proposition~\ref{prop:martingale_pb_slfv}]
We drop the superscript $N$ from $q^N$ in this proof. Let $\mathbb{P}_{t,x,r}$ (resp. $\mathbb{P}^{S}_{t,x,r}$) denote the distribution of the parental type $k$ at a reproduction event $(t,x,r)\in\Pi$ (resp. in $\Pi^S$). Then, from the definition of $\proc{q}$,
\begin{multline}\label{first_moment}
\lim_{\delta t\downarrow 0}\frac{1}{\delta t}\E{\dual{q_{t+\delta t}}-\dual{q_{t}}}{q_{t}=q} =\\ \intrd\ld x\intg{0}{\infty} \mu(\ld r) \intrd \phi(z) u_N \inball{x}{z} \big\lbrace (1-s_N) \E[t,x,r]{\1{k=a}-q_{t}(z)}{q_{t}=q} \\+  s_N \,\E[t,x,r]{\1{k=a}-q_{t}(z)}{q_{t}=q}{S} \big\rbrace\ld z.
\end{multline}
Recall from Definition~\ref{definition_slfv_selection_haploid} that $\P[t,x,r]{k=a}{q_{t}=q} = \avg{q}(x,r)$ and 
$$\P[t,x,r]{k=a}{q_{t}=q}{S} = \avg{q}(x,r) - F( \avg{q}(x,r)).$$
Integrating with respect to the variable $x$ over $B(z,r)$ then yields
\begin{multline*}
\lim_{\delta t\downarrow 0}\frac{1}{\delta t}\E{\dual{q_{t+\delta t}}-\dual{q_{t}}}{q_{t}=q} \\ = \intg{0}{\infty}\mu(\ld r)u_N V_{r}\intrd\phi(z)\braced{(\dbavg{q}(z,r)-q(z)) - s_N \avg{F(\avg{q})}(z,r)}\ld z.
\end{multline*}
Thus \eqref{martingale_w} indeed defines a martingale - see for example \citep[Proposition 4.1.7]{ethier_markov_1986} (we can change the order of integration to do the averaging on $\phi$ instead of $q$ in the first term). To compute its variation process, write
\begin{multline} \label{var_proc_s_error}
\lim_{\delta t\downarrow 0}\frac{1}{\delta t}\E{\left(\dual{q_{t+\delta t}}-\dual{q_{t}}\right)^{2}}{q_{t}=q} \\ = \intrd \intg{0}{\infty} \intrd{2}\phi(z_{1})\phi(z_{2})u_{N}^{2}\inball{z_{1}}{x}{z_{2}}{x} \big\lbrace (1-s_N) \E[t,x,r]{(\1{k=a}-q_{t}(z_{1}))(\1{k=a}-q_{t}(z_{2}))}{q_{t}=q} \\+  s_N \,\E[t,x,r]{(\1{k=a}-q_{t}(z_{1}))(\1{k=a}-q_{t}(z_{2}))}{q_{t}=q}{S} \big\rbrace\ld z_{1}\ld z_{2} \mu(\ld r) \ld x.
\end{multline}
But
\begin{multline*}
\E[t,x,r]{(\1{k=a}-q_{t}(z_{1}))(\1{k=a}-q_{t}(z_{2}))}{q_{t}=q} \\ = \avg{q}(x,r)(1-q(z_{1}))(1-q(z_{2}))  + (1-\avg{q}(x,r))q(z_{1})q(z_{2}),
\end{multline*}
and the other term within the curly brackets is $\bigO{s_N}$.
Thus, integrating with respect to $x$ and using \eqref{def_sigma_r}, we recover
\begin{multline} \label{martingale_qvar}
\lim_{\delta t\downarrow 0}\frac{1}{\delta t}\E{\left(\dual{q_{t+\delta t}}-\dual{q_{t}}\right)^{2}}{q_{t}=q} = \intg{0}{\infty} \mu(\ld r) u_N^{2} V_{r}^{2} \intrd{2} \phi(z_{1}) \phi(z_{2}) \sigma^{(r)}_{z_{1},z_{2}}(q) \ld z_{1} \ld z_{2}\\
+\bigO{s_N}\intg{0}{\infty} \mu(\ld r) u_N^{2} V_{r}^{2} \intrd \left( \frac{1}{V_r}\int_{B(x,r)}\phi(z)dz \right)^2 dx .
\end{multline}
By Jensen's inequality, $\intrd \left( \frac{1}{V_r}\int_{B(x,r)}\phi(z)dz \right)^2 dx \leq \norm[2]{\phi}^2$ and the result follows from the assumption that $\int_{0}^{\infty} V_r^2 \mu(\ld r) < \infty$.
\end{proof}

Now let $\proc{q^N}$ denote the SLFVS with overdominance as defined as in Definition \ref{definition_slfv_diploid} with impact parameter $u_N$, radius of reproduction events $R$, selection parameters $s_{i,N}$ and mutation parameters $\nu_{i,N}$ defined in \eqref{driftload_params}. 
Recall the definition of $F$ in \eqref{definition_F_driftload} and let $\proc{\F}$ denote the natural filtration of this process.

\begin{proposition}\label{prop:martingale_pb_slfv_dr}
Let $s=s_1+s_2$ (and $s_N = s_{1,N} + s_{2,N}$). 
For any $\phi:\R^{d}\to\R$ in $L^{1,\infty}(\R^d)$,
\begin{equation}\label{martingale_q_dr}
\dual{q^N_{t}} - \dual{q^N_{0}} - \intg{0}{t} u_N V_{r} \braced{ \dual{q^N_{s}}{\dbavg{\phi}{R}-\phi} - s_N \dual{\avg{F(\avg{q^N_{s}})}{R}} } \ld s
\end{equation}
defines a (mean zero) square integrable $\F_{t}$-martingale with (predictable) variation process
\begin{equation}\label{variation_q_dr}
\intg{0}{t} u_N^{2} V_{R}^{2} \intrd{2} \phi(z_{1}) \phi(z_{2}) \rho^{(R)}_{z_{1},z_{2}}(q^N_{s}) \ld z_{1} \ld z_{2} \ld s  + \bigO{tu_N^2 \delta_N^2 \norm[2]{\phi}^2},
\end{equation}
where
\begin{multline} \label{def_tau_r}
\rho^{(r)}_{z_{1},z_{2}}(q) = \frac{1}{V_{r}^{2}}\int_{B(z_{1},r)\cap B(z_{2},r)}[\avg{q}(x,r)^2(1-q(z_{1}))(1-q(z_{2}))+2\avg{q}(x,r)(1-\avg{q}(x,r))(\tfrac{1}{2}-q(z_{1}))(\tfrac{1}{2}-q(z_{2})) \\ +(1-\avg{q}(x,r))^2q(z_{1})q(z_{2})]\ld x.
\end{multline}
\end{proposition}

\begin{proof}
Suppose a reproduction event hits the ball $B(x,r)$ at time $t$, and let $w = \avg{q^N_{t^-}}(x,r)$. Then,
\begin{align*}
\P{ \left\lbrace k_1, k_2 \right\rbrace = \left\lbrace A_1, A_1 \right\rbrace} &= (1 - s_{1,N} - s_{2,N} - \nu_{1,N} - \nu_{2,N}) w^2 + s_{1,N} w^4 + s_{2,N} w^2 (1 + (1-w)^2) \\ &\hspace{10pt} + \nu_{2,N}, \\
\P{ \left\lbrace k_1, k_2 \right\rbrace = \left\lbrace A_1, A_2 \right\rbrace} &= (1 - s_{1,N} - s_{2,N} - \nu_{1,N} - \nu_{2,N}) 2w(1-w) + s_{1,N} 2w(1-w)(1+w^2) \\ &\hspace{10pt} + s_{2,N} 2w(1-w) (1 + (1-w)^2), \\
\P{ \left\lbrace k_1, k_2 \right\rbrace = \left\lbrace A_2, A_2 \right\rbrace} &= (1 - s_{1,N} - s_{2,N} - \nu_{1,N} - \nu_{2,N}) (1-w)^2 + s_{1,N} (1-w)^2 (1+w^2) \\ &\hspace{10pt} + s_{2,N} (1-w)^4 + \nu_{1,N}.
\end{align*}
Note that this corresponds to the first order approximation of \eqref{wright-fisher-selection}, modified to take mutations into account.
It is straightforward to check that
\begin{equation*}
\E[t,x,r]{\tfrac{1}{2} \left( \1{k_1=A_1} + \1{k_2=A_1} \right) - q_t(z)}{q_t = q} = \avg{q}(x,r) - q(z) - s_N F(\avg{q}(x,r))
\end{equation*}
where $F$ is given by \eqref{definition_F_driftload}.
It follows as in the proof of Proposition \ref{prop:martingale_pb_slfv} that \eqref{martingale_q_dr} is a martingale.
The result for the variation process also follows as in the proof of Proposition \ref{prop:martingale_pb_slfv}.
(Note that $\sigma^{(r)}$ is replaced by $\rho^{(r)}$ in order to account for the fact that \eqref{update_w} is replaced by \eqref{eq:overdom_update}.)
\end{proof}

\begin{remark}
If $q$ were continuous then as $r\rightarrow 0$, $\sigma^{(r)}_{z_{1},z_{2}}(q)\rightarrow \delta_{z_1=z_2} q(z_1)(1-q(z_1))$ and $\rho^{(r)}_{z_{1},z_{2}}(q)\rightarrow \tfrac{1}{2} \delta_{z_1=z_2} q(z_1)(1-q(z_1))$. The factor of $1/2$ represents the doubling of effective population size for a diploid population compared to a haploid one.
\end{remark}

\subsection{The rescaled martingale problem - Fixed radius case}\label{subsec:mart_pb:fixed}

As at the start of Subsection \ref{subsec:resu:fixed_radius}, let $\proc{\varepsilon_{N}}{N\geq 1}$, $\proc{\delta_{N}}{N\geq 1}$ be sequences in $(0,1]$ decreasing towards zero, and let $F:\R\to \R$.
\begin{definition} [Martingale Problem (M1)] \label{formal_def_M1}
Given $\proc{\varepsilon_{N}}{N\geq 1}$, $\proc{\delta_{N}}{N\geq 1}$ and $F$, let $\eta_N = \varepsilon_N \delta_N^2$, $\tau_N = \varepsilon_N^2 \delta_N^d$ and $r_N=\delta_N R$. Then for $N\geq 1$, we say that a $\Xi$-valued process $(w_t^N)_{t\geq 0}$ satisfies the martingale problem (M1) if 
for all $\phi$ in $L^{1,\infty}(\R^d)$,
\begin{equation} \label{martingale_w_brown}
\dual{w^{N}_{t}} - \dual{w_{0}} - \eta_{N} u V_R \intg{0}{t} \braced{ \frac{2R^2}{d+2} \dual{w^{N}_{s}}{\L{r_{N}}\phi} - s\dual{ \avg{F(\avg{w^{N}_{s}})}{r_{N}} }} \ld s
\end{equation}
defines a (mean zero) square-integrable martingale with (predictable) variation process
\begin{equation} \label{variation_w_brown}
\tau_{N} u^2 V_R^2 \intg{0}{t} \intrd{2} \phi(z_{1}) \phi(z_{2}) \sigma^{(r_{N})}_{z_{1},z_{2}}(w^{N}_{s}) \ld z_{1} \ld z_{2} \ld s + \bigO{t \tau_N \delta_N^2 \norm[2]{\phi}^2}.
\end{equation}
\end{definition}
\begin{remark}
Of course, one cannot expect uniqueness to hold for this martingale problem, due to the unspecified error term in \eqref{variation_w_brown}. In the limit when $N \to \infty$, however, the error terms will vanish.
\end{remark}

Let $\proc{q^N}$ be defined as at the start of Section \ref{subsec:resu:fixed_radius}. 
Set $w^N_t(x) = q^N_t(x/\delta_N)$.
\begin{proposition}\label{prop:mart_pb_m1}
For each $N$, the process $\proc{w^N}$ satisfies the martingale problem (M1).
\end{proposition}

\begin{proof}
From Proposition~\ref{prop:martingale_pb_slfv}, we know that, for $\phi\in L^{1,\infty}(\R^d)$,
\begin{equation*}
\dual{q^{N}_{t}} = \dual{q_{0}^{N}} + u_N V_{R}\intg{0}{t}\braced{ \dual{q_{s}^{N}}{\dbavg{\phi}{R}-\phi} - s_N \dual{\avg{F(\avg{q^{N}_{s}})}{R}}}\ld s + \calM^{N}_{t}(\phi),
\end{equation*}
where $\calM^{N}_{t}(\phi)$ is a martingale. 
By a change of variables,
\begin{equation*}
\dual{w_{t}^{N}} = \delta_{N}^{d}\dual{q^{N}_{t}}{\phi^{(\delta_{N})}},
\end{equation*}
with $\phi^{(\delta)}(x) = \phi(\delta x)$. Also,
\begin{align*}
\delta_{N}^{d}\dual{q^{N}_{s}}{\dbavg{\phi^{(\delta_{N})}}{R}} = \dual{w^{N}_{s}}{\dbavg{\phi}{\delta_{N}R}} &&\mathrm{and}&& \delta_{N}^{d}\dual{\avg{F(\avg{q^{N}_{s}})}{R}}{\phi^{(\delta_{N})}} = \dual{\avg{F(\avg{w^{N}_{s}})}{\delta_{N}R}}.
\end{align*}
Thus, recalling the definition of the operator $\L$ in \eqref{definition_calL} and the initial condition $q_0^N(x)=w_0(\delta_N x)$, we have
\begin{multline*}
\dual{w_{t}^{N}} = \dual{w_{0}} + \varepsilon_{N} \delta_{N}^{2} uV_{R} \intg{0}{t} \braced{ \frac{2R^{2}}{d+2} \dual{w^{N}_{s}}{\L{\delta_{N}R}\phi} - s \dual{\avg{F(\avg{w^{N}_{s}})}{\delta_{N}R}} } \ld s \\ + \delta_{N}^{d} \calM^{N}_{t}(\phi^{(\delta_{N})}).
\end{multline*}
Moreover, by a change of variables in the variation process given in~\eqref{variation_w},
\begin{align}\label{rescaling_covariation}
\delta_{N}^{2d}\qvar{\calM^{N}(\phi^{(\delta_{N})})} &= \varepsilon_{N}^2 u^2 V_R^2 \intg{0}{t}\intrd{2}\phi(z_1)\phi(z_2)\sigma^{(R)}_{z_1/\delta_{N},z_2/\delta_{N}}(q^N_s)\ld z_1\ld z_2\ld s +\bigO{t\varepsilon_N^2 \delta_N^2 \delta_N^d \norm[2]{\phi}^2}, 
\end{align}
and
\begin{equation*}
\sigma^{(R)}_{z_1/\delta_{N},z_2/\delta_{N}}(q^N_s) = \delta_{N}^d \sigma^{(\delta_{N} R)}_{z_1,z_2}(w^N_s).
\end{equation*}
Hence $w^N$ satisfies the martingale problem (M1).
\end{proof}

Proposition~\ref{prop:mart_pb_m1} is the main ingredient in the proof of Theorem~\ref{thm:clt_slfv_fixed_radius}. 
In fact we shall now see that under suitable conditions on the parameters $\proc{\varepsilon_{N}}{N\geq 1}$ and $\proc{\delta_{N}}{N\geq 1}$, the function $F$ and the initial condition $w_0$, any sequence of processes $\proc{w^N}$ satisfying the martingale problem (M1) in Definition \ref{formal_def_M1} will also satisfy a result analogous to Theorem \ref{thm:clt_slfv_fixed_radius}. 
If $\tau_N$ is of a smaller order than $\eta_N$, $w^N$ can be expected to be asymptotically deterministic (on a suitable time-scale), and we can study its fluctuations around a deterministic centering term.
Define $f^{N}:\R_{+}\times\R^{d}\to\R$ as in \eqref{definition:centering_term_fixed_radius}.
Quite naturally, this corresponds to equating \eqref{martingale_w_brown} to zero and making its time-scale fit that of the limiting process. 

Since the operator $\L$ approximates the Laplacian as $r\to 0$ (see Proposition~\ref{prop:average} in the appendix), $\ft$ converges to $f:\R_{+}\times\R^{d}\to\R$ as $N\to\infty$, where $f_t$ is the solution of the following equation,
\begin{equation} \label{fisher_kpp}
\left\lbrace
\begin{aligned}
& \deriv*{f_{t}}{t} = uV_R \left(\frac{R^2}{d+2}\Delta f_{t} -  s F(f_{t})\right), \\
& f_{0} = w_{0}.
\end{aligned}
\right.
\end{equation}
The following result is proved in Section~\ref{sec:proof_brownian}.

\begin{thm}\label{thm:clt_brownian}
Suppose that $\proc{w^{N}}$ is a $\Xi$-valued process which satisfies the martingale problem (M1) in Definition \ref{formal_def_M1} for some smooth, bounded $F:\R\to\R$ with bounded first and second derivatives and $\proc{\delta_{N}}{N}$, $\proc{\varepsilon_{N}}{N}$ converging to zero as $N\to\infty$. 
Moreover, suppose 
\begin{equation}\label{parameters_regime_brownian}
\tau_{N}/\eta_{N}=\littleO{\delta_{N}^{2d}}.
\end{equation}
Suppose also that $w_{0}$ has uniformly bounded derivatives of up to the fourth order and that there exists $\alpha_N$ such that the jumps of $\proc{w^{N}}$ are (almost surely) dominated by
\begin{equation}\label{bound_jumps}
\sup_{t\geq 0}\abs{\dual{w^{N}_{t}}-\dual{w^{N}_{t^-}}}\leq \alpha_{N}\norm[1]{\phi}
\end{equation}
for every $\phi\in L^{1,\infty}(\R^d)$, with $\alpha_{N}^{2}=\littleO{\tau_{N}/\eta_{N}}$. Then
\begin{equation} \label{convergence_w}
\proc{\wtN}{t\geq 0} \cvgas[L^{1},\,P]{N} \proc{f}
\end{equation}
in $\left( \sko{\Xi}, d \right)$ for every $T>0$ with $d$ given by Definition~\ref{def_metric_xi}. In addition, $$ \zt = (\eta_{N}/\tau_{N})^{1/2}(\wtN-\ft) $$ defines a sequence of distribution-valued processes which converges in distribution in $\sko{\calS'(\R^{d})}$ to the solution of the following \spde,
\begin{equation} \label{limiting_fluctuations}
\left\lbrace
\begin{aligned}
& \ld z_{t} = uV_R [\frac{R^2}{d+2}\Delta z_{t} -  s F'(f_{t})z_{t}]\ld t + u V_R\sqrt{f_{t}(1-f_{t})}\cdot\ld W_{t}, \\
& z_{0} = 0,
\end{aligned}
\right.
\end{equation}
$W$ being a space-time white noise.
\end{thm}
Theorem \ref{thm:clt_slfv_fixed_radius} is now a direct consequence.
\begin{proof}[Proof of Theorem \ref{thm:clt_slfv_fixed_radius}]
Recall that $(\mathbf q^N_t)_{t\geq 0}$ is defined in \eqref{definition_rescaled_process} as a rescaling of $\proc{q^N}$, and by Proposition \ref{prop:mart_pb_m1}, letting $w_t^N(x)=q_t^N(x/\delta_N)$, $\proc{w^N}$ satisfies the martingale problem (M1). 
Also $\tau_{N}/\eta_{N}=\littleO{\delta_{N}^{2d}}$ follows from $\varepsilon_N = o(\delta _N^{d+2})$, and the bound on the jumps \eqref{bound_jumps} holds with $\alpha_N=\varepsilon_N u$ by \eqref{update_w}. 
Hence Theorem~\ref{thm:clt_brownian} applies and the result follows by noting that $\wtN = \mathbf{q}^N_t$.
\end{proof}

The proof of Theorem~\ref{thm:clt_brownian} can be found in full detail in Section~\ref{sec:proof_brownian}, but, in order to shed some light on the limiting equations that we obtain and to identify the difficulties in proving this result, let us outline the first calculations involved in the proof.
As in \cite{kurtz_limit_1971}, we use bounds on the martingale \eqref{martingale_w_brown} to show the convergence of $\proc{\wtN}{t\geq 0}$. 
When properly rescaled, this martingale converges to a continuous Gaussian martingale, implying the convergence of the fluctuation process $\proc{Z^N}$. 

For ease of notation, we shall set the constants $uV_R$, $2R^2/(d+2)$ and $s$ to $1$ in the definition of (M1).
Let $\mathbb M^N_t(\phi)$ denote $\tau_N^{-1/2}$ times the martingale defined in \eqref{martingale_w_brown}. 
Formally, we can then write (M1) as
\begin{equation*}
\ld w^N_t = \eta_N\left[ \L{r_N}w^N_t -  \avg{F(\avg{w^N_t})}{r_N} \right]\ld t + \tau_N^{1/2}\ld \mathbb M^N_t.
\end{equation*}
Now set
\begin{equation*}
M^N_t(\phi) = \eta_N^{1/2}\mathbb M^N_{t/\eta_N}(\phi).
\end{equation*}
(This Brownian scaling is not surprising since in the SLFVS case $\mathbb M^N$ is essentially an integral against a compensated Poisson process, and we expect $M^N$ to converge to an integral against white noise.) 
Replacing $t$ by $t/\eta_N$ above, we have
\begin{equation*}
\ld \wtN = \left[ \L{r_N}\wtN -  \avg{F(\avg{\wtN})}{r_N} \right]\ld t + (\tau_N/\eta_N)^{1/2}\ld M^N_t.
\end{equation*}
Subtracting the equation
\begin{equation*}
\ld \ft = \left[ \L{\rn}\ft -  \avg{F(\avg{\ft})}{\rn} \right] \ld t,
\end{equation*}
and multiplying by $(\eta_N/\tau_N)^{1/2}$ on both sides, we obtain
\begin{equation} \label{zt_before_linearising}
\ld \zt = \left[ \L{\rn}\zt - (\eta_N/\tau_N)^{1/2} \avg{\left(F(\avg{\wtN})-F(\avg{\ft})\right)}{\rn} \right]\ld t + \ld M^N_t.
\end{equation}
Since the function $F:\R\to\R$ is smooth, for $k\in\lbrace 1,2\rbrace$ and $x,y\in[0,1]$, we can define the following:
\begin{equation*}
R_{k}(x,y) = \intg{0}{1}\frac{t^{k-1}}{(k-1)!}F^{(k)}(x+t(y-x))\ld t.
\end{equation*}
Then $R_k$ is continuous and bounded by $\frac{1}{k!}\norm[\infty]{F^{(k)}}$. In addition, by Taylor's formula,
\begin{align}
F(x) &= F(y) + (x-y)R_1(x,y), \label{taylor_F_1} \\
F(x) &= F(y) + (x-y)F'(y) + (x-y)^2 R_2(x,y). \label{taylor_F_2}
\end{align}
Substituting the second relation into \eqref{zt_before_linearising} yields
\begin{equation*}
\ld \zt = \left[ \L{\rn}\zt - \avg{\avg{\zt} F'(\avg{\ft})}{\rn} - (\tau_N/\eta_N)^{1/2} \avg{(\avg{\zt})^2 R_2(\avg{\wtN},\avg{\ft})}{\rn} \right]\ld t + \ld M^N_t.
\end{equation*}
In fact, this equality holds in mild form,
\begin{multline}\label{martingale_z}
\dual{\zt} = \intg{0}{t}\dual{\zt{s}}{ \L{\rn}\phi -  \avg{F'(\avg{\ft{s}})\avg{\phi}}{\rn}}\ld s \\ - (\tau_N/\eta_N)^{1/2} \intg{0}{t}\dual{(\avg{\zt{s}})^{2}}{R_{2}(\avg{\wtN{s}},\avg{\ft{s}})\avg{\phi}{\rn}}\ld s + M^N_t(\phi).
\end{multline}
(In other words, every step above can be done using the integral form, yielding \eqref{martingale_z}.) 
We can see $M^N$ as a martingale measure and, from a change of variables in \eqref{variation_w_brown}, it can be seen that its covariation measure is given by
\begin{equation}\label{covariation_MN}
Q^N(\ld z_1\ld z_2\ld s) = \sigma^{(\rn)}_{z_1,z_2}(\wtN{s})\ld z_1\ld z_2\ld s + \bigO{\delta_N^2} \delta_{z_1=z_2} (\ld z_1 \ld z_2) \ld s.
\end{equation}
Accordingly, we will sometimes write $M^N_t(\phi)$ as a stochastic integral (as defined in \citep[Chapter~2]{walsh_introduction_1986}),
\begin{equation*}
M^N_t(\phi) = \intg{0}{t}\intrd\phi(x)M^N(\ld x\ld s).
\end{equation*}

Note that we have linearised the drift term in \eqref{martingale_w_brown} around the deterministic centering term, and that the remaining term (where $R_2$ appears) is the error due to this linearisation. 
The main difficulty in proving the convergence of $Z^N$ is to control this error. 
At first sight, it would seem that the factor $(\tau_N/\eta_N)^{1/2}$ in front of it is enough to make it vanish in the limit. 
However, some care is needed in dealing with the quadratic term in the spatial integral. 
Since $Z^N$ is going to converge as a distribution-valued process, its square does not make sense in the limit. 
The control of this term is achieved through Lemma~\ref{lemma_bound_avg}, where we bound the square of the average of $\zt$ over a ball of radius $\rn$. 
It is for this purpose that we require that $\tau_N/\eta_N = \littleO{r_N^{2d}}$.

Once this is done, we will be in a good position to prove the convergence of $Z^N$. 
Indeed, as $\rn$ tends to zero, $\L{\rn}\phi - \avg{F'(\avg{\ft{s}})\avg{\phi}}(\rn)$ is well approximated by $\frac{1}{2}\Delta\phi - F'(f_s)\phi$ (see Proposition~\ref{prop:average}). 
We also prove that $M^N$ converges to $\sqrt{f_t(1-f_t)}\cdot W_t$ using the expression \eqref{covariation_MN} for its covariance. 

The proof of convergence of $Z^N$ follows the classical strategy of proving that the sequence is tight before uniquely characterising its possible limit points. 
We are outside the safe borders of real-valued processes, but the theory presented in \cite{walsh_introduction_1986} provides the main tools needed for the proof of our result. 
In particular, the argument relies heavily on Mitoma's Theorem (Theorem 6.13 in \cite{walsh_introduction_1986}), which states that a sequence of processes $\proc{X^n}$, $n\geq 1$ with sample paths in $\sko{\calS'(\R^d)}$ a.s. is tight if and only if, for each $\phi\in\calS(\R^d)$, the sequence of real-valued processes $\proc{\dual{X^n}}{n\geq 1}$ is tight in $\sko{\R}$ (see also Theorem~\ref{thm:mitoma}).

\subsection{The rescaled martingale problem - Stable radii case}\label{subsec:mart_pb:stable}

For $\phi \in L^{1,\infty}(\R^d)$, and $\alpha \in (0, d\wedge 2)$, define the following norm
\begin{equation} \label{def:norm_alpha}
\norm[(\alpha)]{\phi}^2 = \intrd{2} \phi(z_1) \phi(z_2) \abs{ z_1 - z_2 }^{-\alpha} \ld z_1 \ld z_2.
\end{equation}

Let $\proc{\varepsilon_{N}}{N\geq 1}$, $\proc{\delta_{N}}{N\geq 1}$ be sequences in $(0,1]$ decreasing towards zero, and let $F:\R\to \R$.
\begin{definition} [Martingale Problem (M2)] \label{formal_def_M2}
Given $\proc{\varepsilon_{N}}{N\geq 1}$, $\proc{\delta_{N}}{N\geq 1}$ and $F$, let $\eta_N = \varepsilon_N\delta_N^\alpha$ and $\tau_N = \varepsilon_N^2\delta_N^\alpha$. Then for $N\geq 1$, we say that a $\Xi$-valued process $(w_t^N)_{t\geq 0}$ satisfies the martingale problem (M2) if 
for all $\phi$ in $L^{1,\infty}(\R^d)$,
\begin{equation} \label{martingale_w_stable}
\dual{w^{N}_{t}} - \dual{w_{0}} - \eta_{N} u \intg{0}{t} \braced{ \dual{w^{N}_{s}}{\L*{\delta_{N}}\phi} - \frac{sV_1}{\alpha} \dual{ F^{(\delta_N)}(w^{N}_{s})}}\ld s
\end{equation}
defines a (mean zero) square-integrable martingale with (predictable) variation process
\begin{equation} \label{variation_w_stable}
\tau_{N} u^2 \intg{0}{t} \intrd{2} \phi(z_{1}) \phi(z_{2}) \sigma^{(\alpha,\delta_N)}_{z_{1},z_{2}}(w^{N}_{s}) \ld z_{1} \ld z_{2} \ld s + \bigO{ t \tau_N \delta_N^\alpha \norm[(\alpha)]{\phi}^2 },
\end{equation}
where
\begin{equation} \label{definition_sigma_stable}
\sigma^{(\alpha, \delta)}_{z_1,z_2}(w) = \intg{\frac{\abs{z_1-z_2}}{2}\vee \delta}{\infty}V_r^2\sigma^{(r)}_{z_1,z_2}(w)\frac{\ld r}{r^{d+\alpha+1}}.
\end{equation} 
\end{definition}
(Note that the remark about uniqueness made after Definition~\ref{formal_def_M1} also applies to the martingale problem (M2).)

Let $\proc{q^N}$ be defined as at the start of Section \ref{subsec:resu:stable_radii}.
Set
$w^N_t(x) = q^N_t(x/\delta_N)$.

\begin{proposition}\label{prop:mart_pb_m2}
For each $N$ the process $\proc{w^N}$ satisfies the martingale problem (M2).
\end{proposition}

\begin{proof}
This is proved in a similar way to Proposition~\ref{prop:mart_pb_m1}, using change of variables and the definitions of $\calD^{(\alpha,\delta)}$, $F^{(\delta)}$ and $\sigma^{(\alpha, r)}$ in \eqref{definition_fractional_laplacian}, \eqref{definition_F_delta} and \eqref{definition_sigma_stable} respectively.

Note that we cannot apply Proposition~\ref{prop:martingale_pb_slfv} directly, since in the stable case, $\int_{0}^{\infty} V_r^2 \mu(\ld r) = \infty$, but the term from the second line of \eqref{var_proc_s_error} in the proof of Proposition~\ref{prop:martingale_pb_slfv} can be bounded by
\begin{equation*}
\bigO{s_N} u_N^2 \intrd{2} \int_{0}^{\infty} \phi(z_1) \phi(z_2) V_r(z_1, z_2) \frac{\ld r}{r^{1+d+\alpha}} \ld z_1 \ld z_2.
\end{equation*}
We recover \eqref{variation_w_stable} since $V_r(z_1,z_2)\leq r^d \1{r\geq \frac{1}{2}|z_1-z_2|}$.
\end{proof}

As in Subsection~\ref{subsec:mart_pb:fixed}, we can now state a general result for a sequence of processes satisfying (M2) which will imply Theorem~\ref{thm:clt_slfv_stable_radius}.
Let $f^N$ be defined as in \eqref{definition_centering_term_stable}
and define $f$ as the solution to
\begin{equation}\label{definition_limit_stable}
\left\lbrace
\begin{aligned}
& \deriv*{f_t}{t} = u(\D f_t - \tfrac{sV_1}{\alpha} F(f_t)), \\
& f_{0} = w_0.
\end{aligned}
\right.
\end{equation}
The following result is proved in Section~\ref{sec:proof_stable}.

\begin{thm}\label{thm:clt_stable}
Suppose that $\proc{w^{N}}$ satisfies the martingale problem (M2) in Definition \ref{formal_def_M2} for some smooth, bounded function $F:\R\to\R$ with bounded first and second derivatives and $\proc{\delta_{N}}{N}$, $\proc{\varepsilon_{N}}{N}$ converging to zero as $N\to\infty$. 
Moreover, suppose 
\begin{equation}\label{parameters_regime_stable}
\tau_{N}/\eta_{N}=\littleO{\delta_{N}^{2\alpha}}.
\end{equation}
Suppose also that $w_{0}$ has uniformly bounded derivatives of up to the second order and that there exists $\alpha_N$ such that the jumps of $\proc{w^{N}}$ are dominated by
\begin{equation*}
\sup_{t\geq 0}\abs{\dual{w^{N}_{t}}-\dual{w^{N}_{t^-}}}\leq \alpha_{N}\norm[1]{\phi}
\end{equation*}
for every $\phi\in L^{1,\infty}(\R^d)$, with $\alpha_{N}^{2}=\littleO{\tau_{N}/\eta_{N}}$. Then
\begin{equation*}
\proc{\wtN}{t\geq 0} \cvgas[L^{1},\,P]{N} \proc{f}
\end{equation*}
in $\left( \sko{\Xi}, d \right)$. In addition, $$ \zt = (\eta_{N}/\tau_{N})^{1/2}(\wtN-\ft) $$ defines a sequence of distribution-valued processes which converges in distribution in $\sko{\calS'(\R^{d})}$ to the solution of the following \spde,
\begin{equation} \label{limiting_fluctuations_stable}
\left\lbrace
\begin{aligned}
& \ld z_{t} = u[\D z_{t} - \tfrac{sV_1}{\alpha} F'(f_{t})z_{t}]\ld t + u \cdot \ld W^\alpha_{t} \\
& z_{0} = 0,
\end{aligned}
\right.
\end{equation}
where $W^\alpha$ is a coloured noise with covariation measure given by \eqref{stable_noise_correlations}.
\end{thm}
Theorem \ref{thm:clt_slfv_stable_radius} is now a direct consequence.
\begin{proof}[Proof of Theorem \ref{thm:clt_slfv_stable_radius}]
Recall that $(\mathbf q^N_t)_{t\geq 0}$ is defined in \eqref{definition_rescaled_stable} as a rescaling of $\proc{q^N}$, and by Proposition \ref{prop:mart_pb_m2}, letting $w_t^N(x)=q_t^N(x/\delta_N)$, $\proc{w^N}$ satisfies the martingale problem (M2). 
Also $\tau_{N}/\eta_{N}=\littleO{\delta_{N}^{2\alpha }}$ follows from $\varepsilon_N = o(\delta _N^{2 \alpha })$, and the bound on the jumps \eqref{bound_jumps} holds with $\alpha_N=\varepsilon_N u$ by \eqref{update_w}. 
We conclude by applying Theorem~\ref{thm:clt_stable} to $\wtN = \mathbf{q}^N_t$.
\end{proof}

The proof of Theorem \ref{thm:clt_stable} will make use of the same ideas as in the proof of Theorem \ref{thm:clt_brownian} and, to improve readability, the steps of the proof which are most similar to those in the Brownian case will be dealt with more quickly, going into details only when the two arguments differ.

\section{The Brownian case - proof of Theorem~\ref{thm:clt_brownian}}\label{sec:proof_brownian}

As in the sketch of the proof in Subsection \ref{subsec:mart_pb:fixed}, for ease of notation, we shall set the constants $uV_R$, $2R^2/(d+2)$ and $s$ to $1$ in the definition of (M1). 
Recall the expression for $\dual{\zt}$ in \eqref{martingale_z}; the next subsection shows how time-dependent test functions can be used to write $\dual{\zt}$ as the sum of a stochastic integral against a martingale measure and a non-linear term. 
Subsection~\ref{subsec:brownian:regularity_estimate} will provide a bound on this quadratic term using a Gronwall estimate. 
We can then prove the convergence of the process $\proc{\wtN}{t\geq 0}$ to $\proc{f}$ in Subsection~\ref{subsec:brownian:convergence_w}.

The following result is used to reduce the convergence of distribution-valued processes to the convergence of a family of real-valued processes; it is a direct corollary of Mitoma's theorem \citep[Theorem 6.13]{walsh_introduction_1986}.
\begin{thm}[{\cite[Theorem 6.15]{walsh_introduction_1986}}]\label{thm:mitoma}
 Let $\proc{X^n}{n\geq 1}$ be a sequence of processes with sample paths in $\sko{\calS'(\R^d)}$. Suppose
\begin{enumerate}[label=\roman*)]
\item for each $\phi\in\calS(\R^d)$, $\proc{\dual{X^n}}{n\geq 1}$ is tight,
\item for each $\phi_1,\ldots,\phi_k$ in $\calS(\R^d)$ and $t_1,\ldots,t_k$ in $[0,T]$, the distribution of $(\dual{X^n_{t_1}}{\phi_1},\ldots,\dual{X^n_{t_k}}{\phi_k})$ converges weakly on $\R^k$.
\end{enumerate}
Then there exists a process $\proc{X}$ with sample paths in $\sko{\calS'(\R^d)}$ such that $X^n$ converges in distribution to $X$.
\end{thm}

In order to apply this result to the sequence of distribution-valued processes $\proc{Z^N}{N\geq 1}$, we need to check that the two conditions (\textit{i}) and (\textit{ii}) are satisfied. 
The first one is proved in Subsection~\ref{subsec:brownian:tightness}, thus implying the tightness of the sequence by Mitoma's theorem. 
Subsection~\ref{subsec:brownian:convergence_M} deals with the convergence of the martingale measure $M^N$ (again as a distribution valued process, so this subsection will use Theorem~\ref{thm:mitoma}). 
Finally condition (\textit{ii}) is checked in Subsection~\ref{subsec:brownian:conclusion}.

In this section, in order to simplify the notation we often drop the sub- and superscripts $N$ when there is no ambiguity; for instance, $\calL$ should always be read $\L$, with $r=r_N$. 

\subsection{Time dependent test functions}\label{subsec:brownian:time_dependent_tf}

Fix $\phi\in\calS(\R^{d})$. We consider time-dependent test functions $\varphi:\R^{d}\times\braced{(s,t):0\leq s\leq t\leq T}\to\R$ such that (with a slight abuse of notation) $\varphi(s,t)\in\calS(\R^{d})$ for all $0\leq s\leq t$ and $\varphi$ is continuously differentiable with respect to the time variables. The following is proved by adapting Exercise 5.1 of \cite{walsh_introduction_1986}.

\begin{proposition}\label{prop:time_dependent_tf}
 Let $M$ be a worthy martingale measure and suppose that $V_t$ is a mild solution to the following equation:
\begin{equation*}
\ld V_t = A_t(V_t)\ld t + \ld M_t.
\end{equation*}
Suppose that $A_t(V_t)$ is adapted and that this equation is well posed. Then if $\varphi$ is a time dependent test function,
\begin{equation*}
 \dual{V_t}{\varphi(t,t)} = \dual{V_0}{\varphi(0,t)} + \intg{0}{t}\braced{\dual{V_s}{\partial_s\varphi(s,t)} + \dual{A_s(V_s)}{\varphi(s,t)}}\ld s + \intg{0}{t}\intrd\varphi(x,s,t)M(\ld x\ld s).
\end{equation*}
\end{proposition}
Returning to \eqref{martingale_z}, we define a time dependent test function $\varphi^N$ as the solution to
\begin{equation} \label{definition_varphi_n}
\left\lbrace
\begin{aligned}
& \partial_{s}\varphi^{N}(x,s,t) + \L{r_N}\varphi^{N}(x,s,t) - \avg{F'(\avg{\ft{s}})\avg{\varphi^{N}(s,t)}}(x,r_N) = 0, \\
& \varphi^{N}(x,t,t) = \phi(x).
\end{aligned}
\right.
\end{equation}
It is straightforward to check that $\varphi^N(s,t)\in\calS(\R^d)$ for all $0\leq s\leq t$.
Proposition~\ref{prop:time_dependent_tf} and \eqref{martingale_z} then yield
\begin{equation} \label{stochastic_integral_zn}
\dual{\zt} = -(\tau_{N}/\eta_{N})^{1/2}\intg{0}{t}\dual{(\avg{\zt{s}})^{2}}{R_{2}(\avg{\wt{s}},\avg{\ft{s}})\avg{\varphi^{N}(s,t)}}\ld s + \intg{0}{t}\intrd\varphi^{N}(x,s,t)M^{N}(\ld x\ld s).
\end{equation}
Here we see that in the special case where $F$ is linear, $R_2 = 0$ and it remains to prove the convergence of the stochastic integral of $\varphi^N$ against the martingale measure $M^N$. 
Using \cite[Theorem 7.13]{walsh_introduction_1986} we need only prove the convergence of $M^N$ and that of $\varphi^N$ to $\varphi$, where $\varphi$ solves
\begin{equation} \label{definition_varphi}
\left\lbrace
\begin{aligned}
& \partial_{s}\varphi(x,s,t) + \frac{1}{2}\Delta\varphi(x,s,t) - F'(f_{s}(x))\varphi(x,s,t) = 0, \\
& \varphi(x,t,t) = \phi(x).
\end{aligned}
\right.
\end{equation}
The following lemma, whose proof is given in Appendix \ref{append:test_functions}, provides the convergence of $\varphi^N$ to $\varphi$.
\begin{lemma}\label{lemma:convergence_varphi}
 For $T>0$, there exists a constant $K_1$ such that, for all $N\geq 1$ and for $q\in\left\lbrace 1,2 \right\rbrace$,
\begin{equation*}
 \sup_{0\leq s\leq t\leq T}\norm[q]{\varphi^{N}(s,t)-\varphi(s,t)} \leq K_{1}r_{N}^{2}.
\end{equation*}
In addition, there exist constants $K_2$ and $K_3$ such that, for $0<\abs{\beta}\leq 4$,
\begin{align*}
\sup_{0\leq s\leq t\leq T}\norm[q]{\varphi^N(s,t)} \leq K_2 \norm[q]{\phi}, && and && \sup_{0\leq s\leq t\leq T}\norm[q]{\partial_{\beta}\varphi^{N}(s,t)}  \leq K_{3}.
\end{align*}
and $K_2$ does not depend on $\phi$.
\end{lemma}

\begin{remark}
Recall the definition of $R_1$ in \eqref{taylor_F_1};
it is tempting to try to define $\varphi^N$ as the solution to
\begin{equation*}
\left\lbrace
\begin{aligned}
& \partial_{s}\varphi^{N}(x,s,t) + \L{r_N}\varphi^{N}(x,s,t) - \avg{R_{1}(\avg{\wt{s}},\avg{\ft{s}})\avg{\varphi^{N}(s,t)}}(x,r_N) = 0, \\
& \varphi^{N}(x,t,t) = \phi(x).
\end{aligned}
\right.
\end{equation*}
In this way, according to our previous calculations in \eqref{martingale_z} and using Proposition \ref{prop:time_dependent_tf}, we would get rid of the first integral in \eqref{stochastic_integral_zn}. 
However, in this case, $s\mapsto\varphi^{N}(\cdot,s,\cdot)$ is not adapted to the canonical filtration of our process and the stochastic integral with respect to the martingale measure $M^{N}$ is not well defined.
\end{remark}

\subsection{Regularity estimate}\label{subsec:brownian:regularity_estimate}

The following result is an easy consequence of the definition of $M^N$.
\begin{lemma}\label{lemma:bound_qvar}
There exists a constant $K_4$ such that if
for all $0\leq t\leq T$, $\phi_{t}:\R^{d}\to\R$ is in $L^{2}(\R^{d})$, then
\begin{equation*}
\E{\left(\intg{0}{t}\intrd\phi_{s}(x)M^{N}(\ld x\ld s)\right)^{2}} \leq K_{4}\intg{0}{t}\norm[2]{\phi_{s}}^{2}\ld s.
\end{equation*}
\end{lemma}

\begin{proof}
 From the definition of $Q^N$ in \eqref{covariation_MN} and the definition of $\sigma^{(r)}_{z_1,z_2}$ in \eqref{def_sigma_r},
\begin{align*}
&\E{\left(\intg{0}{t}\intrd\phi_{s}(x)M^{N}(\ld x\ld s)\right)^{2}} \\
&\hspace{1cm}= \E{\intg{0}{t}\intrd{2}\phi_{s}(z_{1})\phi_{s}(z_{2})\sigma_{z_{1},z_{2}}(\wt{s})\ld z_{1}\ld z_{2}\ld s}+\bigO{\delta_N^2 \int_0^t \norm[2]{\phi_s}^2 ds} \\
&\hspace{1cm}\leq \intg{0}{t}\intrd{3}\frac{1}{V_{r}^{2}}\inball{x}{z_{1}}{x}{z_{2}}\abs{\phi_{s}(z_{1})}\abs{\phi_{s}(z_{2})}\ld x\ld z_{1}\ld z_{2}\ld s +\bigO{\delta_N^2 \int_0^t \norm[2]{\phi_s}^2 ds}\\
&\hspace{1cm}= \intg{0}{t}\intrd \left(\frac{1}{V_{r}}\intbr{x}\abs{\phi_{s}(z)}\ld z\right)^{2}\ld x\ld s +\bigO{\delta_N^2 \int_0^t \norm[2]{\phi_s}^2 ds}\\
&\hspace{1cm}\leq K_{4}\intg{0}{t}\norm[2]{\phi_{s}}^{2}\ld s.
\end{align*}
(We have used Jensen's inequality in the last line.)
\end{proof}

For $t>0$ and $x\in\R^{d}$, let $$ G_{t}(x) = (2\pi t)^{-d/2}\exp\left(-\frac{\abs{x}^{2}}{2t}\right) $$ be the fundamental solution to the heat equation on $\R^d$; $\phi\mapsto G_{t}\ast\phi$ is then the semigroup of standard Brownian motion. 
Then $f_t$ as defined in \eqref{fisher_kpp} satisfies
\begin{equation*}
 f_t(x) = G_{t}\ast w_0(x) - \intg{0}{t}G_{t-s}\ast F(f_s)(x)\ld s.
\end{equation*}
Likewise, for $r>0$, let $\proc{\xi^{(r)}}$ be a symmetric L\'evy process on $\R^{d}$ with generator $\phi\mapsto\L\phi$ and
let $\G$ be the corresponding semigroup.
Note that since $\xi^{(r)}_t=0$ with positive probability, $\G$ is not a well-defined function, but we do have $\avg{\G_{t}}\in L^{1,\infty}$. 
Then $f^N$ as defined in \eqref{definition:centering_term_fixed_radius} satisfies
\begin{equation} \label{centering_term_green_function}
f^{N}_{t}(x) = \G_{t}\ast w_{0}(x) - \intg{0}{t}\G_{t-s}\ast\avg{F(\avg{\ft{s}})}(x)\ld s.
\end{equation}

The following provides a bound on the second moment of $\avg{\zt}$, which allows us to control the quadratic term in \eqref{stochastic_integral_zn}. Note that $x\mapsto\avg{\zt}(x)$ is a well defined function (despite the fact that $\wt$ is only defined up to a Lebesgue-null set) and that for each $N\geq 1$, it is uniformly bounded on $\R^{d}$ (by $(1+\norm[\infty]{f^N_t})(\eta_N/\tau_N)^{1/2}$).
\begin{lemma} \label{lemma_bound_avg}
For $T>0$, there exists a constant $K_{5}>0$, independent of $N$, such that for $0\leq t\leq T,$
\begin{equation*}
\sup_{x\in\R^{d}}\E{\avg{\zt}(x,r_N)^{2}} \leq \frac{K_{5}}{V_{r_N}}.
\end{equation*}
\end{lemma}
The proof of this result mirrors that of Theorem 1 in \cite{norman_approximation_1975}, although it is more technical because of the Laplacian and the various spatial averages.

\begin{proof}
Coming back to equation \eqref{martingale_z}, and using \eqref{taylor_F_1} instead of \eqref{taylor_F_2}, we write
\begin{align} \label{z_average}
\dual{\zt} &= \intg{0}{t}\braced{\dual{\calL\zt{s}}-\dual{\avg{\avg{\zt{s}}R_{1}(\avg{\wt{s}},\avg{\ft{s}})}}}\ld s + \intg{0}{t}\intrd\phi(y)M^{N}(\ld y\ld s).
\end{align}
To get rid of the operator $\calL$, we use Proposition~\ref{prop:time_dependent_tf} with the time-dependent test function $\G_{t-s}\ast\phi$, yielding
\begin{align*}
\dual{\zt} &= -\intg{0}{t}\dual{\G_{t-s}\ast\left(\avg{\avg{\zt{s}}R_{1}(\avg{\wt{s}},\avg{\ft{s}})}\right)}\ld s + \intg{0}{t}\intrd \G_{t-s}\ast\phi(y)M^{N}(\ld y\ld s).
\end{align*}
Now we take $\phi(y)=\frac{1}{V_{r}}\inball{x}{y}$, and we obtain
\begin{align*}
\avg{\zt}(x) &= -\intg{0}{t}\G_{t-s}\ast\left(\dbavg{\avg{\zt{s}}R_{1}(\avg{\wt{s}},\avg{\ft{s}})}\right)(x)\ld s + \intg{0}{t}\intrd\avg{\G_{t-s}}(x-y)M^{N}(\ld y\ld s) \\
&= -\intg{0}{t}\intrd\dbavg{\G_{t-s}}(x-y)\avg{\zt{s}}(y)R_{1}(\avg{\wt{s}}(y),\avg{\ft{s}}(y))\ld y \ld s + \intg{0}{t}\intrd \avg{\G_{t-s}}(x-y)M^{N}(\ld y\ld s).
\end{align*}
We now want to apply Gronwall's lemma, but the last term must be controlled carefully. Taking the square of both sides and using $(a+b)^{2}\leq 2(a^{2}+b^{2})$, we have
\begin{multline*}
\avg{\zt}(x)^{2} \leq 2\left(\intg{0}{t}\intrd\dbavg{\G_{t-s}}(x-y)\avg{\zt{s}}(y)R_{1}(\avg{\wt{s}}(y),\avg{\ft{s}}(y))\ld y\ld s\right)^{2} \\ + 2\left(\intg{0}{t}\intrd \avg{\G_{t-s}}(x-y)M^{N}(\ld y\ld s)\right)^{2}.
\end{multline*}
By Jensen's inequality (and noting that $\intrd \avg {\G_{t}}(x)\ld x = 1$), we have
\begin{align*}
\avg{\zt}(x)^{2} &\leq 2t\intg{0}{t}\intrd\dbavg{\G_{t-s}}(x-y)\norm[\infty]{F'}^{2}\avg{\zt{s}}(y)^{2}\ld y\ld s + 2\left(\intg{0}{t}\intrd \avg{\G_{t-s}}(x-y)M^{N}(\ld y\ld s)\right)^{2}.
\end{align*}
Taking expectations on both sides and using Fubini's theorem, we obtain
\begin{multline*}
\E{\avg{\zt}(x)^{2}} \leq 2t\norm[\infty]{F'}^{2}\intg{0}{t}\intrd\dbavg{\G_{t-s}}(x-y)\E{\avg{\zt{s}}(y)^{2}}\ld y\ld s \\ + 2\E{\left(\intg{0}{t}\intrd \avg{\G_{t-s}}(x-y)M^{N}(\ld y\ld s)\right)^{2}}.
\end{multline*}
From Lemma~\ref{lemma:bound_qvar}, we have
\begin{align*}
\E{\left(\intg{0}{t}\intrd \avg{\G_{t-s}}(x-y)M^{N}(\ld y\ld s)\right)^{2}} &\leq K_{4}\intg{0}{t}\norm[2]{\avg{\G_{t-s}}(x-\cdot)}^{2}\ld s \\
&= K_{4}\intg{0}{t}\intrd \E[0]{\frac{1}{V_r}\inball{\xi^{(r)}_{t-s}}{(x-y)}}^2\ld y\ld s\\
&\leq K_{4}\intg{0}{t} \E[0]{\intrd \left(\frac{1}{V_r}\inball{\xi^{(r)}_{t-s}}{(x-y)}\right)^2\ld y}\ld s\\
&= \frac{K_4}{V_{r}}t. \numberthis \label{1_over_Vr_bound}
\end{align*}
($\E[0]{\cdot}$ denotes the expectation with respect to the law of $\proc{\xi^{(r)}}$ started from the origin.)
In addition, we note that $\E{\avg{\zt{s}}(y)^{2}}\leq\sup_{x\in\R^{d}}\E{\avg{\zt{s}}(x)^{2}}$ and, combined with the fact that $\intrd \avg{\G_{t}}(x)\ld x = 1$, this yields
\begin{align*}
\E{\avg{\zt}(x)^{2}} &\leq 2t\norm[\infty]{F'}^{2}\intg{0}{t}\sup_{y\in\R^{d}}\E{\avg{\zt{s}}(y)^{2}}\ld s + 2\frac{K_4}{V_{r}}t.
\intertext{The right hand side does not depend on $x$, so we can take the supremum over $x\in\R^{d}$ on the left and write for $0\leq t\leq T$}
\sup_{x\in\R^{d}}\E{\avg{\zt}(x)^{2}} &\leq 2T\norm[\infty]{F'}^{2}\intg{0}{t}\sup_{x\in\R^{d}}\E{\avg{\zt{s}}(x)^{2}}\ld s + 2\frac{K_4}{V_{r}}T.
\end{align*}
Finally, we can apply Gronwall's lemma to deduce that
\begin{equation*}
\sup_{x\in\R^{d}}\E{\avg{\zt}(x)^{2}} \leq 2\frac{K_4}{V_{r}}Te^{2Tt\norm{F'}^{2}} \leq \frac{K_{5}}{V_{r}}.
\end{equation*}
\end{proof}

\subsection{Convergence to the deterministic limit}\label{subsec:brownian:convergence_w}
The following result, proved in Appendix~\ref{append:centering_term}, shows that $f^N$ converges to $f$. 

\begin{proposition} \label{prop:convergence_centering_term}
For $T>0$, there exist constants $K_6$ and $K_7$ such that, for all $N\geq 1$,
\begin{equation*}
\sup_{0\leq t\leq T}\norm[\infty]{\ft-f_t} \leq K_6 r_{N}^{2},
\end{equation*}
and, for all $0\leq \abs{\beta}\leq 4$, 
\begin{equation*}
\sup_{0\leq t\leq T}\norm[\infty]{\partial_{\beta}\ft} \leq K_{7},
\end{equation*}
where $\partial_\beta f$ is the spatial derivative with respect to the multi-index $\beta$.
\end{proposition}
We are now in a position to prove the first statement of Theorem~\ref{thm:clt_brownian}, namely the convergence of the process $\proc{w^{N}}$. 
We are going to prove the following lemma.
\begin{lemma} \label{lem_sup_zt_bound}
There exists a constant $K_{8}$ such that for all $N\geq 1$ and for any function $\phi$ satisfying $\norm[q]{\phi}\leq 1$ and $\max_{\abs{\beta}=2}\norm[q]{\partial_{\beta}\phi}\leq 1$ for $q\in\braced{1,2}$,
\begin{equation} \label{bound_sup_z}
\E{\sup_{0\leq t\leq T}\abs{\dual{\zt}}} \leq K_{8}.
\end{equation}
\end{lemma}
Before we prove Lemma \ref{lem_sup_zt_bound}, we show that it implies the convergence of $\proc{w^N}$.
We can choose a separating family $\proc{\phi_{n}}{n\geq 1}$ of compactly supported smooth functions satisfying $\norm[q]{\phi}\leq 1$ and $\max_{\abs{\beta}=2}\norm[q]{\partial_{\beta}\phi}\leq 1$ for $q\in\braced{1,2}$, and define $d$ as in \eqref{d_defn_family} using this family. Then 
\begin{align*}
\E{\sup_{0\leq t\leq T}d(\wt,f_{t})} &\leq \sum_{n\geq 1}\frac{1}{2^{n}}\braced{\E{\sup_{0\leq t\leq T}\abs{\dual{\wt}{\phi_{n}}-\dual{\ft}{\phi_{n}}}} + \sup_{0\leq t\leq T}\abs{\dual{\ft}{\phi_{n}}-\dual{f_{t}}{\phi_{n}}}} \\
&\leq \sum_{n\geq 1}\frac{1}{2^{n}}\braced{(\tau_{N}/\eta_{N})^{1/2}\E{\sup_{0\leq t\leq T}\abs{\dual{\zt}{\phi_{n}}}} + \sup_{0\leq t\leq T}\norm[\infty]{\ft-f_{t}}\norm[1]{\phi_{n}}}\\
&\leq \sum_{n\geq 1}\frac{1}{2^{n}}\braced{K_{8}(\tau_{N}/\eta_{N})^{1/2} + K_{6}r_{N}^{2}} = K_{8}(\tau_{N}/\eta_{N})^{1/2} + K_{6}r_{N}^{2},
\end{align*}
where the last line follows by Proposition~\ref{prop:convergence_centering_term} and Lemma \ref{lem_sup_zt_bound}.
The right-hand-side converges to zero as $N\to\infty$, yielding the uniform convergence (on compact time intervals) of $\proc{w^N}$ to $\proc{f}$, the solution of equation \eqref{fisher_kpp}. Note that, as soon as $d\geq 2$, $r_N^2$ is the leading order on the right-hand-side (see \eqref{parameters_regime_brownian}).

\begin{proof}[Proof of Lemma \ref{lem_sup_zt_bound}]
We are going to make use of \eqref{martingale_z} and apply Doob's maximal inequality to the martingale part. Let us first show that there exist two constants $K$ and $K'$ such that, for $t\in[0,T]$,
\begin{equation} \label{L1_bound_z}
\E{\abs{\dual{\zt}}} \leq K\norm[2]{\phi} + K'\frac{(\tau/\eta)^{1/2}}{V_{r}}\norm[1]{\phi}.
\end{equation}
Indeed, taking the expectation of the absolute value of both sides of \eqref{stochastic_integral_zn} and using Lemma~\ref{lemma:bound_qvar}, we have
\begin{align*}
\E{\abs{\dual{\zt}}} & \leq (\tau/\eta)^{1/2}\frac{1}{2}\norm[\infty]{F''}\intg{0}{t}\dual{\E{(\avg{\zt{s}})^{2}}}{\abs{\avg{\varphi^{N}(s,t)}}}\ld s + \left(K_{4}\intg{0}{t}\norm[2]{\varphi^{N}(s,t)}^{2}\ld s\right)^{1/2}. \\
& \leq \frac{1}{2}\norm[\infty]{F''}K_{5}T\frac{(\tau/\eta)^{1/2}}{V_{r}} K_2 \norm[1]{\phi} + K_{4}^{1/2}T^{1/2} K_2 \norm[2]{\phi},
\end{align*}
where we used Lemmas~\ref{lemma_bound_avg} and~\ref{lemma:convergence_varphi} in the last line.
We have thus proved \eqref{L1_bound_z}. Recalling \eqref{martingale_z} and the notation $M^{N}_{t}(\phi) = \intg{0}{t}\intrd\phi(x)M^{N}(\ld x\ld s)$, we write
\begin{multline*}
\sup_{t\in[0,T]}\abs{\dual{\zt}} \leq \intg{0}{T}\abs{\dual{\zt{s}}{\calL\phi - \avg{F'(\avg{\ft{s}})\avg{\phi}}}}\ld s + \frac{1}{2}\norm[\infty]{F''}(\tau/\eta)^{1/2}\intg{0}{T}\dual{(\avg{\zt{s}})^{2}}{\abs{\avg{\phi}}}\ld s \\ + \sup_{t\in[0,T]}\abs{M^{N}_{t}(\phi)}.
\end{multline*}
Taking expectations on both sides, we use Lemma~\ref{lemma_bound_avg} and apply \eqref{L1_bound_z} with $\phi$ replaced by $\calL\phi - \avg{F'(\avg{\ft{s}})\avg{\phi}}$ to write
\begin{multline} \label{sup_zt_expr}
\E{\sup_{t\in[0,T]}\abs{\dual{\zt}}} \leq \intg{0}{T}\braced{K(\norm[2]{\calL\phi} + \norm[\infty]{F'}\norm[2]{\phi}) + K'\frac{(\tau/\eta)^{1/2}}{V_{r}}(\norm[1]{\calL\phi} + \norm[\infty]{F'}\norm[1]{\phi})}\ld s \\ + \frac{1}{2}\norm[\infty]{F''}K_{5}T\frac{(\tau/\eta)^{1/2}}{V_{r}}\norm[1]{\phi} + \E{\sup_{t\in[0,T]}\abs{M^{N}_{t}(\phi)}^{2}}^{1/2}.
\end{multline}
By Doob's inequality and Lemma \ref{lemma:bound_qvar},
\begin{equation*}
\E{\sup_{t\in[0,T]}\abs{M^{N}_{t}(\phi)}^{2}} \leq 4K_{4}T\norm[2]{\phi}^{2}.
\end{equation*}
Furthermore, $\norm[q]{\calL\phi} \leq \frac{d(d+2)}{2}\max_{\abs{\beta}=2}\norm[q]{\partial_{\beta}\phi}$ by Proposition~\ref{prop:average}.i in Appendix~\ref{append:average}, and $\frac{(\tau/\eta)^{1/2}}{V_{r}}$ tends to zero as $N\to\infty$ due to assumption \eqref{parameters_regime_brownian}. Hence, if $\norm[q]{\phi}\leq 1$ and $\max_{\abs{\beta}=2}\norm[q]{\partial_{\beta}\phi}\leq 1$ for $q\in\braced{1,2}$, the right-hand-side of \eqref{sup_zt_expr} is bounded by some constant independent of $N$ and $\phi$.
\end{proof}

\subsection{Tightness}\label{subsec:brownian:tightness}

To prove that the sequence $\proc{\zt{}}{N\geq 1}$ is tight in $\sko{\calS'(\R^{d})}$, we adapt the argument from the proof of Theorem 7.13 in \cite{walsh_introduction_1986}. 
\begin{proposition} \label{tightness_prop}
For any $\phi\in\calS(\R^d)$, for any arbitrary sequence $\proc{T_{N},\rho_{N}}{N\geq 1}$ such that $T_{N}$ is a stopping time (with respect to the natural filtration of the process $\proc{Z^{N}}$) with values in $[0,T]$ for all $N$ and $\rho_{N}$ is a deterministic sequence of positive numbers decreasing to zero as $N\to\infty$,
\begin{equation} \label{aldous_criterion_zn}
\dual{\zt{T_{N}+\rho_{N}}}-\dual{\zt{T_{N}}} \to 0
\end{equation}
in probability as $N\to \infty$.
\end{proposition}
By Aldous' criterion (\cite{aldous_stopping_1978} and \citep[Theorem 6.8]{walsh_introduction_1986}), Proposition \ref{tightness_prop} together with Lemma~\ref{lem_sup_zt_bound}  imply that the sequence of real-valued processes $\proc{\dual{\zt{}}}{N\geq 1}$ is tight in $\sko{\R}$ for any $\phi\in\calS(\R^{d})$. In turn, Mitoma's theorem \citep[Theorem 6.13]{walsh_introduction_1986} implies the tightness of $\proc{\zt{}}{N\geq 1}$ in $\sko{\calS'(\R^{d})}$.

\begin{proof}[Proof of Proposition \ref{tightness_prop}] 
We are going to treat each term in \eqref{stochastic_integral_zn} separately. The first one converges to zero in $L^{1}$, uniformly on $[0,T]$, as a consequence of Lemma~\ref{lemma_bound_avg}. The second one is dealt with as in \citep[Theorem 7.13]{walsh_introduction_1986}. 

The proof requires three auxiliary lemmas as follows (the first two are proved in Appendix~\ref{append:test_functions}).
We extend $\varphi^{N}$ to $\R^d \times [0,T]^{2}$ by setting, for $s,t\in[0,T]$,
\begin{equation}\label{varphi_extension}
\varphi^{N}(x,s,t) := \varphi^{N}(x,s\wedge t,t).
\end{equation}
In other words, for $s>t$, $\varphi^{N}(s,t)$ equals $\phi$. 
\begin{lemma}\label{lemma:continuity_varphi}
 For $T>0$, there exists a constant $K_9$ such that, for all $N\geq 1$ and for $q\in\left\lbrace 1,2 \right\rbrace$,
\begin{equation*}
 \forall s,t,t'\in[0,T], \quad \norm[q]{\varphi^{N}(s,t')-\varphi^{N}(s,t)} \leq K_{9}\abs{t'-t}.
\end{equation*}
\end{lemma}
\begin{lemma} \label{lemma_varphi_uniform_bound}
There exists a constant $K_{10}$ such that, for all $s\in[0,T]$,
\begin{equation*}
\norm[1]{\sup_{t\in[s,T]}\abs{\varphi^{N}(s,t)}} \leq K_{10}.
\end{equation*}
\end{lemma}
Now define
\begin{equation*}
V^{N}_{t} = \intg{0}{T}\intrd\varphi^{N}(x,s,t)M^{N}(\ld x\ld s).
\end{equation*}
\begin{lemma} \label{holder_vn_lemma}
For any $0<\beta<1/2$, there exists a random variable $Y_{N}$ such that
\begin{equation} \label{holder_vn}
\forall t,t'\in [0,T], \quad \abs{V^{N}_{t'}-V^{N}_{t}} \leq Y_{N}\abs{t'-t}^{\beta},
\end{equation}
almost surely, and $\E{Y_{N}^{2}} \leq C'$ for all $N\geq 1$. 
\end{lemma}
\begin{proof}
By Lemma~\ref{lemma:bound_qvar} and then Lemma \ref{lemma:continuity_varphi},
\begin{align*}
\E{\abs{V^{N}_{t'}-V^{N}_{t}}^{2}} &= \E{\left(\intg{0}{T}\intrd(\varphi^{N}(x,s,t')-\varphi^{N}(x,s,t))M^{N}(\ld x\ld s)\right)^{2}}.\\
&\leq K_{4}\intg{0}{T}\norm[2]{\varphi^{N}(s,t')-\varphi^{N}(s,t)}^{2}\ld s \\
&\leq (K_{9})^{2}TK_{4}\abs{t'-t}^{2}.
\end{align*}
The result follows by Kolmogorov's continuity theorem \citep[Corollary 1.2]{walsh_introduction_1986}.
\end{proof}
Returning to the process $\dual{\zt}$, by \eqref{stochastic_integral_zn}, we can write
\begin{multline} \label{tightness_expr}
\dual{\zt{T_{N}+\rho_{N}}}-\dual{\zt{T_{N}}} = (\tau/\eta)^{1/2}\intg{0}{T_{N}}\dual{(\avg{\zt{s}})^{2}}{R_{2}(\avg{\wt{s}},\avg{\ft{s}})\avg{\varphi^{N}(s,T_{N})}}\ld s \\ \hspace{3.5cm} - (\tau/\eta)^{1/2}\intg{0}{T_{N}+\rho_{N}}\dual{(\avg{\zt{s}})^{2}}{R_{2}(\avg{\wt{s}},\avg{\ft{s}})\avg{\varphi^{N}(s,T_{N}+\rho_{N})}}\ld s \\ + \left(V^{N}_{T_{N}+\rho_{N}}-V^{N}_{T_{N}}\right) + \intg{T_{N}}{T_{N}+\rho_{N}}\intrd\phi(x)M^{N}(\ld x\ld s).
\end{multline}
Let us deal with each term separately. The first two are similar so we need only consider the first one. Since inside the integral $s\leq T_{N}\leq T$, $\abs{\varphi^N(s,T_N)}\leq \sup _{t\in[s,T]}\abs{\varphi^N(s,t)}$ and we have
\begin{equation*}
\abs{\intg{0}{T_{N}}\dual{(\avg{\zt{s}})^{2}}{R_{2}(\avg{\wt{s}},\avg{\ft{s}})\avg{\varphi^{N}(s,T_{N})}}\ld s} \leq \frac{1}{2}\norm[\infty]{F''}\intg{0}{T}\dual{(\avg{\zt{s}})^{2}}{\avg{\sup_{t\in[s,T]}\abs{\varphi^{N}(s,t)}}}\ld s.
\end{equation*}
Taking the expectation on both sides, we get
\begin{align} \label{bound_error_term}
\E{\abs{\intg{0}{T_{N}}\dual{(\avg{\zt{s}})^{2}}{R_{2}(\avg{\wt{s}},\avg{\ft{s}})\avg{\varphi^{N}(s,T_{N})}}\ld s}} &\leq \frac{1}{2}\norm[\infty]{F''}\intg{0}{T}\dual{\E{(\avg{\zt{s}})^{2}}}{\avg{\sup_{t\in[s,T]}\abs{\varphi^{N}(s,t)}}}\ld s \notag \\
&\leq \frac{1}{2}\norm[\infty]{F''}\frac{K_{5}}{V_{r}}\intg{0}{T}\norm[1]{\sup_{t\in[s,T]}\abs{\varphi^{N}(s,t)}}\ld s \notag \\
&\leq \frac{1}{2}\norm[\infty]{F''}\frac{1}{V_{r}}T K_{5} K_{10}.
\end{align}
where the second line follows by Lemma~\ref{lemma_bound_avg} and the third line follows by Lemma \ref{lemma_varphi_uniform_bound}.
Recall that we assumed in \eqref{parameters_regime_brownian} that $\tau_N/\eta_N = \littleO{r_N^{2d}}$; hence the first term on the right-hand-side of \eqref{tightness_expr} converges to zero in $L^1$. By Lemma \ref{holder_vn_lemma}, we have, almost surely,
\begin{equation*}
\abs{V^{N}_{T_{N}+\rho_{N}}-V^{N}_{T_{N}}} \leq Y_{N}\rho_{N}^{1/4}.
\end{equation*}
Taking the expectation of the square of both sides, we write
\begin{equation*}
\E{\abs{V^{N}_{T_{N}+\rho_{N}}-V^{N}_{T_{N}}}^{2}} \leq C'\rho_{N}^{1/2}.
\end{equation*}
Hence the third term converges to zero in $L^{2}$ and in probability as $N\to\infty$. Finally, since $T_{N}$ is a stopping time, we can apply Lemma~\ref{lemma:bound_qvar} to the fourth term,
\begin{align*}
\E{\left(\intg{T_{N}}{T_{N}+\rho_{N}}\intrd\phi(x)M^{N}(\ld x\ld s)\right)^{2}} &\leq K_{4}\E{\intg{T_{N}}{T_{N}+\rho_{N}}\norm[2]{\phi}^{2}\ld s} \\
&\leq K_{4}\norm[2]{\phi}^{2}\rho_{N}.
\end{align*}
This concludes the proof of Proposition \ref{tightness_prop}.
\end{proof}

\subsection{Convergence of the martingale measure $M^N$}\label{subsec:brownian:convergence_M}

The next step is to show that the martingale measure $M^{N}$ converges weakly in $\sko{\calS'(\R^{d})}$ as $N\to\infty$ to $M$, where $M_{t} = \sqrt{f_{t}(1-f_{t})}\cdot W_{t}$ is a stochastic integral against the space-time white noise $W$ and $f$ is the solution of \eqref{fisher_kpp}. We will naturally use Theorem~\ref{thm:mitoma}, along with the following result on convergence to Gaussian martingales (which is a consequence of L\'evy's characterisation of Brownian motion). For any $\R^d$-valued process $\proc{Y}$, define $\Delta Y_t = Y_t-Y_{t^-}$.

\begin{thm}[{\cite[Theorem VIII 3.11]{jacod_limit_1987}}]\label{thm:convergence_gaussian_martingale}
Suppose $\proc{X}=(X^1_t,\ldots,X^d_t)_{t\geq 0}$ is a continuous $d$-dimensional Gaussian martingale and for each $n\geq 1$, $\proc{X^n}=(X^{n,1}_t,\ldots,X^{n,d}_t)_{t\geq 0}$ is a local martingale such that
\begin{enumerate}[label=(\roman*)]
\item $\abs{\Delta X^n_t}$ is bounded uniformly in $n$ for all $t$, and $\sup_{t\leq T}|\Delta X_t^n|\cvgas[P]{n} 0$.
\item For each $t\in \Q\cap [0,T]$, $\dual{X^{n,i}}{X^{n,j}}_t \cvgas[P]{n} \dual{X^{i}}{X^{j}}_t$.
\end{enumerate}
Then $X^n$ converges in distribution to $X$ in $D([0,T],\R^d)$.
\end{thm}

In our setting, the limiting process $\proc{M_t(\phi)}{t\geq 0}$ is a continuous martingale with quadratic variation
\begin{equation*}
\qvar{M(\phi)} = \intg{0}{t}\intrd\phi(x)^{2}f_{s}(x)(1-f_{s}(x))\ld x\ld s.
\end{equation*}
(See \cite[Theorem 2.5]{walsh_introduction_1986}.) Since this quantity is deterministic, $\proc{M_t(\phi)}{t\geq 0}$ is Gaussian, and we can apply the result above. The following lemma is then enough to conclude that $M^N$ converges to $M$.
\begin{lemma}\label{lemma:convergence_MN}
 For any $\phi\in\calS(\R^d)$,
\begin{enumerate}[label=\roman*)]
 \item For all $t\geq 0$, $\abs{\Delta M^N_t(\phi)}\leq K$ for some constant $K$, and $\sup_{0\leq t\leq T}\abs{\Delta M^N_t(\phi)} \cvgas[P]{N} 0$.
\item For each $t\in[0,T]$, $\qvar{M^N(\phi)} \cvgas[P]{N} \qvar{M(\phi)}$.
\end{enumerate}
\end{lemma}
Indeed, by polarisation, we can recover $\dual{M^N(\phi_i)}{M^N(\phi_j)}_t$ from $\qvar{M^N(\phi_i+\phi_j)}$ and $\langle M^N(\phi_i$ $-\phi_j)\rangle _{t}$, and (ii) of Theorem~\ref{thm:convergence_gaussian_martingale} is satisfied by vectors of the form $\proc{M^N_t(\phi_1),\ldots,M^N_t(\phi_p)}{t\geq 0}$. As a result, for any $\phi_1,\ldots,\phi_p$ in $\calS(\R^d)$, $\proc{M^N_t(\phi_1),\ldots,M^N_t(\phi_p)}{t\geq 0}$ converges in distribution to $\proc{M_t(\phi_1),\ldots,M_t(\phi_p)}{t\geq 0}$ in $\sko{\R^d}$. In particular, for any $\phi\in\calS(\R^d)$, $\proc{M^N(\phi)}{N\geq 1}$ is tight, and $M^N$ satisfies the assumptions of Theorem~\ref{thm:mitoma}, hence $M^N$ converges in distribution to $M$ as $N\to\infty$ in $\sko{\calS'(\R^d)}$.

\begin{proof}[Proof of Lemma~\ref{lemma:convergence_MN}]
By the definition of $M^{N}(\phi)$,
\begin{align*}
M^{N}_{t}(\phi)-M^{N}_{t^{-}}(\phi) &= \dual{\zt}-\dual{\zt{t^{-}}} \\
&= (\eta_{N}/\tau_{N})^{1/2}\left(\dual{\wt}-\dual{\wt{t^{-}}}\right).
\end{align*}
The bound on the jumps of $\dual{w_t}$ in \eqref{bound_jumps} implies
\begin{align*}
\sup_{t\geq 0}\abs{\Delta M^{N}_t(\phi)} &\leq (\eta_{N}/\tau_{N})^{1/2}\sup_{t\geq 0}\abs{\dual{w^{N}_{t}-w^{N}_{t^{-}}}{\phi}} \\
&\leq \alpha_N(\eta_N/\tau_N)^{1/2}\norm[1]{\phi}.
\end{align*}
But we have assumed that $\alpha_N^2=\littleO{\tau_N/\eta_N}$, so (i) is satisfied. To prove (ii), recall from \eqref{covariation_MN} that
\begin{equation*}
\qvar{M^{N}(\phi)} = \intg{0}{t}\intrd{2}\phi(z_{1})\phi(z_{2})\sigma^{(r_N)}_{z_{1},z_{2}}(\wt{s})\ld z_{1}\ld z_{2}\ld s+\bigO{\delta_N^2 t\norm[2]{\phi}^2}.
\end{equation*}
The rationale here is to show that the main contribution to this term comes from the diagonal $\{(z_1,z_2):z_1=z_2\}$ when $r\to 0$. 
From the definition of $\sigma^{(r_N)}$ in \eqref{def_sigma_r},
\begin{multline*}
\intrd{2}\phi(z_{1})\phi(z_{2})\sigma^{(r_N)}_{z_{1},z_{2}}(\wt{s})\ld z_{1}\ld z_{2} = \frac{1}{V_{r}^{2}}\intrd{3}\phi(z_{1})\phi(z_{2})[\avg{\wt{s}}(x)(1-\wt{s}(z_{1}))(1-\wt{s}(z_{2})) \\ + (1-\avg{\wt{s}}(x))\wt{s}(z_{1})\wt{s}(z_{2})]\inball{z_{1}}{x}{z_{2}}{x}\ld x\ld z_{1}\ld z_{2}.
\end{multline*}
Changing the order of integration gives
\begin{equation} \label{quad_var_expr}
\intrd{2}\phi(z_{1})\phi(z_{2})\sigma^{(r_N)}_{z_{1},z_{2}}(\wt{s})\ld z_{1}\ld z_{2} = \dual{\avg{\wt{s}}}{\left(\avg{\phi(1-\wt{s})}\right)^{2}} + \dual{1-\avg{\wt{s}}}{\left(\avg{\phi\,\wt{s}}\right)^{2}} .
\end{equation}
We are left with showing that the right-hand-side of \eqref{quad_var_expr} converges in probability to
\begin{equation*}
\dual{f_{s}}{(\phi(1-f_{s}))^{2}} + \dual{1-f_{s}}{(\phi f_{s})^{2}} = \dual{f_{s}(1-f_{s})}{\phi^{2}}.
\end{equation*}
To do this, we first justify that $\phi$ can be let out of the average, we use Lemma~\ref{lemma_bound_avg} to argue that we can replace $\wt{s}$ by $\ft{s}$, then the regularity of $\ft{}$ allows us to remove the averages and finally we know from Proposition~\ref{prop:convergence_centering_term} that $\ft{}$ converges to $f$. 
First note that
\begin{equation*}
\avg{\phi w}(x)-\phi(x)\avg{w}(x) = \frac{1}{V_{r}}\intbr{x}w(y)(\phi(y)-\phi(x))\ld y.
\end{equation*}
Since $0\leq w(y)\leq 1$ a.e., we have
\begin{align*}
\abs{\avg{\phi w}(x)-\phi(x)\avg{w}(x)} &\leq \frac{1}{V_{r}}\intbr{x}\abs{\phi(y)-\phi(x)}\ld y \leq \frac{1}{V_r}\intbr{x}\sum_{i=1}^{d}\norm[\infty]{\partial_i\phi}\abs{y-x}_i \ld y.
\end{align*}
Hence
\begin{equation*}
\norm[\infty]{\avg{\phi w}-\phi\avg{w}} \leq d\,r_{N}\max_{i}\norm[\infty]{\partial_{i}\phi}.
\end{equation*}
As a consequence,
\begin{align*}
\dual{1-\avg{w}}{(\avg{\phi w})^{2}}-\dual{1-\avg{w}}{\phi^{2}\avg{w}^{2}} &= \dual{1-\avg{w}}{(\avg{\phi w}-\phi\avg{w})(\avg{\phi w}+\phi\avg{w})} \\
\abs{\dual{1-\avg{w}}{(\avg{\phi w})^{2}}-\dual{1-\avg{w}}{\phi^{2}\avg{w}^{2}}} &\leq 2\norm[1]{\phi}\norm[\infty]{\avg{\phi w}-\phi\avg{w}} \\
&\leq 2d\norm[1]{\phi}r_{N}\max_{i}\norm[\infty]{\partial_{i}\phi}.
\end{align*}
By the same argument (replacing $w$ by $1-w$), we can also let $\phi$ out of the average in the first term on the right-hand-side of \eqref{quad_var_expr}, and the problem reduces to showing the convergence of
\begin{equation*}
\dual{\avg{\wt{s}}}{\phi^{2}(1-\avg{\wt{s}})^{2}} + \dual{1-\avg{\wt{s}}}{\phi^{2}\avg{\wt{s}}^{2}} = \dual{\avg{\wt{s}}(1-\avg{\wt{s}})}{\phi^{2}}.
\end{equation*}
We now see that it is enough to show
\begin{equation*}
\sup_{x\in\R^{d}}\E{\abs{\avg{\wt{s}}(x)-f_{s}(x)}} \cvgas{N} 0.
\end{equation*}
But, by the triangle inequality,
\begin{align*}
\E{\abs{\avg{\wt{s}}(x)-f_{s}(x)}} &\leq (\tau_{N}/\eta_{N})^{1/2}\E{\avg{\zt{s}}(x)^{2}}^{1/2} + \norm[\infty]{\avg{\ft{s}}-\ft{s}} + \norm[\infty]{\ft{s}-f_{s}} \\
&\leq \frac{(\tau_{N}/\eta_{N})^{1/2}}{V_{r}^{1/2}}K_{5}^{1/2} + \frac{d}{2}r_{N}^{2}K_{7} + K_{6}r_{N}^{2}.
\end{align*}
(We have used Lemma~\ref{lemma_bound_avg}, Proposition~\ref{prop:average} in Appendix~\ref{append:average} and Proposition~\ref{prop:convergence_centering_term}.) The right-hand-side converges to zero as $N\to\infty$ (due to assumption \eqref{parameters_regime_brownian}), providing the desired result. 
From all this we conclude
\begin{equation*}
\intrd{2}\phi(z_{1})\phi(z_{2})\sigma^{(r_N)}_{z_{1},z_{2}}(\wt{s})\ld z_{1}\ld z_{2} \cvgas[L^1]{N} \intrd\phi(x)^{2}f_{s}(x)(1-f_{s}(x))\ld x,
\end{equation*}
uniformly for $s\in[0,T]$, which gives us (ii).
\end{proof}

\subsection{Conclusion of the proof}\label{subsec:brownian:conclusion}

We are almost done. We have proved that the sequence of processes $\proc{\zt{}}{N\geq 1}$ is tight, and we need only characterise its potential limit points. Recall the following expression for $\dual{\zt}$ from \eqref{stochastic_integral_zn}:
\begin{equation*}
 \dual{\zt} = -(\tau_{N}/\eta_{N})^{1/2}\intg{0}{t}\dual{(\avg{\zt{s}})^{2}}{R_{2}(\avg{\wt{s}},\avg{\ft{s}})\avg{\varphi^{N}(s,t)}}\ld s + \intg{0}{t}\intrd\varphi^{N}(x,s,t)M^{N}(\ld x\ld s).
\end{equation*}
The first term converges to zero in $L^1$ from \eqref{bound_error_term}. Also,
\begin{equation*}
\intg{0}{t}\intrd\varphi^{N}(x,s,t)M^{N}(\ld x\ld s) - \intg{0}{t}\intrd\varphi(x,s,t)M^{N}(\ld x\ld s) \cvgas[L^{2}]{N} 0,
\end{equation*}
since, from Lemma~\ref{lemma:bound_qvar} and Lemma~\ref{lemma:convergence_varphi},
\begin{align*}
\E{\left(\intg{0}{t}\intrd(\varphi^{N}(x,s,t)-\varphi(x,s,t))M^{N}(\ld x\ld s)\right)^{2}} &\leq K_{4}\intg{0}{t}\norm[2]{\varphi^{N}(s,t)-\varphi(s,t)}^{2}\ld s \\
&\leq K_{1}^{2}TK_{4}r_{N}^{4}.
\end{align*}
For $\phi_1,\ldots,\phi_p$ in $\calS(\R^d)$, let $\varphi_1,\ldots,\varphi_p$ be the corresponding solutions of \eqref{definition_varphi} with $\phi=\phi_i$. 
Since we showed in Section \ref{subsec:brownian:convergence_M} that $M^{N}$ converges weakly to $M$, by \citep[Proposition 7.12]{walsh_introduction_1986}, for $t_1,\ldots,t_p\in[0,T]$,
\begin{multline*}
 \left(\intg{0}{t_1}\intrd\varphi_1(x,s,t_1)M^{N}(\ld x\ld s),\ldots,\intg{0}{t_k}\intrd\varphi_k(x,s,t_k)M^{N}(\ld x\ld s)\right) \\ \cvgas[d]{N} \left(\intg{0}{t_1}\intrd\varphi_1(x,s,t_1)M(\ld x\ld s),\ldots,\intg{0}{t_k}\intrd\varphi_k(x,s,t_k)M(\ld x\ld s)\right).
\end{multline*}
This uniquely characterises the potential limit points of $\proc{Z^{N}}{N\geq 1}$. 
By Theorem~\ref{thm:mitoma}, $\proc{Z^{N}}$ converges in distribution to a distribution-valued process $\proc{z}$ given by
\begin{equation}\label{limiting_stochastic_integral}
\dual{z_{t}} = \intg{0}{t}\intrd\varphi(x,s,t)M(\ld x\ld s),
\end{equation}
where $\varphi$ satisfies the backwards heat equation \eqref{definition_varphi} with terminal condition $\phi$,
\begin{equation*}
\left\lbrace
\begin{aligned}
& \partial_{s}\varphi(x,s,t) + \frac{1}{2}\Delta\varphi(x,s,t) - F'(f_{s}(x))\varphi(x,s,t) = 0, \\
& \varphi(x,t,t) = \phi(x).
\end{aligned}
\right.
\end{equation*}
It is an easy exercise to prove that $z_{t}$ satisfies
\begin{equation} \label{mild_solution}
\dual{z_{t}} = \intg{0}{t}\dual{z_{s}}{\frac{1}{2}\Delta\phi - F'(f_{s})\phi}\ld s + \intg{0}{t}\intrd\phi(x)M(\ld x\ld s).
\end{equation}
(See the proof of \citep[Theorem 5.2]{walsh_introduction_1986}.) In other words, $\proc{z}$ is the (mild) solution of \eqref{limiting_fluctuations} (recall that $M_{t}=\sqrt{f_{t}(1-f_{t})}\cdot W_{t}$) and Theorem~\ref{thm:clt_brownian} is proved.

\section{The stable case - proof of Theorem~\ref{thm:clt_stable}}\label{sec:proof_stable}

Turning to the proof of the central limit theorem in the stable case, we warn that its overall structure is the same as that in the Brownian case. 
Some steps need a different treatment however, and we explain those in more detail. 
Whenever the details of the argument are exactly the same as previously, we simply mention intermediate results without detailing their proof. To simplify our formulae, we use the following notation:
\begin{equation}\label{definition_lesssim}
a_n \lesssim b_n \Leftrightarrow \exists K>0 : \forall n\geq 1, \; 0 \leq a_n \leq K \,b_n.
\end{equation}
The specific constants can always be retrieved from Section~\ref{sec:proof_brownian} or from a trivial calculation. 
Also as in Section \ref{sec:proof_brownian} we set the constants $u$ and $(sV_1)/\alpha $ to $1$ in the martingale problem (M2) defined in Definition~\ref{formal_def_M2}.
Let us write (M2) as
\begin{equation*}
\ld w_t^N = \eta_N \left[ \L*{\delta_N}w_t^N - F^{(\delta_N)}(w_t^N)\right]\ld t + \tau_N^{1/2}\ld \mathbb M_t^N.
\end{equation*}
Setting $M_t^N(\phi) = \eta_N^{1/2}\mathbb M^N_{t/\eta_N}(\phi)$ and using the definition of $F^{(\delta_N)}$ in~\eqref{definition_F_delta}, we have, by the same argument as for \eqref{martingale_z},
\begin{multline}\label{good_mart_z_sr}
\dual{\zt} = \intg{0}{t}\dual{\zt{s}}{\L*{\delta_N}\phi-\alpha\intg{1}{\infty}\avg{F'(\avg{\ft{s}})\avg{\phi}}{\delta_N r}\frac{\ld r}{r^{\alpha+1}}}\ld s \\ - \left( \frac{\tau_N}{\eta_N}\right)^{\tfrac{1}{2}}\alpha\intg{0}{t}\intg{1}{\infty}\dual{(\avg{\zt{s}}{\delta_N r})^{2}}{R_{2}(\avg{\wt{s}},\avg{\ft{s}})\avg{\phi}{\delta_N r}}\frac{\ld r}{r^{\alpha+1}}\ld s + M^N_t(\phi),
\end{multline}
and the covariation measure of $M^N$ is given by
\begin{equation} \label{covariation_stable}
Q^N(\ld z_1\ld z_2\ld s) = (\sigma^{(\alpha,\delta_N)}_{z_1,z_2}(\wt{s})+ \abs{ z_1-z_2 }^{-\alpha} \bigO{\delta_N^\alpha})\ld z_1\ld z_2\ld s.
\end{equation}

\subsection{Time dependent test functions}\label{subsec:stable:time_dependent_tf}

Recall Proposition~\ref{prop:time_dependent_tf} and how we used it in the previous proof.
Define a time dependent test function $\varphi^N$ as the solution to the following.
\begin{equation} \label{defn_varphi_n_stable}
\left\lbrace
\begin{aligned}
& \partial_{s}\varphi^{N}(x,s,t) + \L*{\delta_N}\varphi^{N}(x,s,t) -\alpha\intg{1}{\infty} \avg{F'(\avg{\ft{s}})\avg{\varphi^{N}(s,t)}}(x,\delta_N r)\frac{\ld r}{r^{\alpha +1}} = 0 \\
& \varphi^{N}(x,t,t) = \phi(x).
\end{aligned}
\right.
\end{equation}
By Proposition~\ref{prop:time_dependent_tf}, we have
\begin{multline}\label{stochastic_integral_zn_sr}
\dual{\zt} = -\left(\frac{\tau_{N}}{\eta_{N}}\right)^{1/2}\alpha\intg{0}{t}\intg{1}{\infty}\dual{(\avg{\zt{s}}{\delta_N r})^{2}}{R_{2}(\avg{\wt{s}},\avg{\ft{s}})\avg{\varphi^{N}(s,t)}{\delta_N r}}\frac{\ld r}{r^{\alpha+1}}\ld s \\ + \intg{0}{t}\intrd\varphi^{N}(x,s,t)M^{N}(\ld x\ld s).
\end{multline}
We are thus left with finding a suitable way to bound the first term above and showing the convergence of the stochastic integral against $M^N$. 
The convergence of the martingale measure $M^N$ is going to involve slightly different calculations compared to the previous case as the limiting noise is not a space-time white noise. 
The convergence of $\varphi^N$, however, is proved in a similar way to before. 
Define $\varphi$ as the solution to the following.
\begin{equation} \label{definition_varphi_sr}
\left\lbrace
\begin{aligned}
& \partial_{s}\varphi(x,s,t) + \D\varphi(x,s,t) - F'(f_{s}(x))\varphi(x,s,t) = 0 \\
& \varphi(x,t,t) = \phi(x).
\end{aligned}
\right.
\end{equation}
The following lemma, whose proof is given in Appendix \ref{append:test_functions}, provides the convergence of $\varphi^N$ to $\varphi$.
\begin{lemma}\label{lemma:stable:convergence_varphi}
For $T>0$ and for $q\in\braced{1,\infty},$
\begin{equation*}
\sup_{0\leq s\leq t\leq T}\norm[q]{\varphi^N(s,t)-\varphi(s,t)} \lesssim \delta_N^{\alpha\wedge(2-\alpha)}.
\end{equation*}
In addition, for $0<\abs{\beta} \leq 2$,
\begin{align*}
\sup_{0\leq s\leq t\leq T}\norm[q]{\varphi^N(s,t)} \lesssim \norm[q]{\phi} && and && \sup_{0\leq s\leq t\leq T}\norm[q]{\partial_\beta\varphi^N(s,t)} \lesssim 1.
\end{align*}
\end{lemma}

\subsection{Regularity estimate}\label{subsec:stable:regularity_estimate}

Let us first state the following $L^2$ bound for the stochastic integral.

\begin{lemma} \label{lem_bound_qvar_sr}
For $0\leq t\leq T$ and $\alpha<d$, suppose $\phi_{t}:\R^{d}\to\R$ is in $L^{1,\infty}(\R^{d})$; then
\begin{align*}
\E{\left(\intg{0}{t}\intrd\phi_{s}(x)M^{N}(\ld x\ld s)\right)^{2}} &\lesssim \intg{0}{t}\intrd{2}\abs{\phi_{s}(z_{1})}\abs{\phi_{s}(z_{2})}(\delta_N\vee \tfrac{|z_1-z_2|}{2})^{-\alpha}\ld z_{1}\ld z_{2}\ld s\\
&\lesssim \intg{0}{t}\| \phi _s \|_1 (\| \phi _s \|_\infty +\|\phi_s \|_1)\ld s.
\end{align*}
\end{lemma}
The proof uses the following lemma, which is proved in Appendix \ref{append:test_functions}.
\begin{lemma} \label{lem:alpha_int_bound}
For $\alpha<d$, then for $f,g\in L^{1,\infty}(\R ^d)$
$$ \abs{\intrd{2}f(z_1)g(z_2)\abs{z_1-z_2}^{-\alpha}\ld z_1 \ld z_2}\leq \norm[1]{f}(\frac{dV_1}{d-\alpha}\norm[\infty]{g}+\norm[1]{g}).$$
\end{lemma}
\begin{proof}[Proof of Lemma~\ref{lem_bound_qvar_sr}]
From the expression for the covariation measure in \eqref{covariation_stable},
\begin{align*}
\E{\left(\intg{0}{t}\intrd\phi_{s}(x)M^{N}(\ld x\ld s)\right)^{2}} &= \E{\intg{0}{t}\intrd{2}\phi_{s}(z_{1})\phi_{s}(z_{2})\sigma_{z_{1},z_{2}}^{\alpha, \delta_N}(\wt{s})\ld z_{1}\ld z_{2}\ld s}\\
&\hspace{2cm}+\bigO{\delta_N^\alpha}\intg{0}{t}\intrd{2}\phi_{s}(z_{1})\phi_{s}(z_{2})|z_1-z_2|^{-\alpha}\ld z_{1}\ld z_{2}\ld s.
\end{align*}
But, by the definition of $\sigma^{(\alpha,\delta)}$ in \eqref{definition_sigma_stable},
\begin{align*}
\abs{\sigma_{z_{1},z_{2}}^{\alpha, \delta}(w)} &\leq \int_{\delta\vee \tfrac{|z_1-z_2|}{2}}^\infty V_r(z_1,z_2)\frac{\ld r}{r^{d+\alpha +1}} \\
&\lesssim V_1\int_{\delta\vee \tfrac{|z_1-z_2|}{2}}^\infty \frac{\ld r}{r^{\alpha +1}} \\
&\lesssim \left(\delta \vee \frac{\abs{z_1-z_2}}{2}\right)^{-\alpha}.
\end{align*}
The second inequality is obtained from the first one and Lemma~\ref{lem:alpha_int_bound}.
\end{proof}

Let $\mathcal{G}^{(\alpha)}$ (resp. $\G*$) denote the fundamental solution to the fractional heat equation with the operator $\D$ (resp. the fractional heat equation with the truncated operator $\L*$). Then the centering term $\ft{}$ as defined in \eqref{definition_centering_term_stable} can be written as
\begin{equation}\label{stable_centering_term_green_function}
\ft(x) = \G*{\delta_N}_t\ast w_0(x) - \intg{0}{t}\G*{\delta_N}_{t-s}\ast F^{(\delta_N)}(\ft{s})(x)\ld s.
\end{equation}
Likewise, using the definition of $f_t$ in \eqref{definition_limit_stable},
\begin{equation*}
f_t(x) = \mathcal{G}^{(\alpha)}_t\ast w_0(x) - \intg{0}{t}\mathcal{G}^{(\alpha)}_{t-s}\ast F(f_s)(x)\ld s.
\end{equation*}

We can now prove the following counterpart of the regularity estimate (Lemma~\ref{lemma_bound_avg}), which allows us to bound the quadratic error term in \eqref{stochastic_integral_zn_sr}.
\begin{lemma} \label{lemma_bound_avg_sr}
Fix $T>0$; for $0\leq t\leq T$,
\begin{equation*}
\intg{1}{\infty}\sup_{x\in\R^d} \E{\avg{\zt}(x,\delta_N r)^{2}}\frac{\ld r}{r^{\alpha +1}} \lesssim \frac{1}{\delta_N^\alpha}.
\end{equation*}
\end{lemma}

\begin{proof}
From \eqref{good_mart_z_sr} and the definition of $R_1$ in~\eqref{taylor_F_1}, we have
\begin{equation*}
\dual{\zt} = \intg{0}{t}\dual{\zt{s}}{\L*{\delta_N}\phi-\alpha\intg{1}{\infty}\avg{R_1(\avg{\wt{s}},\avg{\ft{s}})\avg{\phi}}{\delta_N r}\frac{\ld r}{r^{\alpha+1}}}\ld s + \intg{0}{t}\intrd\phi(y)M^N(\ld y\ld s).
\end{equation*}
Using Proposition~\ref{prop:time_dependent_tf} with $\varphi(x,s,t) = \G*{\delta_N}_{t-s}\ast\phi(x)$ yields
\begin{multline*}
\dual{\zt} = -\alpha\intg{0}{t}\intg{1}{\infty}\dual{\G*{\delta_N}_{t-s}\ast\avg{\avg{\zt{s}}R_1(\avg{\wt{s}},\avg{\ft{s}})}{\delta_N r}}{\phi}\frac{\ld r}{r^{\alpha+1}}\ld s \\ + \intg{0}{t}\intrd\G*{\delta_N}_{t-s}\ast\phi(y)M^N(\ld y\ld s).
\end{multline*}
Now we take $\phi(y)=\frac{1}{V_{R}}\1{\abs{x-y}<R}$ to obtain
\begin{multline*}
\avg{\zt}(x,R) = -\alpha\intg{0}{t}\intg{1}{\infty}\left(\avg{\G*{\delta_N}_{t-s}}{R}\ast\avg{\avg{\zt{s}}R_1(\avg{\wt{s}},\avg{\ft{s}})}{\delta_N r}\right)(x)\frac{\ld r}{r^{\alpha+1}}\ld s \\ + \intg{0}{t}\intrd\avg{\G*{\delta_N}_{t-s}}(x-y,R)M^N(\ld y\ld s).
\end{multline*}
Repeating the same steps as in the proof of Lemma~\ref{lemma_bound_avg} and using Jensen's inequality, we get
\begin{multline*}
\left(\avg{\zt}(x,R)\right)^2 \lesssim \alpha\intg{0}{t}\intg{1}{\infty}\intrd\dbavg{\G*{\delta}_{t-s}}(x-y,R,\delta_N r)\left(\avg{\zt{s}}(y,\delta_N r)\right)^2\ld y\frac{\ld r}{r^{\alpha+1}}\ld s \\ + \left(\intg{0}{t}\intrd\avg{\G*{\delta}_{t-s}}(x-y,R)M^N(\ld y\ld s)\right)^2.
\end{multline*}
Using the first inequality of Lemma~\ref{lem_bound_qvar_sr} and bounding $\left( \delta_N\vee \frac{\abs{z_1-z_2}}{2} \right)^{-\alpha}$ by $\delta_N^{-\alpha}$, we have, for $0\leq t\leq T$,
\begin{align*}
\E{\left(\intg{0}{t}\intrd\avg{\G*{\delta}_{t-s}}(x-y,R)M^{N}(\ld y\ld s)\right)^{2}} &\lesssim \delta_N^{-\alpha}\intg{0}{t}\left(\intrd\avg{\G*{\delta}_{t-s}}(x-y,R)\ld y\right)^2\ld s \\
&\lesssim \delta_N^{-\alpha}.
\end{align*}
As a result
\begin{equation*}
\E{\left(\avg{\zt}(x,R)\right)^2} \lesssim \intg{0}{t}\intg{1}{\infty}\intrd\dbavg{\G*{\delta}_{t-s}}(x-y,R,\delta_N r)\E{\left(\avg{\zt{s}}(y,\delta_N r)\right)^2}\ld y\frac{\ld r}{r^{\alpha+1}}\ld s + \delta_N^{-\alpha}.
\end{equation*}
Taking the supremum of $\E{\avg{\zt{s}}(y,\delta_N r)^2}$ over $y$ inside the integral on the right-hand-side, the function $\avg{\G*{\delta}}$ integrates to 1, yielding
\begin{equation*}
\sup_{x\in\R^d}\E{\left(\avg{\zt}(x,R)\right)^2} \lesssim \intg{0}{t}\intg{1}{\infty}\sup_{x\in\R^d}\E{\left(\avg{\zt{s}}(x,\delta_N r)\right)^2}\frac{\ld r}{r^{\alpha+1}}\ld s + \delta_N^{-\alpha}.
\end{equation*}
Integrating over $R$, we get
\begin{equation*}
\intg{1}{\infty}\sup_{x\in\R^d}\E{\left(\avg{\zt}(x,\delta_N r)\right)^2}\frac{\ld r}{r^{\alpha+1}} \lesssim \intg{0}{t}\intg{1}{\infty}\sup_{x\in\R^d}\E{\left(\avg{\zt{s}}(x,\delta_N r)\right)^2}\frac{\ld r}{r^{\alpha+1}}\ld s + \delta_N^{-\alpha}.
\end{equation*}
Hence, by Gronwall's inequality, for $0\leq t\leq T$,
\begin{equation*}
\intg{1}{\infty}\sup_{x\in\R^d}\E{\left(\avg{\zt}(x,\delta_N r)\right)^2}\frac{\ld r}{r^{\alpha+1}} \lesssim \delta_N^{-\alpha}.
\end{equation*}
\end{proof}

\subsection{Convergence to the deterministic limit}\label{subsec:stable:convergence_w}

The following result is proved in Appendix~\ref{append:centering_term}.
\begin{proposition}\label{prop:stable:convergence_centering_term}
For $T>0$,
\begin{equation*}
\sup_{0\leq t\leq T}\norm[\infty]{\ft-f_t} \lesssim \delta_N^{\alpha\wedge(2-\alpha)},
\end{equation*}
and for $0\leq \abs{\beta} \leq 2$,
\begin{equation*}
\sup_{0\leq t\leq T}\norm[\infty]{\partial_\beta \ft} \lesssim 1.
\end{equation*}
\end{proposition}

The convergence of $\wt$ to $f_t$ in $L^1$ will follow from the next lemma. 

\begin{lemma} \label{bound_sup_z_sr}
For any function $\phi$ satisfying $\norm[q]{\phi}\leq 1$ and $\max_{\abs{\beta}=2}\norm[q]{\partial_{\beta}\phi}\leq 1$ for $q\in\braced{1,\infty}$,
\begin{equation*}
\E{\sup_{0\leq t\leq T}\abs{\dual{\zt}}} \lesssim 1.
\end{equation*}
\end{lemma}
Indeed, by the same argument as in Section~\ref{subsec:brownian:convergence_w}, choosing a separating family $\proc{\phi_n}{n\geq 1}$ of compactly supported smooth functions satisfying this condition and using the corresponding metric $d$ on $\Xi$, one has
\begin{align*}
\E{\sup_{0\leq t\leq T}d(\wt,\ft)} & \leq \sum_{n\geq 1}\frac{1}{2^n}\left\lbrace (\tau_N/\eta_N)^{1/2}\E{\sup _{0\leq t\leq T}\abs{\dual{\zt}{\phi_n}}} + \sup _{0\leq t\leq T}\norm[\infty]{\ft-f_t}\norm[1]{\phi_n} \right\rbrace \\
& \lesssim (\tau_N/\eta_N)^{1/2} + \delta_N^{\alpha\wedge(2-\alpha)},
\end{align*}
by Proposition~\ref{prop:stable:convergence_centering_term} and Lemma~\ref{bound_sup_z_sr}.
From \eqref{parameters_regime_stable}, it can be seen that the leading term on the right-hand-side is $\delta_N^{\alpha\wedge(2-\alpha)}$, which goes to zero as $N\to\infty$, yielding the convergence of $\wt$. The following lemma is needed for the proof of Lemma~\ref{bound_sup_z_sr} and is proved in the same manner as \eqref{L1_bound_z} in Section~\ref{subsec:brownian:convergence_w}.
\begin{lemma} \label{lem_L1_bound_z_sr}
For $\phi\in L^{1,\infty}(\R^d)$ and $t\in[0,T]$,
\begin{equation*}
\E{\abs{\dual{\zt}}} \lesssim \norm[1]{\phi}+\norm[\infty]{\phi}.
\end{equation*}
\end{lemma}
\begin{proof}
Taking expectations on both sides of \eqref{stochastic_integral_zn_sr},
\begin{multline} \label{eq_stable_conv_determ}
\E{\abs{\dual{\zt}}} \lesssim (\tau/\eta)^{1/2}\intg{0}{t}\intg{1}{\infty}\dual{\E{\left( \avg{\zt{s}}{\delta_N r} \right)^2}}{\avg{\abs{\varphi^N(s,t)}}{\delta_N r}}\frac{\ld r}{r^{\alpha+1}}\ld s \\ + \E{\left( \intg{0}{t}\intrd\varphi^N(x,s,t)M^N(\ld x\ld s) \right)^2}^{1/2}.
\end{multline}
In the first integral, we have
\begin{equation*}
\dual{\E{\left( \avg{\zt{s}}{\delta_N r} \right)^2}}{\avg{\abs{\varphi^N(s,t)}}{\delta_N r}} \leq \norm[1]{\varphi^N(s,t)}\sup _{x\in\R^d}\E{\left( \avg{\zt{s}}(x,\delta_N r) \right)^2}.
\end{equation*}
Hence, applying Lemma~\ref{lemma_bound_avg_sr} to the first term and Lemma~\ref{lem_bound_qvar_sr} to the second term on the right-hand-side of \eqref{eq_stable_conv_determ} yields
\begin{align*}
\E{\abs{\dual{\zt}}} & \lesssim \frac{(\tau/\eta)^{1/2}}{\delta_N^{\alpha}}\intg{0}{t}\norm[1]{\varphi^N(s,t)}\ld s + \left( \intg{0}{t}\norm[1]{\varphi^N(s,t)}(\norm[\infty]{\varphi^N(s,t)} + \norm[1]{\varphi^N(s,t)})\ld s \right)^{1/2} \\
& \lesssim \frac{(\tau/\eta)^{1/2}}{\delta_N^{\alpha}}\norm[1]{\phi} + \left( \norm[1]{\phi}(\norm[\infty]{\phi}+\norm[1]{\phi}) \right)^{1/2} \\
& \lesssim \norm[1]{\phi} + \norm[\infty]{\phi}.
\end{align*}
We have used the fact that (by Lemma~\ref{lemma:stable:convergence_varphi}) $\norm[q]{\varphi^N(s,t)} \lesssim \norm[q]{\phi}$ to pass from the first line to the second. The third line follows since $\tau_N/\eta_N=\littleO{\delta_N^{2\alpha}}$ by~\eqref{parameters_regime_stable}.
\end{proof}

\begin{proof}[Proof of Lemma~\ref{bound_sup_z_sr}]
The proof of Lemma~\ref{bound_sup_z_sr} is similar to the proof of Lemma~\ref{lem_sup_zt_bound}. Setting $\psi_s = \L*{\delta_N}\phi-\alpha\intg{1}{\infty}\avg{F'(\avg{\ft{s}})\avg{\phi}}{\delta_N r}\tfrac{\ld r}{r^{\alpha+1}}$ and using \eqref{good_mart_z_sr},
\begin{multline*}
\sup _{0\leq t\leq T}\abs{\dual{\zt}} \lesssim \intg{0}{T}\abs{\dual{\zt{s}}{\psi_s}}\ld s + (\tau/\eta)^{1/2}\norm[\infty]{F''}\intg{0}{T}\intg{1}{\infty}\dual{\left( \avg{\zt{s}}{\delta_N r} \right)^2}{\avg{\abs{\phi}}{\delta_N r}}\frac{\ld r}{r^{\alpha+1}}\ld s \\ + \sup _{0\leq t \leq T}\abs{M^N_t(\phi)}.
\end{multline*}
Taking the expectation on both sides, Lemma~\ref{lem_L1_bound_z_sr} can be used in the first term, and Lemma~\ref{lemma_bound_avg_sr} in the second one, to yield
\begin{equation*}
\E{\sup _{0\leq t\leq T}\abs{\dual{\zt}}} \lesssim \intg{0}{T}(\norm[1]{\psi_s} + \norm[\infty]{\psi_s})\ld s + \frac{(\tau/\eta)^{1/2}}{\delta_N^\alpha}\norm[\infty]{F''}\norm[1]{\phi} + \E{\sup _{0\leq t \leq T}\abs{M^N_t(\phi)}^2}^{1/2}.
\end{equation*}
But $\norm[q]{\psi_s} \lesssim \norm[q]{\L*\phi} + \norm[\infty]{F'}\norm[q]{\phi}$ and, by Proposition~\ref{prop:stable_average}.i in Appendix~\ref{append:average}, $\norm[q]{\L*\phi} \lesssim \norm[q]{\phi} + \max_{\abs{\beta}=2}\norm[q]{\partial_\beta\phi}$. 
In addition, by Doob's inequality, and using Lemma~\ref{lem_bound_qvar_sr},
\begin{align*}
\E{\sup _{0\leq t\leq T}\abs{M^N_t(\phi)}^2} &\lesssim \E{M^N_T(\phi)^2}\\
& \lesssim \norm[1]{\phi}(\norm[1]{\phi} + \norm[\infty]{\phi}).
\end{align*}
As a result, if $\norm[q]{\phi}\leq 1$ and $\max_{\abs{\beta}=2}\norm[q]{\partial_\beta\phi}\leq 1$ for $q\in\braced{1,\infty}$,
\begin{equation*}
\E{\sup _{0\leq t\leq T}\abs{\dual{\zt}}} \lesssim 1.
\end{equation*}
\end{proof}

\subsection{Tightness}\label{subsec:stable:tightness}

The overall argument for the tightness of the sequence $\proc{Z^N}$ is the same as in Section~\ref{subsec:brownian:tightness}.

\begin{proposition} \label{aldous_criterion_zn_sr}
For any $\phi\in\calS(\R^d)$ and for any sequence $\proc{T_N,\rho_N}{N\geq 1}$ such that $T_N$ is a stopping time with values in $[0,T]$ for every $N\geq 1$ and $\rho_N\downarrow 0$ as $N\to\infty$,
\begin{equation} \label{aldous_convergence}
\dual{\zt{T_{N}+\rho_{N}}}-\dual{\zt{T_{N}}} \cvgas[P]{N} 0.
\end{equation}
\end{proposition}
Tightness of $\proc{Z^N}{N\geq 1}$ in $\sko{\calS'(\R^d)}$ then follows from Aldous' criterion \cite{aldous_stopping_1978} and Mitoma's theorem \citep[Theorem 6.13]{walsh_introduction_1986}. 

\begin{proof}[Proof of Proposition~\ref{aldous_criterion_zn_sr}]
Extend $\varphi^N$ to $\R^d \times [0,T]^2$ as in \eqref{varphi_extension}; we need estimates on $\varphi^N$ as in Lemmas~\ref{lemma:continuity_varphi} and~\ref{lemma_varphi_uniform_bound}. 
The proof of the following lemma is in Appendix \ref{append:test_functions}.
\begin{lemma}\label{lemma:continuity_unif_bound_varphi_sr}
For $T>0$, $q\in\braced{1,\infty}$ and for all $s,t,t'\in[0,T]$,
\begin{equation*}
\norm[q]{\varphi^N(s,t')-\varphi^N(s,t)} \lesssim \abs{t-t'}.
\end{equation*}
In addition, for all $s\in[0,T]$,
\begin{equation*}
\norm[1]{\sup_{t\in[s,T]}\abs{\varphi^N(s,t)}} \lesssim 1.
\end{equation*}
\end{lemma}
We shall only detail how the quadratic part of \eqref{stochastic_integral_zn_sr} can be bounded using Lemma~\ref{lemma_bound_avg_sr}, and refer to Section~\ref{subsec:brownian:tightness} for the rest of the proof of Proposition~\ref{aldous_criterion_zn_sr}. 
For $T_N$ a stopping time with values in $[0,T]$, write
\begin{multline*}
\abs{\intg{0}{T_N}\intg{1}{\infty}\dual{(\avg{\zt{s}}{\delta_N r})^{2}}{R_{2}(\avg{\wt{s}},\avg{\ft{s}})\avg{\varphi^{N}(s,T_N)}{\delta_N r}}\frac{\ld r}{r^{\alpha+1}}\ld s} \\ \lesssim \norm[\infty]{F''}\intg{0}{T}\intg{1}{\infty}\dual{(\avg{\zt{s}}{\delta_N r})^{2}}{\avg{\sup_{t\in[s,T]}\abs{\varphi^{N}(s,t)}}{\delta_N r}}\frac{\ld r}{r^{\alpha+1}}\ld s.
\end{multline*}
Taking the expectation on both sides and the supremum inside the spatial integral against $\varphi^N$, we get
\begin{multline*}
\E{\abs{\intg{0}{T_N}\intg{1}{\infty}\dual{(\avg{\zt{s}}{\delta_N r})^{2}}{R_{2}(\avg{\wt{s}},\avg{\ft{s}})\avg{\varphi^{N}(s,T_N)}{\delta_N r}}\frac{\ld r}{r^{\alpha+1}}\ld s}} \\ 
\begin{aligned}
&\lesssim \norm[\infty]{F''} \intg{0}{T}\intg{1}{\infty}\sup_{x\in\R^d}\E{\avg{\zt{s}}(x,\delta_N r)^{2}}\norm[1]{\sup_{t\in[s,T]}\abs{\varphi^{N}(s,t)}}\frac{\ld r}{r^{\alpha+1}}\ld s \\
&\lesssim \delta_N^{-\alpha},
\end{aligned}
\end{multline*}
by Lemma~\ref{lemma_bound_avg_sr} and Lemma~\ref{lemma:continuity_unif_bound_varphi_sr}. The other terms in \eqref{aldous_convergence} are bounded as in the proof of Proposition~\ref{tightness_prop} in Section~\ref{subsec:brownian:tightness}, using Lemmas~\ref{lemma:continuity_unif_bound_varphi_sr} and~\ref{lem_bound_qvar_sr}.
\end{proof}

\subsection{Convergence of the martingale measure $M^N$}\label{subsec:stable:convergence_M}

The convergence of $M^N$ relies on applying Theorem~\ref{thm:convergence_gaussian_martingale} to vectors of the form $\left( M^N_t(\phi_1),\ldots,\right.$ $\left. M^N_t(\phi_p)\right)_{t\geq 0}$, although the details differ from the proof in the Brownian case (in Section~\ref{subsec:brownian:convergence_M}). 
Indeed, $M^N$ no longer converges to a stochastic integral against a space-time white noise, but to $W^{\alpha}$, a coloured Gaussian noise such that
\begin{equation*}
\qvar{W^{\alpha}(\phi)} = \intg{0}{t}\intrd{2}\phi(z_1)\phi(z_2)\sigma^{\alpha}_{z_1,z_2}(f_s)\ld z_1\ld z_2\ld s
\end{equation*}
with
\begin{equation*}
\sigma^\alpha_{z_1,z_2}(f) = \intg{\tfrac{\abs{z_1-z_2}}{2}}{\infty}\frac{\ld r}{r^{d+\alpha+1}}\int_{B(z_1,r)\cap B(z_2,r)}\left[ \avg{f}(x,r)(1-f(z_1)-f(z_2)) + f(z_1)f(z_2) \right]\ld x.
\end{equation*}
Hence the weak convergence of $M^N$ to $W^{\alpha}$ in $\sko{\calS'(\R^{d})}$ will follow (as in Section \ref{subsec:brownian:convergence_M}) from the following lemma.
\begin{lemma} \label{lemma:stable:convergence_MN}
For any $\phi\in\calS(\R^d)$,
\begin{enumerate}[label=\roman*)]
\item For all $t\geq 0$, $\abs{\Delta M^N_t(\phi)}\lesssim 1$, and $\sup_{0\leq t\leq T}\abs{\Delta M^N_t(\phi)} \cvgas[P]{N} 0$.
\item For each $t\in[0,T]$, $\qvar{M^N(\phi)} \cvgas[P]{N} \qvar{W^{\alpha}(\phi)}$.
\end{enumerate}
\end{lemma}

\begin{proof}
The proof of the first part is the same as for Lemma \ref{lemma:convergence_MN}:
\begin{equation*}
\sup _{t\geq 0} \abs{\Delta M^N_t(\phi)} \leq \alpha_{N}(\eta_N/\tau_N)^{1/2}\norm[1]{\phi},
\end{equation*}
which tends to zero since $\alpha_{N}^{2}=\littleO{\tau_N/\eta_N}$.	For the second part of the statement, we first show that
\begin{equation}\label{sigma_wf_L1conv}
\abs{\intrd{2}\phi(z_1)\phi(z_2)\sigma_{z_1,z_2}^{(\alpha, \delta_N)}(\wt{s})\ld z_1\ld z_2 - \intrd{2}\phi(z_1)\phi(z_2)\sigma_{z_1,z_2}^{(\alpha, \delta_N)}(\ft{s})\ld z_1\ld z_2}\cvgas[L^1]{N} 0.%
\end{equation}
We have from the definition of $\sigma_{z_1,z_2}^{(\alpha, \delta_N)}$ in \eqref{definition_sigma_stable} that
\begin{multline*}
\sigma_{z_1,z_2}^{(\alpha, \delta_N)}(w) = \intg{\delta_N\vee \frac{|z_1-z_2|}{2}}{\infty} \Big\lbrace (1-w(z_1)-w(z_2))\int_{B(z_1,r)\cap B(z_2,r)}\avg{w}(x,r)\ld x \\ + V_r(z_1,z_2)w(z_1)w(z_2) \Big\rbrace \frac{\ld r}{r^{d+\alpha +1}}.
\end{multline*}
Subtracting the corresponding expressions with $\wt{s}$ and $\ft{s}$ and reordering terms, we write
\begin{multline}\label{difference_sigma}
\sigma_{z_1,z_2}^{(\alpha, \delta_N)}(\wt{s})-\sigma_{z_1,z_2}^{(\alpha, \delta_N)}(\ft{s}) \\ = \intg{\delta_N\vee \frac{|z_1-z_2|}{2}}{\infty} \Big\lbrace (1-\wt{s}(z_1)-\wt{s}(z_2))\int_{B(z_1,r)\cap B(z_2,r)}\left( \avg{\wt{s}}(x,r)-\avg{\ft{s}}(x,r) \right)\ld x \\ + (\ft{s}(z_1)-\wt{s}(z_1)+\ft{s}(z_2)-\wt{s}(z_2))\int_{B(z_1,r)\cap B(z_2,r)}\avg{\ft{s}}(x,r)\ld x \\ + V_r(z_1,z_2)\left(\wt{s}(z_1)(\wt{s}(z_2)-\ft{s}(z_2))+\ft{s}(z_2)(\wt{s}(z_1)-\ft{s}(z_1))\right) \Big\rbrace \frac{\ld r}{r^{d+\alpha +1}}.
\end{multline}
We shall deal with the terms from each of the three lines separately, so let us call them $A(z_1,z_2)$, $B(z_1,z_2)$ and $C(z_1,z_2)$ (they are in fact defined for a.e. $z_1$ and $z_2$, and so is all that follows, but this is not a problem since what we really show is \eqref{sigma_wf_L1conv}). 
For the first term write
\begin{align*}
\E{\abs{A(z_1,z_2)}} &\leq (\tau_N/\eta_N)^{1/2}\intg{\delta_N\vee \frac{|z_1-z_2|}{2}}{\infty} \int_{B(z_1,r)\cap B(z_2,r)}\E{\abs{\avg{\zt{s}}(x,r)}}\ld x  \frac{\ld r}{r^{d+\alpha +1}} \\
& \leq (\tau_N/\eta_N)^{1/2}\intg{\delta_N\vee \frac{|z_1-z_2|}{2}}{\infty} V_r(z_1,z_2)\sup _{x\in\R^d} \E{\abs{\avg{\zt{s}}(x,r)}}  \frac{\ld r}{r^{d+\alpha +1}} \\
& \leq (\tau_N/\eta_N)^{1/2}V_1\intg{\delta_N}{\infty} \1{2r>\abs{z_1-z_2}}\sup _{x\in\R^d} \E{\abs{\avg{\zt{s}}(x,r)}^2}^{1/2}  \frac{\ld r}{r^{\alpha +1}} \\
& \leq (\tau_N/\eta_N)^{1/2}\frac{V_1}{\alpha^{1/2}}\left( \delta_N\vee\frac{\abs{z_1-z_2}}{2} \right)^{-\alpha/2}\left(\intg{\delta_N}{\infty} \sup _{x\in\R^d} \E{\abs{\avg{\zt{s}}(x,r)}^2} \frac{\ld r}{r^{\alpha +1}}\right)^{1/2}.
\end{align*}
(We have used the Cauchy-Schwartz inequality in the last line.) 
In addition, by Lemma~\ref{lemma_bound_avg_sr},
\begin{align*}
\intg{\delta_N}{\infty} \sup _{x\in\R^d} \E{\abs{\avg{\zt{s}}(x,r)}^2} \frac{\ld r}{r^{\alpha +1}} &=  \delta_N^{-\alpha}\intg{1}{\infty} \sup _{x\in\R^d} \E{\abs{\avg{\zt{s}}(x,\delta_N r)}^2} \frac{\ld r}{r^{\alpha +1}} \\
& \lesssim \delta_N^{-2\alpha}.
\end{align*}
Hence
\begin{equation*}
\E{\abs{A(z_1,z_2)}} \lesssim (\tau_N/\eta_N)^{1/2}\delta_N^{-\alpha}\abs{z_1-z_2}^{-\alpha/2},
\end{equation*}
and, using Lemma~\ref{lem:alpha_int_bound} and \eqref{parameters_regime_stable},
\begin{equation*}
\intrd{2}\phi(z_1)\phi(z_2)A(z_1,z_2)\ld z_1\ld z_2 \cvgas[L^1]{N} 0.
\end{equation*}
For the second term, by symmetry,
\begin{equation*}
\abs{\intrd{2}\phi(z_1)\phi(z_2)B(z_1,z_2)\ld z_1\ld z_2} \leq 2(\tau_N/\eta_N)^{1/2}\intrd\abs{\phi(z_2)}\abs{\dual{\zt{s}}{\psi^N_{z_2}}}\ld z_2,
\end{equation*}
where
\begin{equation*}
\psi^N_{z_2}(z_1)=\phi(z_1)\intg{\delta_N\vee \frac{\abs{z_1-z_2}}{2}}{\infty} \int_{B(z_1,r)\cap B(z_2,r)}\avg{\ft{s}}(x,r)\ld x\frac{\ld r}{r^{d+\alpha+1}}.
\end{equation*}
In particular, by Proposition \ref{prop:stable:convergence_centering_term} $\norm[q]{\psi^N_{z_2}}\lesssim \delta_N^{-\alpha}\norm[q]{\phi}$ for $q\in\braced{1,\infty}$ and, since $\psi^N_{z_2}$ is deterministic, by Lemma~\ref{lem_L1_bound_z_sr}
\begin{equation*}
\E{\abs{\intrd{2}\phi(z_1)\phi(z_2)B(z_1,z_2)\ld z_1\ld z_2}} \lesssim (\tau_N/\eta_N)^{1/2}\delta_N^{-\alpha}\norm[1]{\phi}(\norm[1]{\phi}+\norm[\infty]{\phi}).
\end{equation*}
Hence, by \eqref{parameters_regime_stable},
\begin{equation*}
\intrd{2}\phi(z_1)\phi(z_2)B(z_1,z_2)\ld z_1\ld z_2 \cvgas[L^1]{N} 0.
\end{equation*}
The third term is controlled in a similar way, this time setting
\begin{equation*}
\psi^N_{z_2}(z_1) = \phi(z_1)\intg{\delta_N\vee \frac{\abs{z_1-z_2}}{2}}{\infty} V_r(z_1,z_2) \frac{\ld r}{r^{d+\alpha+1}},
\end{equation*}
which satisfies the same inequalities as the previous $\psi^N_{z_2}$ and using the bound on $\norm[\infty]{\ft{s}}$ from Proposition~\ref{prop:stable:convergence_centering_term}. As a result we have proved \eqref{sigma_wf_L1conv}. Now write
\begin{align*}
\abs{\sigma^{(\alpha,\delta_N)}_{z_1,z_2}(\ft{s})-\sigma^{\alpha}_{z_1,z_2}(\ft{s})}&\lesssim \1{\abs{z_1-z_2}\leq 2\delta_N}\intg{\frac{\abs{z_1-z_2}}{2}}{\delta_N}V_r(z_1,z_2)\frac{\ld r}{r^{d+\alpha+1}}\\
&\lesssim \1{\abs{z_1-z_2}\leq 2\delta_N}\abs{z_1-z_2}^{-\alpha}.
\end{align*}
Hence 
\begin{multline*}
\abs{\intrd{2}\phi(z_1)\phi(z_2)(\sigma^{(\alpha,\delta_N)}_{z_1,z_2}(\ft{s})-\sigma^{\alpha}_{z_1,z_2}(\ft{s}))\ld z_1 \ld z_2 } \\
\begin{aligned} 
&\lesssim  \intrd{2}\abs{\phi(z_1)}\abs{\phi(z_2)}\1{\abs{z_1-z_2}\leq 2\delta_N}\abs{z_1-z_2}^{-\alpha}\ld z_1 \ld z_2 \\
&\lesssim \norm[\infty]{\phi}\intrd \abs{\phi(z_1)} \intg{0}{2\delta_N}r^{-\alpha+d-1}\ld r \ld z_1 \\
&\lesssim \norm[\infty]{\phi}\norm[1]{\phi}\delta_N^{d-\alpha}\cvgas{N}0.
\end{aligned}
\end{multline*}
Finally, replacing $\wt{s}$ by $f_s$ and $\sigma^{(\alpha,\delta_N)}$ by $\sigma^{\alpha}$ in \eqref{difference_sigma}, one writes
\begin{align*}
\abs{\sigma^{\alpha}_{z_1,z_2}(\ft{s})-\sigma^{\alpha}_{z_1,z_2}(f_s)} &\lesssim\norm[\infty]{\ft{s}-f_s}\intg{\frac{\abs{z_1-z_2}}{2}}{\infty}V_r(z_1,z_2)\frac{\ld r}{r^{1+d+\alpha}} \\
&\lesssim \abs{z_1-z_2}^{-\alpha}\delta_N^{\alpha\wedge(2-\alpha)},
\end{align*}
using Proposition~\ref{prop:stable:convergence_centering_term}. It follows from Lemma~\ref{lem:alpha_int_bound} that
\begin{equation*}
\abs{\intrd{2}\phi(z_1)\phi(z_2)(\sigma^{\alpha}_{z_1,z_2}(\ft{s})-\sigma^{\alpha}_{z_1,z_2}(f_s))\ld z_1 \ld z_2}\cvgas{N} 0,
\end{equation*}
and we have shown that, for all $t\in[0,T]$
\begin{equation*}
\qvar{M^N(\phi)} \cvgas[L^1,P]{N} \qvar{W^{\alpha}(\phi)}.
\end{equation*}
\end{proof}

\subsection{Conclusion of the proof}\label{subsec:stable:conclusion}

We can now conclude the proof of Theorem~\ref{thm:clt_stable}. We have proved that the sequence $\proc{Z^N}{N\geq 1}$ is tight and we can characterise its potential limit points using the convergence of $M^N$. Recall the following expression for $\dual{\zt}$ from \eqref{stochastic_integral_zn_sr}~:
\begin{multline*}
\dual{\zt} = -\left(\frac{\tau_{N}}{\eta_{N}}\right)^{1/2}\alpha\intg{0}{t}\intg{1}{\infty}\dual{(\avg{\zt{s}}{\delta_N r})^{2}}{R_{2}(\avg{\wt{s}},\avg{\ft{s}})\avg{\varphi^{N}(s,t)}{\delta_N r}}\frac{\ld r}{r^{\alpha+1}}\ld s \\ + \intg{0}{t}\intrd\varphi^{N}(x,s,t)M^{N}(\ld x\ld s).
\end{multline*}
In Section~\ref{subsec:stable:tightness}, we showed that the first term converges to zero in $L^1$. In addition, by Lemmas~\ref{lemma:stable:convergence_varphi} and~\ref{lem_bound_qvar_sr},
\begin{equation*}
\intg{0}{t}\intrd\varphi^N(x,s,t)M^N(\ld x\ld s) - \intg{0}{t}\intrd\varphi(x,s,t)M^N(\ld x\ld s) \cvgas[L^2]{N} 0.
\end{equation*}
For $\phi_1,\ldots,\phi_p$ in $\calS(\R^d)$, let $\varphi_1,\ldots,\varphi_p$ be the corresponding solutions of \eqref{definition_varphi_sr} with $\phi=\phi_i$. Since $M^{N}$ converges weakly to $W^{\alpha}$, by \citep[Proposition 7.12]{walsh_introduction_1986}, for $t_1,\ldots,t_p\in[0,T]$
\begin{multline*}
 \left(\intg{0}{t_1}\intrd\varphi_1(x,s,t_1)M^{N}(\ld x\ld s),\ldots,\intg{0}{t_k}\intrd\varphi_k(x,s,t_k)M^{N}(\ld x\ld s)\right) \\ \cvgas[d]{N} \left(\intg{0}{t_1}\intrd\varphi_1(x,s,t_1)W^{\alpha}(\ld x\ld s),\ldots,\intg{0}{t_k}\intrd\varphi_k(x,s,t_k)W^{\alpha}(\ld x\ld s)\right).
\end{multline*}
Hence the same convergence holds (in distribution) for $\left( \dual{\zt{t_1}}{\phi_1},\ldots,\dual{\zt{t_p}}{\phi_p} \right)$ and this characterises the potential limit points of $\proc{Z^N}{N\geq 1}$. By Theorem~\ref{thm:mitoma}, $\proc{Z^N}$ converges in distribution to a distribution-valued process $\proc{z}$ which satisfies
\begin{equation*}
\dual{z_t} = \intg{0}{t}\intrd\varphi(x,s,t)W^{\alpha}(\ld x\ld s).
\end{equation*}
By the same argument as in Section~\ref{subsec:brownian:conclusion}, $\proc{z}$ solves the stochastic PDE \eqref{limiting_fluctuations_stable}, which concludes the proof.

\section{Drift load - proof of Theorem~\ref{thm:driftload_slfv}}\label{sec:proof_drift_load}

Recall the definition of $F$ and $\rho^{(r_{N})}_{z_{1},z_{2}}$ in \eqref{definition_F_driftload} and \eqref{def_tau_r} respectively. 

\begin{definition} [Martingale Problem (M3)] \label{formal_def_M3}
Given $\proc{\varepsilon_{N}}{N\geq 1}$, $\proc{\delta_{N}}{N\geq 1}$ and $F$, let $\eta_N = \varepsilon_N \delta_N^2$, $\tau_N = \varepsilon_N^2 \delta_N^d$ and $r_N=\delta_N R$. Then for $N\geq 1$, we say that a $\Xi$-valued process $(w_t^N)_{t\geq 0}$ satisfies the martingale problem (M3) if 
for all $\phi$ in $L^{1,\infty}(\R^d)$,
\begin{equation} \label{martingale_w_dr}
\dual{w^{N}_{t}} - \dual{w_{0}} - \eta_{N} u V_R \intg{0}{t} \braced{ \frac{2R^2}{d+2} \dual{w^{N}_{s}}{\L{r_{N}}\phi} - s\dual{ \avg{F(\avg{w^{N}_{s}})}{r_{N}} }} \ld s
\end{equation}
defines a (mean zero) square-integrable martingale with (predictable) variation process
\begin{equation} \label{variation_w_dr}
\tau_{N} u^2 V_R^2 \intg{0}{t} \intrd{2} \phi(z_{1}) \phi(z_{2}) \rho^{(r_{N})}_{z_{1},z_{2}}(w^{N}_{s}) \ld z_{1} \ld z_{2} \ld s + \bigO{t \tau_N \delta_N^2 \norm[2]{\phi}^2}.
\end{equation}
\end{definition}
(Again, uniqueness does not hold for this martingale problem, but we will not require it.)

Let $q^N_t$ denote the SLFVS with overdominance defined in Definition~\ref{definition_slfv_diploid} with parameters as defined in~\eqref{driftload_params} in Section \ref{subsec:driftload}.
As in Subsection \ref{subsec:mart_pb:fixed}, we consider the rescaled process
$w^N_t(x)=q_t^N(x/\delta_N)$.
By Proposition~\ref{prop:martingale_pb_slfv_dr}, using the same rescaling argument as in Proposition~\ref{prop:mart_pb_m1}, we have the following result.

\begin{proposition} \label{prop:mart_pb_M3}
The process $\proc{w^N}$ satisfies the martingale problem (M3).
\end{proposition}

As in Theorem \ref{thm:clt_slfv_fixed_radius}, we define the process of rescaled fluctuations by
\begin{equation}\label{definition_zn_driftload}
\zt = \left( \eta_N/\tau_N \right)^{1/2}\left( \wt-\lambda \right).
\end{equation}
(Recall that since $w_0 = \lambda$, the centering term is constant and equals $\lambda$.)
Then by the definition of $\Delta^N$ in \eqref{defn_driftload},
\begin{equation*}
\Delta^N(t,x) = \delta_N^2(s_1+s_2)\varepsilon_N\delta_N^{d-2}\E{\avg{\zt{\eta_N t}}(\delta_N x,r_N)^2}.
\end{equation*} 
Let us define the following notation for any $\phi\in L^{1,\infty}(\R^d)$,
\begin{equation}\label{scaling_phi}
\phi_r(x) = \frac{1}{r^d}\phi(x/r).
\end{equation}
Theorem~\ref{thm:driftload_slfv} is then a direct consequence of the following theorem. 
\begin{thm}\label{thm:driftload}
Suppose that $\tau_N/\eta_N = \littleO{r_N^{d+2}}$. Then for all $\phi\in L^{1,\infty}(\R^d)$, there exists a constant $C>0$ - depending on the dimension $d$ - such that, as $N \to\infty$ with $t \to \infty$ for $d\leq 2$ and $t\delta_N^{-2}\to \infty$ for $d\geq 3$, 
\begin{equation*}
\E{\dual{\zt}{\phi_{r_N}}^2} \underset{N,t\to\infty}{\sim} C \delta_N^{2-d} c_N.
\end{equation*}
\end{thm}
\begin{proof}[Proof of Theorem \ref{thm:driftload_slfv}]
Setting $\phi = \1{|x|\leq 1}$ gives the result for $\Delta ^N(t,0)$; the general result follows by symmetry.
\end{proof}

Note that the only difference between the martingale problems (M1) and (M3) in Definitions~\ref{formal_def_M1} and~\ref{formal_def_M3} is that $\sigma^{(r_{N})}_{z_{1},z_{2}}$ is replaced by $\rho^{(r_{N})}_{z_{1},z_{2}}$.
Hence it is easy to see that Lemma \ref{lemma:bound_qvar} and Lemma \ref{lemma_bound_avg} also hold in this case (with different constants).
It is also possible to adapt the proofs in Section \ref{subsec:brownian:convergence_M} to show that on compact time intervals, $\proc{Z^N}$ converges to the solution of the following SPDE, 
\begin{equation*}
\ld z_t = \left[ \frac{1}{2}\Delta z_t - F'(\lambda)z_t \right]\ld t + \sqrt{\frac{1}{2}\lambda(1-\lambda)}\ld W_t.
\end{equation*}
This process admits a stationary distribution, under which $\dual{z_t}$ is a Gaussian random variable with variance $$\frac{1}{2} \lambda (1-\lambda) \intg{0}{\infty} \intrd e^{-2 F'(\lambda) t} G_t \ast \phi (x)^2 \ld x \ld t.$$
We can thus hope to extend the convergence of $\proc{Z^N}$ to the whole real line (as in \cite{norman_ergodicity_1977}), and use the above expression to estimate the second moment of $\dual{\zt}{\phi_{r_N}}$ for large times. 
Some care is needed though, as we are letting the support of the test function vanish as $N\to\infty$.

\begin{proof}[Proof of Theorem~\ref{thm:driftload}]
Since $q_0^N = \lambda$, by the same argument as for \eqref{martingale_z},
\begin{equation} \label{driftload:spde_zn}
\ld \zt = \left[ \L{r_N}\zt - F'(\lambda)\dbavg{\zt}{r_N} - \left( \tau_N/\eta_N \right)^{1/2}\avg{(\avg{\zt})^2 R_2(\avg{\wt},\lambda)}{r_N} \right]\ld t + \ld M^N_t,
\end{equation}
where $M^N$ is a martingale measure with covariation measure $Q^N$ given by
\begin{equation} \label{eq:Q_N_dr}
Q^N(dz_1 dz_2 ds)= \rho ^{(r_N)}_{z_1,z_2} (\wt{s}) dz_1 dz_2 ds + \bigO{\delta_N^2} \delta_{z_1=z_2} (\ld z_1 \ld z_2) \ld s.
\end{equation}
Consider a time dependent test function $\varphi^N$ which solves
\begin{equation*}
\left\lbrace
\begin{aligned}
& \partial_s\varphi^N(x,s,t) + \L{r_N}\varphi^N(x,s,t) - F'(\lambda)\dbavg{\varphi^N(s,t)}(x,r_N) = 0, \\
& \varphi^N(x,t,t) = \phi(x).
\end{aligned}
\right.
\end{equation*}
Then, by Proposition~\ref{prop:time_dependent_tf} and \eqref{driftload:spde_zn},
\begin{equation}\label{driftload:stochastic_integral}
\dual{\zt} = -\left( \tau/\eta \right)^{1/2}\intg{0}{t}\dual{(\avg{\zt{s}})^2}{R_2(\avg{\wt{s}},\lambda)\avg{\varphi^N(s,t)}{r_N}}\ld s + \intg{0}{t}\intrd\varphi^N(x,s,t)M^N(\ld x\ld s).
\end{equation}
The remainder of the proof now consists of proving that the main contribution to the variance of $\dual{\zt}{\phi_{r_N}}$ is made by the last term on the right-hand-side and then estimating this contribution. Note that $\varphi^N$ is given explicitly by
\begin{equation} \label{eq:driftload_varphi}
 \varphi^N(x,s,t) = e^{-F'(\lambda)(t-s)}\G{r_N}_{D_N(t-s)}\ast\phi(x),
\end{equation}
with $D_N = 1-F'(\lambda)\frac{2r_N^2}{d+2}$. In particular, $\norm[q]{\varphi(s,t)}\leq\norm[q]{\phi}e^{-F'(\lambda)(t-s)}$. The following lemma extends the result of Lemma \ref{lemma_bound_avg} to arbitrarily large times, and will be proved in Subsection~\ref{subsec:driftload_proofs}.
\begin{lemma}\label{lemma:bound_avg_2}
There exist constants $K_1'$ and $K_0'$ such that, for all $x\in\R^d$ and all $t\geq 0$,
\begin{align*}
 \E{\avg{\zt}(x,r_N)^2} \leq \frac{K_1'}{r_N^d}, && \mathrm{and} && \E{\avg{\zt}(x,r_N)^4} \leq \frac{K_0'}{r_N^{2d}}.
\end{align*}
\end{lemma}
Using the expression for $\varphi^N$ in \eqref{eq:driftload_varphi} and the Cauchy-Schwartz inequality,
\begin{multline*}
\E{\left( \intg{0}{t}\dual{(\avg{\zt{s}})^2}{R_2(\avg{\wt{s}},\lambda)\avg{\varphi^N(s,t)}}\ld s \right)^2} \\ \leq \frac{1}{4}\norm[\infty]{F''}^2\frac{1-e^{-F'(\lambda)t}}{F'(\lambda)}\E{\intg{0}{t}e^{-F'(\lambda)(t-s)}\dual{(\avg{\zt{s}})^2}{\abs{\avg{\G{r_N}_{D_N(t-s)}\ast\phi}}}^2\ld s}.
\end{multline*}
Another use of the Cauchy-Schwartz inequality yields
\begin{equation*}
\dual{(\avg{\zt{s}})^2}{\abs{\avg{\G{r_N}_{D_N(t-s)}\ast\phi}}}^2 \leq \norm[1]{\G{r_N}_{D_N(t-s)}\ast\phi}\dual{(\avg{\zt{s}})^4}{\abs{\avg{\G{r_N}_{D_N(t-s)}\ast\phi}}}.
\end{equation*}
Hence, using Lemma~\ref{lemma:bound_avg_2} and the fact that $\norm[1]{\G_t\ast\phi}\leq \norm[1]{\phi}$,
\begin{equation*}
\E{\dual{(\avg{\zt{s}})^2}{\abs{\avg{\G{r_N}_{D_N(t-s)}\ast\phi}}}^2} \leq \norm[1]{\phi}^2\frac{K_0'}{r_N^{2d}}.
\end{equation*}
As a result,
\begin{equation}\label{bound_nonlinearity}
\E{\left( \intg{0}{t}\dual{(\avg{\zt{s}})^2}{R_2(\avg{\wt{s}},\lambda)\avg{\varphi^N(s,t)}}\ld s \right)^2} \lesssim \norm[1]{\phi}^2r_N^{-2d},
\end{equation}
uniformly in $t\in\R_+$. We now move on to estimating the contribution of the second term in \eqref{driftload:stochastic_integral}. The following lemma will be proved in Subsection~\ref{subsec:driftload_proofs}.
\begin{lemma}\label{lemma:estimate_variance}
As $N\to\infty$,
 \begin{equation*}
\E{\left( \intg{0}{t}\intrd\varphi^N(x,s,t)M^N(\ld x\ld s) \right)^2} = \frac{1}{2}\lambda(1-\lambda)\intg{0}{t}\norm[2]{\avg{\varphi^N(s,t)}{r_N}}^2\ld s + \littleO{r_N} \intg{0}{t}\norm[2]{\varphi^N(s,t)}^2\ld s.
 \end{equation*}
\end{lemma}
As we shall see in Subsection \ref{subsec:driftload_proofs}, this is a consequence of the fact that in the expression for $Q^N$ in \eqref{eq:Q_N_dr}, $w^N$ can be replaced by $\lambda$. 
As a result, using \eqref{eq:driftload_varphi} and \eqref{bound_nonlinearity} in~\eqref{driftload:stochastic_integral} and since $\tau_N/\eta_N=\littleO{r_N^{d+2}}$, we have
\begin{multline*}
\E{\dual{\zt}{\phi_{r_N}}^2} = \frac{1}{2}\lambda(1-\lambda)\intg{0}{t}e^{-2F'(\lambda)s}\norm[2]{\avg{\G{r_N}_{D_N s}\ast\phi_{r_N}}{r_N}}^2\ld s \\ + \littleO{r_N} \intg{0}{t}e^{-2F'(\lambda)s}\norm[2]{\G{r_N}_{D_N s}\ast\phi_{r_N}}^2\ld s + \littleO{r_N^{2-d}}.
\end{multline*}
To study the asymptotic behaviour of the first integral, we use the scaling properties of the function $\G$. Recall that $\proc{\xi^{(r)}}$ is a L\'evy process with infinitesimal generator $\L$; it is not difficult to show that it satisfies the following scaling property:
\begin{equation*}
 \E[x]{\phi(\xi_t^{(r)})} = \E[x/c]{\phi(c\,\xi_{t/c^2}^{(r/c)})}.
\end{equation*}
(Simply look at the infinitesimal generator of both processes.) Hence
\begin{equation*}
\G{r_N}_t\ast\phi_{r_N}(x) = r_N^{-d}\G{1}_{t/r_N^2}\ast\phi_1(x/r_N).
\end{equation*}
Set $f(t) = \norm[2]{\avg{\G{1}_t\ast\phi}{1}}^2$; it follows that
\begin{equation} \label{eq:f_rescaling}
\norm[2]{\avg{\G{r_N}_{D_N s}\ast\phi_{r_N}}{r_N}}^2 = r_N^{-d}f(D_N s/r_N^2).
\end{equation}
If we can show that, as $N,t\to\infty$, there is a constant $\tilde{C}>0$ such that
\begin{equation}\label{integral_estimate}
\intg{0}{t}e^{-2F'(\lambda)s}f(D_Ns/r_N^2)\ld s \sim \tilde{C}r_N^2c_N,
\end{equation}
the result will follow.
For this we need the following estimate of $f(t)$ when $t \to \infty$.

\begin{lemma} \label{lem:estimate_f}
For $\phi \geq 0$, as $t \to \infty$,
\begin{equation}\label{asymptotics_f}
f(t) \sim (4\pi t)^{-d/2}\norm[1]{\phi}^2.
\end{equation}
\end{lemma}

For the proof of this estimate we will use the following properties of the semigroup $\G$, which will be proved in Appendix \ref{append:driftload}.

\begin{lemma} \label{lem:density_green_fct}
For any $r>0$ and $t>0$, the law of $\xi^{(r)}_t$ takes the form
\begin{equation*}
\G_t(\ld x) = e^{-\frac{(d+2)}{2r^2} t} \delta_0 (\ld x) + g^{(r)}_t (x) \ld x.
\end{equation*}
Furthermore, $g^{(r)}_t$ is continuous on $\R^d$, is invariant under rotations which fix the origin and $g^{(r)}_t (y)$ is a decreasing function of $\abs{y}$.
\end{lemma}

\begin{proof}[Proof of Lemma~\ref{lem:estimate_f}]
By the semigroup property of $\phi\mapsto\G_t\ast\phi$, $f(t)$ can also be written\\ $ \dual{ \G{1}_{2t} \ast \avg{\phi}(1) }{ \avg{\phi}(1) }$. 
In addition, by the scaling property of $\proc{\xi^{(r)}}$ and using Lemma~\ref{lem:density_green_fct},
\begin{align*}
\G{1}_{2t} \ast \phi (x) &= \E[0]{ \phi( x + \sqrt{t} \xi^{(1/\sqrt{t})}_2 ) } \\
&= \phi(x) e^{-(d+2)t } + \intrd \phi(x + \sqrt{t}y) g^{(1/\sqrt{t})}_2 (y) \ld y \\
&= \phi(x) e^{-(d+2)t } + t^{-d/2} \intrd \phi(x+y) g^{(1/\sqrt{t})}_2(y/\sqrt{t}) \ld y.
\end{align*}
By Proposition~\ref{prop:average}.ii and Theorem 4.8.2 in \cite{ethier_markov_1986}, the finite dimensional distributions of $\proc{\xi^{(r)}}$ converge to those of standard Brownian motion as $r \to 0$.
In particular, $\xi^{(r)}_2 \cvgas[d]{r}{0} \mathcal{N}(0,2)$, and $g^{(r)}_2(x) \to G_2(x)$ as $r \to 0$ for almost every $x \in \R^d$ (the probability that $\xi^{(r)}_t =0$ vanishes as $r \to 0$ for any $t >0$).
Since $G_2$ is continuous on $\R^d$ and $g^{(r)}_2$ is decreasing as a function of the modulus, this convergence takes place uniformly on compact sets by Dini's second theorem.
So, fixing $\epsilon >0$, for any $R > 0$, for $r$ small enough,
\begin{equation*}
\sup_{ \abs{x} < R } \abs{ g^{(r)}_2(x) - G_2(x) } \leq \epsilon.
\end{equation*}
As a result, using the continuity of $G_2$, for any $y$, for $t$ large enough,
\begin{equation*}
\abs{ g^{(1/\sqrt{t})}_2(y/\sqrt{t}) - G_2(0) } \leq 2 \epsilon.
\end{equation*}
Hence, since $g^{(r)}_t(y) \leq g^{(r)}_t (0)$, by dominated convergence,
\begin{equation} \label{convergence_gt}
\intrd \phi(x+y) g^{(1/\sqrt{t})}_2(y/\sqrt{t}) \ld y \cvgas{t} (4\pi)^{-d/2} \intrd \phi(y) \ld y.
\end{equation}
From the above expression for $f$,
\begin{equation*}
f(t) = e^{-(d+2)t} \intrd \avg{\phi}(x,1)^2 \ld x + t^{-d/2} \intrd{2} g^{(1/\sqrt{t})}_2(y/\sqrt{t}) \avg{\phi}(x+y,1) \avg{\phi}(x,1) \ld y \ld x. 
\end{equation*}
Replacing $\phi$ with $\avg{\phi}(1)$ in \eqref{convergence_gt} and letting $t \to \infty$ yields the result.
\end{proof}

Furthermore, $0\leq f(t) \leq \norm[2]{\phi}^2$ for all $t\geq 0$, and thus $f$ is integrable on $(0,\infty)$ if and only if $d\geq 3$. 
\begin{remark}
This is in fact a consequence of the fact that $\proc{\xi^{(1)}}$ is transient if and only if $d\geq 3$ (as with Brownian motion). 
The function $f$ can be expressed in terms of the probability of $\xi^{(1)}_{2t}$ being in a ball of radius $1$, which is integrable on $(0,\infty)$ if and only if $\proc{\xi^{(1)}}$ is transient.
\end{remark}

We now prove \eqref{integral_estimate} separately for each regime.

\paragraph*{High dimension}

If $d\geq 3$, change the variable of integration to write
\begin{equation*}
\intg{0}{t}e^{-2F'(\lambda)s}f(D_N s/r_N^2)\ld s = r_N^2\intg{0}{t/r_N^2}e^{-2F'(\lambda)r_N^2 s}f(D_N s)\ld s.
\end{equation*}
Since $f$ is integrable, by dominated convergence, and since $t\delta_N^{-2}\to \infty$,
\begin{equation*}
\intg{0}{t/r_N^2}e^{-2F'(\lambda)r_N^2 s}f(D_N s)\ld s \cvgas{N,t} \intg{0}{\infty}f(s)\ld s.
\end{equation*}

\paragraph*{Dimension 1}

If $d=1$, however, from \eqref{asymptotics_f}, we see that, as $N\to\infty$, $\frac{1}{r_N}f(s/r_N^2)\to (4\pi s)^{-1/2}\norm[1]{\phi}^2$, so, by dominated convergence,
\begin{equation*}
\intg{0}{t}e^{-2F'(\lambda)s}f(D_Ns/r_N^2)\ld s \underset{N,t\to\infty}{\sim} r_N\norm[1]{\phi}^2\hat{C},
\end{equation*}
for some constant $\hat{C}>0$.

\paragraph*{Dimension 2}

If $d=2$, let $T_1$ and $T_2$ be two positive constants and assume that $t \geq T_2$. We split the integral as follows~:
\begin{multline*}
\intg{0}{t}e^{-2F'(\lambda)s}f(D_Ns/r_N^2)\ld s = \intg{0}{r_N^2 T_1}e^{-2F'(\lambda)s}f(D_Ns/r_N^2)\ld s \\ + \intg{r_N^2 T_1}{T_2}e^{-2F'(\lambda)s}f(D_Ns/r_N^2)\ld s + \intg{T_2}{t}e^{-2F'(\lambda)s}f(D_Ns/r_N^2)\ld s.
\end{multline*}
We first show that the first and last terms are of order $r_N^2$. Since $0\leq f(t) \leq \norm[2]{\phi}^2$ for all $t\geq 0$,
\begin{equation*}
 \abs{\intg{0}{r_N^2 T_1}e^{-2F'(\lambda)s}f(D_Ns/r_N^2)\ld s} \leq r_N^2 T_1\norm[2]{\phi}^2,
\end{equation*}
and by \eqref{asymptotics_f}
\begin{align*}
 \abs{\intg{T_2}{t}e^{-2F'(\lambda)s}f(D_Ns/r_N^2)\ld s} &\lesssim r_N^2\intg{T_2}{\infty}e^{-2F'(\lambda)s}\frac{\ld s}{s}.
\end{align*}
For the middle term, by \eqref{asymptotics_f}, $\frac{1}{r_N^2}f(s/r_N^2) \cvgas{N} (4\pi s)^{-1}\norm[1]{\phi}^2$, so as $N\to\infty$, by dominated convergence,
\begin{equation*}
 \intg{r_N^2 T_1}{T_2}e^{-2F'(\lambda)s}f(D_Ns/r_N^2)\ld s \sim r_N^2(4\pi)^{-1}\norm[1]{\phi}^2\intg{r_N^2 T_1}{T_2}e^{-2F'(\lambda)s}\frac{\ld s}{s}.
\end{equation*}
Further
\begin{equation*}
 \abs{\intg{T_1 r_N^2}{T_2}e^{-2F'(\lambda)s}\frac{\ld s}{s}-\intg{T_1 r_N^2}{T_2}\frac{\ld s}{s}} \leq 2F'(\lambda)\intg{T_1 r_N^2}{T_2}s\frac{\ld s}{s} \leq 2F'(\lambda)T_2,
\end{equation*}
and
\begin{equation*}
 \intg{T_1 r_N^2}{T_2}\frac{\ld s}{s} = \log\left(\frac{T_2}{T_1 r_N^2}\right) \sim \abs{\log r_N^2}.
\end{equation*}
As a result
\begin{equation*}
 \intg{0}{t}e^{-2F'(\lambda)s}f(D_Ns/r_N^2)\ld s \sim \frac{\norm[1]{\phi}^2}{4\pi}r_N^2\abs{\log r_N^2},
\end{equation*}
as $N, t\to\infty$. We have thus proved \eqref{integral_estimate}, and the result.
\end{proof}

\subsection{Proofs of Lemmas \ref{lemma:bound_avg_2} and \ref{lemma:estimate_variance}}\label{subsec:driftload_proofs}

The proof of Lemma \ref{lemma:bound_avg_2} requires the following two technical lemmas, which are proved in Appendix~\ref{append:driftload}.
\begin{lemma}\label{lemma:magic}
 Let $\phi:\R^d\to\R$, $r>0$ and suppose that $g:\R^d\to\R$ satisfies $0<\gamma\leq g(x)\leq 1$ for all $x\in\R^d$. Then
\begin{equation*}
 2\phi(x)\L\phi(x) - 2\phi(x)\dbavg{\phi\,g}(x,r) \leq \L\phi^2(x) - 2\left(\gamma-\frac{r^2}{d+2}\right)\phi(x)^2.
\end{equation*}
Further, for some constant $c>0$, for $r$ small enough,
\begin{equation*}
 4\phi(x)^3\L\phi(x) - 4\phi(x)^3\dbavg{\phi\,g}(x,r) \leq \L\phi^4(x) - 4(\gamma-c\,r^2)\phi(x)^4.
\end{equation*}
\end{lemma}

\begin{lemma}\label{lemma_differential_inequality}
Suppose $h:\R_{+}\times\R^{d}\to\R$ is a function that is continuously differentiable with respect to the time variable $t$ and which satisfies the following differential inequality for some positive $\alpha$~:
\begin{equation*}
\partial_{t}h_{t}(x)\leq \calL h_{t}(x)-\alpha h_{t}(x) + g_{t}(x).
\end{equation*}
Then for all $0\leq s\leq t$ and for any $1\leq q\leq \infty$,
\begin{equation*}
\norm[q]{h_{t}} \leq e^{-\alpha (t-s)}\norm[q]{h_{s}} + \frac{1}{\alpha}\sup_{u\in[s,t]}\norm[q]{g_{u}}.
\end{equation*}
\end{lemma}
\begin{proof}[Proof of Lemma~\ref{lemma:bound_avg_2}]
Set
\begin{equation*}
h(t,x) = \E{\avg{\zt}(x,r_N)^{2}}.
\end{equation*}
We are going to make use of Lemma~\ref{lemma_differential_inequality}, so we want to obtain a differential inequality for $h$. To this end, average \eqref{driftload:spde_zn} on $B(x, r_N)$ to get
\begin{equation*}
\ld\avg{\zt{s}}(x,r_N) = \left[\L{r_N}\avg{\zt{s}}(x,r_N) - \dbavg{\avg{\zt{s}} \, R_{1}(\avg{\wt{s}},\lambda)}(x,r_N)\right]\ld s + \frac{1}{V_{r_N}}\ld M^{N}_{s}(B(x,r_N)).
\end{equation*}
(From now on all averages will be over radius $r_N$.) By the generalised It\^o formula, noting $\Delta Y_s = Y_s - Y_{s^-}$,
\begin{multline*}
\ld \left(\avg{\zt{s}}(x)\right)^{2} = 2\avg{\zt{s}}(x)\ld \avg{\zt{s}}(x) + \ld\left[\avg{\zt{\cdot}}(x)\right]_{s} \\+ \left(\avg{\zt{s^{-}}}(x)+\Delta\avg{\zt{s}}(x)\right)^{2}-\left(\avg{\zt{s^{-}}}(x)\right)^{2} - 2\avg{\zt{s^{-}}}(x)\Delta\avg{\zt{s}}(x) - \left(\Delta\avg{\zt{s}}(x)\right)^{2}.
\end{multline*}
Expanding the brackets, the terms on the second line cancel and, integrating for $s\in[0,t]$, we have
\begin{multline*}
\avg{\zt}(x)^{2} = 2\intg{0}{t}\avg{\zt{s}}(x)\left[\L{r_N}\avg{\zt{s}}(x)-\dbavg{\avg{\zt{s}} \, R_{1}(\avg{\wt{s}},\lambda)}(x)\right]\ld s \\+ \frac{2}{V_{r_N}}\intg{0}{t}\avg{\zt{s}}(x)\ld M^{N}_{s}(B(x,r_N)) + \frac{1}{V_{r_N}^{2}}\left[M^{N}(B(x,r_N))\right]_{t}.
\end{multline*}
Taking expectations on both sides, since the second term is a martingale,
\begin{equation*}
h(t,x) = 2\intg{0}{t}\E{\avg{\zt{s}}(x)\L{r_N}\avg{\zt{s}}(x) - \avg{\zt{s}}(x)\dbavg{\avg{\zt{s}} \, R_{1}(\avg{\wt{s}},\lambda)}(x)}\ld s + \frac{1}{V_{r_N}^{2}}\E{\qvar{M^{N}(B(x,r_N))}}.
\end{equation*}
Differentiating yields
\begin{multline*}
\deriv*{h}{t}(t,x) = 2\E{\avg{\zt{t}}(x)\L{r_N}\avg{\zt{t}}(x) - \avg{\zt{t}}(x)\dbavg{\avg{\zt{t}} \, R_{1}(\avg{\wt},\lambda)}(x)} \\ + \frac{1}{V_{r_N}^2}\E{\int_{B(x,r_N)^2}\rho^{(r_N)}_{z_1,z_2}(\wt)\ld z_1\ld z_2} + \bigO{\frac{\delta_N^2}{V_{r_N}}}.
\end{multline*}
The second term is bounded by $\frac{1}{V_{r_N}}$, and the first one has the same form as the left-hand-side of the first statement of Lemma~\ref{lemma:magic}. In \cite{norman_markovian_1974} (at the beginning of the proof of Theorem~3.2), it is proved that the conditions on $F$ in \eqref{F_stable_equilibrium_1}-\eqref{F_stable_equilibrium_2} imply
\begin{equation}\label{bound_R1}
\inf_{x\in[0,1]}R_{1}(x,\lambda) =: \gamma >0.
\end{equation}
Then, taking $\phi = \avg{\zt}$ and $g = R_1(\avg{\wt{s}},\lambda)$, Lemma~\ref{lemma:magic} implies that, for all $t\geq 0$,
\begin{equation*}
\deriv*{h}{t}(t,x) \leq \calL h(t,x) - \alpha_{N} h(t,x) + \frac{1+\bigO{\delta_N^2}}{V_{r_N}},
\end{equation*}
with $\alpha_N = \gamma+\bigO{r_N^2}$. Using Lemma~\ref{lemma_differential_inequality} (with $s=0$) we can now write, since $\zt{0}=0$,
\begin{equation} \label{z_dr_squared_bound}
\E{\avg{\zt}(x)^{2}} \leq \frac{1+\bigO{\delta_N^2}}{\alpha_{N} V_{r_N}} \lesssim \frac{1}{r_N^d}.
\end{equation}
The second inequality is proved in essentially the same way, although the computations become more involved. We compute the fourth moment of $\avg{\zt}$ with It\^o's formula, as before:
\begin{multline*}
 \ld\left(\avg{\zt}(x)\right)^4 = 4(\avg{\zt}(x))^3\ld\avg{\zt}(x) + \frac{1}{2}4\times 3 (\avg{\zt}(x))^2\ld \left[\avg{\zt{\cdot}}\right]_t \\ + \left(\avg{\zt{t^-}}(x)+\Delta\avg{\zt}(x)\right)^4 - \left(\avg{\zt{t^-}}(x)\right)^4 - 4(\avg{\zt{t^-}}(x))^3\Delta\avg{\zt}(x) - \frac{1}{2}3\times 4(\avg{\zt{t^-}}(x))^2(\Delta\avg{\zt}(x))^2.
\end{multline*}
Hence, taking expectations, the martingale terms can be dropped and we write:
\begin{multline*}
\E{ \left( \avg{\zt}(x) \right)^4 } = 4 \intg{0}{t} \E{ \avg{\zt{s}}(x)^3 \L{r_N} \avg{\zt{s}}(x) - \avg{\zt{s}}(x)^3 \dbavg{\avg{\zt{s}} \, R_1(\avg{\wt{s}},\lambda) }(x) } \ld s \\ + 6 \frac{1}{V_{r_N}^2} \intg{0}{t} \int_{B(x,r_N)^2} \E{ \avg{\zt{s}}(x)^2 \rho^{(r_N)}_{z_1,z_2}(\wt{s}) } \ld z_1 \ld z_2 \ld s + \bigO{\delta_N^2} \frac{1}{V_{r_N}} \int_{0}^{t} \E{ \avg{\zt{s}}(x)^2 } \ld s  \\ + \E{ \sum_{s\leq t} \braced{ 4\avg{\zt{s^{-}}}(x) (\Delta \avg{\zt{s}}(x))^3 + (\Delta \avg{\zt{s}}(x))^4 } },
\end{multline*}
where the sum is over jump times for the process $(\avg{\zt}(x))_{t\geq 0}$. 
We can bound the size of the jumps $\Delta\avg{\zt{s}}(x)$ by a deterministic constant. 
By the definition of the SLFVS with overdominance in Definition \ref{definition_slfv_diploid},
$$\sup_{t\geq 0}\abs{\dual{q^N_t}-\dual{q^N_{t^-}}}\leq u\varepsilon_N \norm[1]{\phi}. $$
Hence $\abs{\Delta\avg{\zt{s}}(x)}\leq u\varepsilon_N (\eta_N/\tau_N)^{1/2}=u\varepsilon_N^{1/2}\delta_N^{1-d/2}$.
As a result
\begin{multline*}
\E{ \sum_{s\leq t} \braced{ 4 \avg{\zt{s^{-}}}(x) ( \Delta \avg{\zt{s}}(x))^3 + (\Delta \avg{\zt{s}}(x))^4 } } \\ \leq \E{ \sum_{s\leq t} \braced{ 4 (u \varepsilon_N^{1/2} \delta_N^{1-d/2})^3 \abs{ \avg{\zt{s^{-}}}(x) } + (u \varepsilon_N^{1/2} \delta_N^{1-d/2})^4 } },
\end{multline*}
where the sum is still over the jump times of $\avg{\zt{s}}(x)$.
These jumps occur according to a Poisson process with rate $V_{2R} \, \eta_N^{-1}$, so, using~\eqref{z_dr_squared_bound} to bound $\E{\abs{\avg{\zt{s^{-}}}(x)}}$, we obtain
\begin{multline*}
\E{ \sum_{s\leq t} \braced{ 4\avg{\zt{s^{-}}}(x) (\Delta \avg{\zt{s}}(x))^3 + (\Delta \avg{\zt{s}}(x))^4 }} \\ \leq V_{2R} \, \eta_N^{-1} \braced{ 4 (u \varepsilon_N^{1/2} \delta_N^{1-d/2})^3 \E{ \int_{0}^{t} \abs{ \avg{\zt{s^{-}}}(x) } \ld s } + t (u \varepsilon_N^{1/2} \delta_N^{1-d/2})^4 } = \littleO{r_N^{-2d}}.
\end{multline*}
Now note that
\begin{align*}
 \int_{B(x,r_N)^2}\E{\avg{\zt{s}}(x)^2\rho^{(r)}_{z_1,z_2}(\wt{s})}\ld z_1\ld z_2 &\lesssim \frac{1}{r_N^d}\int_{B(x,r_N)^2}\frac{V_{r_N}(z_1,z_2)}{V_{r_N}^2}\ld z_1\ld z_2\\
 & \lesssim 1.
\end{align*}
Hence, setting $h(t,x) = \E{(\avg{\zt}(x))^4}$,
\begin{equation*}
\deriv*{h}{t}(t,x) = 4\E{\avg{\zt}(x)^3\L{r_N}\avg{\zt}(x) - \avg{\zt}(x)^3\dbavg{\avg{\zt}\,R_1(\avg{\wt},\lambda)}} + \frac{g_t(x)}{r_N^{2d}},
\end{equation*}
where $\abs{g_t(x)}\lesssim 1$. Now the second statement of Lemma~\ref{lemma:magic} yields~:
\begin{equation*}
\deriv*{h}{t}(t,x) \leq \calL h(t,x) - 4(\gamma-cr_N^2)h(t,x) + \frac{g_t(x)}{r_N^{2d}},
\end{equation*}
and by Lemma~\ref{lemma_differential_inequality}, we have
\begin{equation*}
h(t,x) \lesssim \frac{1}{r_N^{2d}},
\end{equation*}
uniformly in $t\geq 0$.
\end{proof}
The following lemma is needed in the proof of Lemma \ref{lemma:estimate_variance}.
\begin{lemma}\label{lemma:bound_phi}
The following holds uniformly for all $t\geq 0$:
\begin{equation*}
 \E{\abs{\dual{\zt}}} \lesssim r_N^{1-d/2}c_N^{1/2}(\norm[1]{\phi}+r_N^{d/2}\norm[2]{\phi}).
\end{equation*}
\end{lemma}

\begin{proof}
Recall the expression for $\dual{\zt}$ in \eqref{driftload:stochastic_integral}; using Lemma~\ref{lemma:bound_avg_2} and Lemma~\ref{lemma:bound_qvar}, we can write
\begin{equation*}
 \E{\abs{\dual{\zt}}} \lesssim \frac{\left( \tau/\eta \right)^{1/2}}{r_N^d}\intg{0}{t}\norm[1]{\phi}e^{-F'(\lambda)(t-s)}\ld s + \left(\intg{0}{t}e^{-2F'(\lambda)(t-s)}\norm[2]{\avg{\G{r_N}_{D_N(t-s)}\ast\phi}{r_N}}^2\ld s\right)^{1/2}.
\end{equation*}
Replacing $\phi$ by $(\phi_{1/r_N})_{r_N}$ - as defined in \eqref{scaling_phi} - to use \eqref{eq:f_rescaling} and then looking at the proof of \eqref{integral_estimate} in the proof of Theorem~\ref{thm:driftload}, we see that
\begin{equation*}
 \intg{0}{t}e^{-2F'(\lambda)(t-s)}\norm[2]{\avg{\G{r_N}_{D_N(t-s)}\ast\phi}{r_N}}^2\ld s \lesssim r_N^{2-d}c_N(\norm[1]{\phi_{1/r_N}}^2 + \norm[2]{\phi_{1/r_N}}^2).
\end{equation*}
But $\norm[1]{\phi_{1/r_N}} = \norm[1]{\phi}$ and $\norm[2]{\phi_{1/r_N}} = r_N^{d/2}\norm[2]{\phi}$, hence
\begin{equation*}
 \E{\abs{\dual{\zt}}} \lesssim \norm[1]{\phi}\frac{(\tau/\eta)^{1/2}}{r_N^{d}} + r_N^{1-d/2}c_N^{1/2}(\norm[1]{\phi}+r_N^{d/2}\norm[2]{\phi}),
\end{equation*}
and we have the required result since $\tau_N/\eta_N = \littleO{r_N^{d+2}}$.
\end{proof}

\begin{proof}[Proof of Lemma~\ref{lemma:estimate_variance}]
We drop the superscript $N$ from $\varphi^N$ throughout the proof and take averages over radius $r:=r_N$. Recall from the expressions for $Q^N$ in \eqref{eq:Q_N_dr} and $\rho^{(r)}$ in \eqref{def_tau_r} that the variance of the stochastic integral $\intg{0}{t}\intrd\varphi(x,s,t)$ $M^N(\ld x\ld s)$ is given by
\begin{multline*}
 \intg{0}{t}\intrd{3}\frac{1}{V_{r}^2}\inball{x}{z_1}{x}{z_2}\varphi(z_1,s,t)\varphi(z_2,s,t)\mathbb{E}\Big[ \avg{\wt{s}}(x,r_N)^2(1-\wt{s}(z_1))(1-\wt{s}(z_2)) \\ 
 \quad \quad \quad +2\avg{\wt{s}}(x,r_N)(1-\avg{\wt{s}}(x,r_N))(\tfrac{1}{2}-\wt{s}(z_1))(\tfrac{1}{2}-\wt{s}(z_2))\\
 + (1-\avg{\wt{s}}(x,r_N))^2\wt{s}(z_1)\wt{s}(z_2)\Big]\ld x\ld z_1\ld z_2\ld s + \bigO{\delta_N^2} \int_{0}^{t} \norm[2]{\varphi(s,t)}^2\ld s,
\end{multline*}
which can also be written
\begin{align}
\intg{0}{t}\mathbb{E} \left[ \dual{(\avg{\wt{s}})^2}{\left( \avg{(1-\wt{s})\varphi(s,t)}\right)^2}
+ \dual{2\avg{\wt{s}}(1-\avg{\wt{s}})}{\left( \avg{(\tfrac{1}{2}-\wt{s})\varphi(s,t)} \right)^2}\nonumber \right. \\
\left. + \dual{(1-\avg{\wt{s}})^2}{\left( \avg{\wt{s}\varphi(s,t)} \right)^2} + \bigO{\delta_N^2\norm[2]{\varphi(s,t)}^2} \right] \ld s. \label{eq:lambda_exp}
\end{align}
We want to show that in this expression, $\wt{s}$ can (asymptotically) be replaced by $\lambda$, hence we write
\begin{multline*}
\dual{(\avg{\wt})^2}{\left( \avg{(1-\wt)\varphi} \right)^2} - \dual{\lambda^2}{(1-\lambda)^2\avg{\varphi}^2} \\ = \dual{(\avg{\wt})^2-\lambda^2}{\left(\avg{(1-\wt)\varphi}\right)^2} + \dual{\lambda^2}{\left( \avg{(1-\wt)\varphi} \right)^2-(1-\lambda)^2\avg{\varphi}^2}.
\end{multline*}
Since $(\avg{\wt})^2-\lambda^2=(\tau/\eta)^{1/2}\avg{\zt}(\avg{\wt}+\lambda)$, using Lemma~\ref{lemma:bound_avg_2},
\begin{align*}
 \E{\abs{\dual{(\avg{\wt})^2-\lambda^2}{\left(\avg{(1-\wt)\varphi}\right)^2}}} &\leq 2(\tau/\eta)^{1/2}\dual{\E{(\avg{\zt})^2}^{1/2}}{\avg{\abs{\varphi}}^2} \\
&\lesssim \frac{(\tau/\eta)^{1/2}}{r_N^{d/2}}\norm[2]{\varphi}^2 = \littleO{r_N\norm[2]{\varphi}^2}.
\end{align*}
In addition, 
\begin{multline*}
 \dual{\lambda^2}{\left( \avg{(1-\wt)\varphi} \right)^2-(1-\lambda)^2\avg{\varphi}^2} \\= \lambda^2\intrd{3}\frac{1}{V_r^2}\inball{x}{z_1}{x}{z_2}\varphi(z_1)\varphi(z_2)(\wt(z_1)-\lambda)(\wt(z_2)+\lambda-2)\ld x\ld z_1\ld z_2.
\end{multline*}
(The cross terms cancel out by symmetry.) Thus,
\begin{equation*}
\abs{\dual{\lambda^2}{\left( \avg{(1-\wt)\varphi} \right)^2-(1-\lambda)^2\avg{\varphi}^2}} \leq 2\lambda^2(\tau/\eta)^{1/2}\intrd\abs{\varphi(z_2)}\abs{\dual{\zt}{\psi^N_{z_2}}}\ld z_2,
\end{equation*}
where $\psi^N_{z_2}(z_1) = \frac{V_r(z_1,z_2)}{V_r^2}\varphi(z_1)$. In particular, 
\begin{align} \label{norm_psi}
\norm[1]{\psi^N_{z_2}} = \dbavg{\abs{\varphi}}(z_2,r_N), && \text{and} && \norm[2]{\psi^N_{z_2}}^2 &\leq  \frac{1}{V_{r_N}}\dbavg{\abs{\varphi}^2}(z_2,r_N).
\end{align}
By Lemma~\ref{lemma:bound_phi}, we get
\begin{align*}
\E{\intrd\abs{\varphi(z_2)}\abs{\dual{\zt}{\psi_{z_2}}}\ld z_2} 
&\lesssim r_N^{1-d/2}c_N^{1/2}\intrd\abs{\varphi(z_2)}\left(\norm[1]{\psi^N_{z_2}} + r_N^{d/2}\norm[2]{\psi^N_{z_2}}\right)\ld z_2 \\
&\lesssim r_N^{1-d/2}c_N^{1/2}\norm[2]{\varphi} \left( \intrd(\norm[1]{\psi^N_{z_2}}^2+r_N^d\norm[2]{\psi^N_{z_2}}^2)\ld z_2 \right)^{1/2},
\end{align*}
using the Cauchy-Schwartz inequality in the second line.
By \eqref{norm_psi},
\begin{equation*}
\intrd(\norm[1]{\psi^N_{z_2}}^2+r_N^d\norm[2]{\psi^N_{z_2}}^2)\ld z_2 \lesssim \norm[2]{\varphi}^2.
\end{equation*}
Since $\tau_N/\eta_N = \littleO{r_N^{d+2}}$,
\begin{equation*}
\E{\abs{\dual{\lambda^2}{\left( \avg{(1-\wt)\varphi} \right)^2-(1-\lambda)^2\avg{\varphi}^2}}} = \littleO{r_N^2c_N^{1/2}\norm[2]{\varphi}^2}.
\end{equation*}
We use a similar argument for the other terms in \eqref{eq:lambda_exp} to show that replacing $\wt{s}$ by $\lambda$ makes a difference of $\littleO{r_N^2c_N^{1/2}\norm[2]{\varphi}^2}$. We have thus shown that, since $r_Nc_N^{1/2}\cvgas{N}0$,
\begin{multline*}
\mathbb{E} \Bigg[ \dual{(\avg{\wt{s}})^2}{\left( \avg{(1-\wt{s})\varphi}\right)^2}
+ \dual{2\avg{\wt{s}}(1-\avg{\wt{s}})}{\left( \avg{(\tfrac{1}{2}-\wt{s})\varphi} \right)^2} \\
+ \dual{(1-\avg{\wt{s}})^2}{\left( \avg{\wt{s}\varphi} \right)^2} \Bigg]
= \tfrac{1}{2}\lambda(1-\lambda)\norm[2]{\avg{\varphi}}^2 + \littleO{r_N\norm[2]{\varphi}^2},
\end{multline*}
uniformly in $s\geq 0$. The result follows. 
\end{proof}

\appendix

\section{Approximating the (fractional) Laplacian}\label{append:average}

We use here the notation $\lesssim$ defined in \eqref{definition_lesssim}.
\begin{proposition} \label{prop:average}
Let $\phi:\R^{d}\to\R$ be twice continuously differentiable and suppose that $\norm[q]{\partial_{\beta}\phi}<\infty$ for $0\leq \abs{\beta}\leq 2$ and $1\leq q\leq \infty$. Then
\begin{enumerate}[label=\roman*)]
\item $\norm[q]{\avg{\phi}(r) -\phi} \leq \frac{d}{2} r^{2} \max_{\abs{\beta}=2}\norm[q]{\partial_{\beta}\phi}$.
\end{enumerate}
If in addition, $\phi$ admits $\norm[q]{\cdot}$-bounded derivatives of up to the fourth order,
\begin{enumerate}[label=\roman*),resume]
\item $\norm[q]{\dbavg{\phi}(r) -\phi-\frac{r^{2}}{d+2}\Delta\phi} \leq \frac{d^{3}}{3} r^{4}\max_{\abs{\beta}=4}\norm[q]{\partial_{\beta}\phi}$.
\end{enumerate}
\end{proposition}
\begin{proof}[Proof of Proposition~\ref{prop:average}]
By Taylor's theorem,
\begin{equation*}
\phi(y)=\phi(x)+\sum_{i=1}^{d}\partial_{i}\phi(x)(y-x)_{i} + \sum_{i,j}R_{ij}(y)(y-x)_{ij},
\end{equation*}
where $R_{ij}(y) = \intg{0}{1}(1-t)\partial_{ij}\phi(x+t(y-x))\ld t$ (we use the notation $x_{i_{1}\ldots i_{k}}=x_{i_{1}}\ldots x_{i_{k}}$). By symmetry, the integral of the first sum over a ball vanishes, and
\begin{align} \label{avg_1}
\abs{\avg{\phi}(x,r)-\phi(x)} &\leq \sum_{i,j}\frac{1}{V_{r}}\intbr{x}\abs{R_{ij}(y)}\abs{y-x}_{ij}\ld y.
\end{align}
If $q=\infty$, then $\abs{R_{ij}(y)}\leq \frac{1}{2}\norm[\infty]{\partial_{ij}\phi}$ and we write
\begin{align*}
\norm[\infty]{\avg{\phi}(r) -\phi} &\leq \frac{1}{2}d\max_{\abs{\beta}=2}\norm[\infty]{\partial_{\beta}\phi}\frac{1}{ V_r}\intbr{0}\abs{y}^2\ld y=\frac{d^2}{2(d+2)}r^2 \max_{\abs{\beta}=2}\norm[\infty]{\partial_{\beta}\phi}.
\end{align*}
If instead $1\leq q<\infty$, write
\begin{align*}
\norm[q]{\avg{\phi}(r) -\phi} &\leq \sum_{i,j}\left(\intrd\left(\frac{1}{V_{r}}\intbr{0}\abs{R_{ij}(x+y)}\abs{y}_{ij}\ld y\right)^{q}\ld x\right)^{1/q} \\
 &\leq \sum_{i,j}\left(\intrd\left(\frac{1}{V_{r}}\intbr{0}\abs{y}_{ij}\ld y\right)^{q-1}\frac{1}{V_{r}}\intbr{0}\abs{R_{ij}(x+y)}^{q}\abs{y}_{ij}\ld y\ld x\right)^{1/q},
\end{align*}
by Jensen's inequality.
But, by the definition of $R_{ij}$
\begin{align*}
\intrd\abs{R_{ij}(x+y)}^{q}\ld x &\leq \frac{1}{2^{q-1}}\intg{0}{1}(1-t)\intrd\abs{\partial_{ij}\phi(x+ty)}^{q}\ld x\ld t \\
&= \frac{1}{2^{q}}\norm[q]{\partial_{ij}\phi}^{q}.
\end{align*}
Plugging this into the previous inequality, we get
\begin{align*}
\norm[q]{\avg{\phi}(r) -\phi} &\leq \sum_{i,j}\frac{1}{2}\norm[q]{\partial_{ij}\phi}\left(\left(\frac{1}{V_{r}}\intbr{0}\abs{y}_{ij}\ld y\right)^{q-1}\frac{1}{V_{r}}\intbr{0}\abs{y}_{ij}\ld y\right)^{1/q} \\
 &\leq \frac{1}{2}d\max_{\abs{\beta}=2}\norm[q]{\partial_{\beta}\phi}\frac{1}{ V_r}\intbr{0}\abs{y}^2\ld y \\ 
 &\leq \frac{d}{2}r^{2}\max_{\abs{\beta}=2}\norm[q]{\partial_{\beta}\phi}.
\end{align*}

The second inequality is proved in essentially the same way. We expand $\phi$ according to Taylor's theorem to the fourth order:
\begin{multline*}
\phi(y)=\phi(x)+\sum_{i}\partial_{i}\phi(x)(y-x)_{i} + \frac{1}{2}\sum_{i,j}\partial_{ij}\phi(x)(y-x)_{ij} \\ +\frac{1}{3!}\sum_{i,j,k}\partial_{ijk}\phi(x)(y-x)_{ijk} + \sum_{ijkl}R_{ijkl}(y)(y-x)_{ijkl},
\end{multline*}
where $R_{ijkl}(y) = \frac{1}{3!}\intg{0}{1}(1-t)^{3}\partial_{ijkl}\phi(x+t(y-x))\ld t$. Integrating, all the antisymmetric terms vanish and we obtain
\begin{multline*}
\dbavg{\phi}(x,r)-\phi(x)=\frac{1}{2}\sum_{i}\partial_{ii}\phi(x)\frac{1}{V_{r}^{2}}\intrd{2}(y-x)_{ii}\inball{x}{z}{y}{z}\ld z\ld y \\ + \sum_{ijkl}\frac{1}{V_{r}^{2}}\intrd{2} R_{ijkl}(y)(y-x)_{ijkl}\inball{x}{z}{y}{z}\ld z\ld y.
\end{multline*}
We begin by calculating the first term before bounding the second one. Note that, by symmetry, the integral of $(y-x)_{ii}$ does not depend on $i$, so the first sum above can be written as
\begin{equation*}
\frac{1}{2}\Delta\phi(x)\frac{1}{dV_{r}^{2}}\intrd{2}\abs{y-x}^{2}\inball{x}{z}{y}{z}\ld z\ld y.
\end{equation*}
By the parallelogram identity, $\abs{y-x}^{2} = 2(\abs{x-z}^{2}+\abs{y-z}^{2})-\abs{2z-(x+y)}^{2}$. Integrating, we see that
\begin{align*}
\frac{1}{V_{r}^{2}}\intrd{2}\abs{y-x}^{2}\inball{x}{z}{y}{z}\ld z\ld y &= 4\frac{1}{V_{r}}\intbr{0}\abs{y}^{2}\ld y - \frac{1}{V_{r}^{2}}\intrd{2}\abs{(2z-y)-x}^{2}\inball{x}{z}{(2z-y)}{z}\ld z\ld y.
\end{align*}
Changing the variable of integration in the rightmost integral, we obtain
\begin{align*}
\frac{1}{V_{r}^{2}}\intrd{2}\abs{y-x}^{2}\inball{x}{z}{y}{z}\ld z\ld y &= \frac{2}{V_{r}}\intbr{0}\abs{y}^{2}\ld y \\
&= \frac{2d}{d+2}r^{2}.
\end{align*}
Replacing this term in the equation above, we can write
\begin{equation*}
\abs{\dbavg{\phi}(x,r)-\phi(x)-\frac{r^{2}}{d+2}\Delta\phi(x)} \leq \sum_{ijkl}\frac{1}{V_{r}^{2}}\intrd{2} \abs{R_{ijkl}(y)}\abs{y-x}_{ijkl}\inball{x}{z}{y}{z}\ld z\ld y.
\end{equation*}
Proceeding exactly as before and writing $\abs{y}_{ijkl}\leq \frac{1}{4}(\abs{y_{i}}^{4}+\abs{y_{j}}^{4}+\abs{y_{k}}^{4}+\abs{y_{l}}^{4})$, one shows that
\begin{align*}
\norm[q]{\dbavg{\phi}(r) -\phi-\frac{r^{2}}{d+2}\Delta\phi} &\leq \frac{d^{3}}{4!}\max_{\abs{\beta}=4}\norm[q]{\partial_{\beta}\phi}\sum_{i}\frac{1}{V_{r}^{2}}\intrd{2}\abs{y_{i}}^{4}\inball{}{z}{y}{z}\ld z\ld y.
\end{align*}
Note that $\sum_{i}\abs{y_{i}}^{4}\leq \abs{y}^{4}$, and by the parallelogram identity, $\abs{y}^{4}+\abs{2z-y}^{4} \leq 8(\abs{z}^{4}+\abs{z-y}^{4})$. As before, we can integrate on both sides:
\begin{equation*}
\frac{1}{V_{r}^{2}}\intrd{2}\abs{y}^{4}\inball{}{z}{y}{z}\ld z\ld y + \frac{1}{V_{r}^{2}}\intrd{2}\abs{2z-y}^{4}\inball{}{z}{(2z-y)}{z}\ld z\ld y \leq 16 \frac{1}{V_{r}}\intbr{0}\abs{y}^{4}\ld y.
\end{equation*}
Hence,
\begin{equation*}
\frac{1}{V_{r}^{2}}\intrd{2}\abs{y}^{4}\inball{}{z}{y}{z}\ld z\ld y \leq 8 \frac{1}{V_{r}}\intbr{0}\abs{y}^{4}\ld y = \frac{8d}{d+4}r^4.
\end{equation*}
\end{proof}

\begin{proposition}\label{prop:stable_average}
Take $\phi:\R^{d}\to\R$ to be twice continuously differentiable and suppose that $\norm[q]{\partial_{\beta}\phi}<\infty$ for $0\leq \abs{\beta}\leq 2$ and $q\in\braced{1,\infty}$. Then
\begin{enumerate}[label=\roman*)]
\item $\norm[q]{\L*\phi} \lesssim \norm[q]{\phi} + \max_{\abs{\beta}=2}\norm[q]{\partial_\beta\phi}$,
\item $\norm[q]{\L*\phi-\D\phi} \lesssim \delta^{2-\alpha} \max_{\abs{\beta}=2}\norm[q]{\partial_\beta\phi}$.
\end{enumerate}
Further if $0\leq \phi\leq 1$,
\begin{enumerate}[label=\roman*),resume]
\item $\norm[\infty]{F^{(\delta)}(\phi)-F(\phi)} \lesssim \delta^\alpha\left( 1 + \max_{\abs{\beta}=1}\norm[\infty]{\partial_\beta\phi}^2 + \max_{\abs{\beta}= 2}\norm[\infty]{\partial_\beta\phi} \right)$.
\end{enumerate}
\end{proposition}

\begin{proof}[Proof of Proposition~\ref{prop:stable_average}]
From the definition of $\L*$ and $\Phi^{(\delta)}$ in~\eqref{definition_fractional_laplacian}, note that
\begin{equation*}
\L*\phi(x) = V_1\intg{\delta}{\infty}\left( \dbavg{\phi}(x,r) - \phi(x) \right)\frac{\ld r}{r^{\alpha+1}}.
\end{equation*}
For $q\in\lbrace 1,\infty \rbrace$, using Proposition~\ref{prop:average}
\begin{align*}
\norm[q]{\L*\phi} &\leq V_1\intg{\delta}{\infty}\norm[q]{\dbavg{\phi}(r) -\phi}\frac{\ld r}{r^{\alpha+1}}\\
& \lesssim \max_{\abs{\beta}=2}\norm[q]{\partial_\beta\phi}\intg{\delta}{1}r^2\frac{\ld r}{r^{\alpha+1}} + \norm[q]{\phi} \intg{1}{\infty}\frac{\ld r}{r^{\alpha+1}}\\
& \lesssim \max_{\abs{\beta}=2}\norm[q]{\partial_\beta\phi} + \norm[q]{\phi}.
\end{align*}
Likewise, we have
\begin{equation*}
\D\phi(x)-\L*\phi(x) = V_1 \intg{0}{\delta}\left( \dbavg{\phi}(x,r)-\phi(x) \right)\frac{\ld r}{r^{\alpha+1}}.
\end{equation*}
By Proposition~\ref{prop:average}, we then write
\begin{align*}
\norm[q]{\D\phi-\L*\phi} &\leq V_1 \intg{0}{\delta}\norm[q]{\dbavg{\phi}(r) -\phi}\frac{\ld r}{r^{\alpha+1}}\\
& \lesssim \max_{\abs{\beta}=2}\norm[q]{\partial_\beta\phi} \intg{0}{\delta}r^2\frac{\ld r}{r^{\alpha+1}}\\
& \lesssim \delta^{2-\alpha}\max_{\abs{\beta}=2}\norm[q]{\partial_\beta\phi}.
\end{align*}
The third statement is a rewording of the first one in a slightly different setting. Indeed by~\eqref{definition_F_delta},
\begin{equation*}
F^{(\delta)}(\phi)(x)-F(\phi(x)) = \alpha\delta^{\alpha}\intg{\delta}{\infty}\left( \avg{F(\avg{\phi})}(x,r)-F(\phi(x)) \right)\frac{\ld r}{r^{\alpha+1}}.
\end{equation*}
Hence as in the proof of (\textit{i})
\begin{equation*}
\norm[\infty]{F^{(\delta)}(\phi)-F(\phi)} \lesssim \delta^{\alpha}\left( \norm[\infty]{F(\phi)} + \max_{\abs{\beta}=2}\norm[\infty]{\partial_\beta F(\avg{\phi})} + \norm[\infty]{F'}\max_{\abs{\beta}=2}\norm[\infty]{\partial_\beta\phi} \right).
\end{equation*}
The last term appears because there is an average inside the function $F$. The result then follows from the fact that $\partial_{ij}F(\phi) = \partial_{ij}\phi\, F'(\phi) + \partial_i\phi\,\partial_j\phi\, F''(\phi)$.
\end{proof}

\section{The centering term}\label{append:centering_term}

\subsection{The Brownian case}

\begin{proof}[Proof of Proposition~\ref{prop:convergence_centering_term}]
Recall the following expression for $f^{N}$ from~\eqref{centering_term_green_function},
\begin{align}\label{centering_term_green_function_rpt}
f^{N}_{t}(x) &= \G_{t}\ast w_{0}(x) - \intg{0}{t}\G_{t-s}\ast\avg{F(\avg{\ft{s}})}(x)\ld s.
\end{align}
Since $\norm[\infty]{\G_{t}\ast\phi}\leq \norm[\infty]{\phi}$, it follows that $\norm[\infty]{f^{N}_{t}}\leq \norm[\infty]{w_0}+T\norm[\infty]{F}$ for $t\leq T$.
We can now prove the second part of the statement by induction on $\abs{\beta}$. 
Suppose that the result is established for every $0\leq \abs{\beta}< k \leq 4$ and take $\beta$ such that $\abs{\beta}=k$. 
(From now on we omit the superscript N in the induction proof.) Noting that
\begin{equation*}
\partial_{\beta}F(f) = \sum_{k\geq 1}F^{(k)}(f)\left(\sum_{\alpha_{1}+\ldots+\alpha_{k}=\beta} \partial_{\alpha_{1}}f\ldots\partial_{\alpha_{k}}f\right),
\end{equation*}
and recalling that $w_0$ is assumed to have uniformly bounded derivatives of up to the fourth order, we can differentiate on both sides of \eqref{centering_term_green_function_rpt}:
\begin{align*}
\partial_{\beta}f_{t}(x) &= \G_{t}\ast\partial_{\beta}w_{0}(x) - \intg{0}{t}\G_{t-s}\ast\left( \avg{F'(\avg{f_{s}})\avg{\partial_{\beta}f_{s}}} + \sum_{\substack{ \alpha_{1}+\ldots+\alpha_{k}=\beta \\ k\geq 2} } \avg{F^{(k)}(\avg{f_{s}}) \avg{\partial_{\alpha_{1}}f_{s}} \ldots \avg{\partial_{\alpha_{k}}f_{s}}} \right)(x)\ld s.
\end{align*}
The sum is uniformly bounded by a constant $K$ by the induction hypothesis, and so, using the fact that $\norm[\infty]{\G_{t}\ast\phi}\leq \norm[\infty]{\phi}$,
\begin{equation*}
\norm[\infty]{\partial_{\beta}f_{t}} \leq \norm[\infty]{\partial_{\beta}w_{0}} + TK + \norm[\infty]{F'}\intg{0}{t}\norm[\infty]{\partial_{\beta}f_{s}}\ld s.
\end{equation*}
We can apply Gronwall's inequality to conclude
\begin{equation*}
\norm[\infty]{\partial_{\beta}f_{t}} \leq \left(\norm[\infty]{\partial_{\beta}w_{0}} + TK\right)e^{\norm{F'}T},
\end{equation*}
where the right hand side is independent of both $t\in[0,T]$ and $N\geq 1$. 
We can now prove the first statement using Gronwall's inequality again, together with Proposition~\ref{prop:average} and the first part of the proof. 
Recall that $G_{t}$ denotes the fundamental solution to the heat equation.
Recalling that we set the constants $uV_R$, $2R^2/(d+2)$ and $s$ to $1$, equations \eqref{definition:centering_term_fixed_radius} and \eqref{fisher_kpp} can be written as
\begin{equation*}
\ft(x) = G_t\ast w_{0}(x) + \intg{0}{t}G_{t-s}\ast\left( \calL\ft{s}-\frac{1}{2}\Delta\ft{s}-\avg{F(\avg{\ft{s}})} \right)(x)\ld s,
\end{equation*}
and
\begin{equation*}
f_{t}(x) = G_t\ast w_{0}(x) + \intg{0}{t}G_{t-s}\ast F(f_s)\ld s.
\end{equation*}
By Proposition~\ref{prop:average},
\begin{equation*}
\norm[\infty]{\calL \ft{s}-\frac{1}{2}\Delta \ft{s}} \leq \frac{d^{3}(d+2)}{6}r_{N}^{2}\max_{\abs{\beta}=4}\norm[\infty]{\partial_{\beta}\ft{s}} \lesssim r_N^2,
\end{equation*}
since $\max_{\abs{\beta}=4}\norm[\infty]{\partial_{\beta}\ft{s}}$ is uniformly bounded from the previous argument. 
Also by Proposition~\ref{prop:average},
\begin{equation*}
\norm[\infty]{\avg{F(\avg{\ft{s}})}-F(\ft{s})} \leq \frac{d}{2}r_N^{2}\left(\max_{\abs{\beta}=2}\norm[\infty]{\partial_{\beta}F(\avg{\ft{s}})} + \norm[\infty]{F'}\max_{\abs{\beta}=2}\norm[\infty]{\partial_{\beta}\ft{s}}\right) \lesssim r_N^2.
\end{equation*}
(The term within brackets is uniformly bounded from the first part of the proof.)
Finally, we also have
\begin{equation*}
\norm[\infty]{F(\ft{s})-F(f_s)} \leq \norm[\infty]{F'}\norm[\infty]{\ft{s}-f_s}.
\end{equation*}
Hence, using the fact that $\norm[\infty]{G_t\ast\phi}\leq\norm[\infty]{\phi}$, there exists a constant $C>0$ such that, for $t\in[0,T]$,
\begin{equation*}
\norm[\infty]{\ft-f_{t}} \leq C r_{N}^{2} + \norm[\infty]{F'}\intg{0}{t}\norm[\infty]{\ft{s}-f_{s}}\ld s.
\end{equation*}
Applying Gronwall's inequality,
\begin{align*}
\norm[\infty]{f^{N}_{t}-f_{t}} &\leq Ce^{\norm{F'}T}r_{N}^{2}.
\end{align*}
\end{proof}

\subsection{The stable case}

\begin{proof}[Proof of Proposition~\ref{prop:stable:convergence_centering_term}]
The proof of the convergence of the centering term in the stable case goes along the same lines as in the Brownian case of Proposition~\ref{prop:convergence_centering_term}. Differentiating \eqref{stable_centering_term_green_function} yields~:
\begin{multline*}
\partial_{\beta}\ft(x) = \G*_{t}\ast\partial_{\beta}w_{0}(x) - \alpha\intg{0}{t}\intg{1}{\infty}\G*_{t-s}\ast\bigg( \avg{F'(\avg{\ft{s}})\avg{\partial_{\beta}\ft{s}}}{\delta_N r} \\ + \sum_{\substack{ \alpha_{1}+\ldots+\alpha_{k}=\beta \\ k\geq 2}} \avg{F^{(k)}(\avg{\ft{s}}) \avg{\partial_{\alpha_{1}}\ft{s}} \ldots \avg{\partial_{\alpha_{k}}\ft{s}}}{\delta_N r} \bigg)(x)\frac{\ld r}{r^{\alpha+1}}\ld s.
\end{multline*}
One can then proceed by induction as previously to show
\begin{equation*}
\norm[\infty]{\partial_{\beta}f^N_{t}} \lesssim \norm[\infty]{\partial_{\beta}w_{0}} + T + \norm[\infty]{F'}\intg{0}{t}\norm[\infty]{\partial_{\beta}f_{s}}\ld s,
\end{equation*}
and Gronwall's inequality yields the second part of the statement. For the first part, the proof is identical to that in the Brownian case, one simply has to replace the operators $\frac{1}{2}\Delta$ and $\L$ by $\D$ and $\L*$, respectively, and likewise replace $\avg{F(\avg{\ft})}$ by $F^{(\delta)}(\ft)$. Proposition~\ref{prop:stable_average} then yields the correct estimates on the corresponding error terms.
\end{proof}

\section{Time dependent test functions}\label{append:test_functions}

\subsection{The Brownian case}

\begin{proof}[Proof of Lemma~\ref{lemma:convergence_varphi}]
The proof of Lemma~\ref{lemma:convergence_varphi} is similar in spirit to that of Proposition~\ref{prop:convergence_centering_term}. We start by proving the bound on the derivatives of $\varphi^N$. By the definition of $\varphi^{N}$ in~\eqref{definition_varphi_n},
\begin{equation} \label{varphi_green}
\varphi^{N}(x,s,t) = \G_{t-s}\ast\phi(x) - \intg{s}{t}\G_{u-s}\ast\avg{F'(\avg{\ft{u}})\avg{\varphi^{N}(u,t)}}(x)\ld u.
\end{equation}
Using the fact that $\G_{t}$ is a contraction in $L^{q}$, we have, for $q=1,2$,
\begin{align*}
\norm[q]{\varphi^{N}(s,t)}^{q} &\leq 2^{q-1}\norm[q]{\phi}^{q} + (2(t-s))^{q-1}\intg{s}{t}\norm[q]{\avg{F'(\avg{\ft{u}})\avg{\varphi^{N}(u,t)}}}^{q}\ld u \\
&\leq 2^{q-1}\norm[q]{\phi}^{q} + (2(t-s))^{q-1}\norm[\infty]{F'}^{q}\intg{s}{t}\norm[q]{\varphi^{N}(u,t)}^{q}\ld u.
\end{align*}
By Gronwall's inequality, we conclude that
\begin{equation*} 
\norm[q]{\varphi^{N}(s,t)}\leq 2^{(q-1)/q}\norm[q]{\phi}e^{\frac{2^{q-1}}{q}T^{q}\norm{F'}^{q}}.
\end{equation*}
Thus the statement holds for $\beta=0$. We can then proceed by induction on $\abs{\beta}$ as in the proof of Proposition~\ref{prop:convergence_centering_term} to show that the same holds for every $0\leq\abs{\beta}\leq 4$ (making use of the fact that by Proposition~\ref{prop:convergence_centering_term}, $\ft{}$ has uniformly bounded derivatives). We omit the details.

We are left with proving the convergence estimate for $\varphi^{N}$ which is again a Gronwall estimate. 
As in the proof of Proposition~\ref{prop:convergence_centering_term}, write~\eqref{definition_varphi_n} and~\eqref{definition_varphi} as
\begin{equation}\label{varphi_n_gronwall}
\varphi^N(x,s,t) = G_{t-s}\ast\phi(x) + \intg{s}{t}G_{u-s}\ast\left( \L\varphi^N(u,t)-\frac{1}{2}\Delta\varphi^N(u,t) - \avg{F'(\avg{\ft{u}})\avg{\varphi^N(u,t)}} \right)(x)\ld u,
\end{equation}
and
\begin{equation}\label{varphi_gronwall}
\varphi(x,s,t) = G_{t-s}\ast\phi(x) - \intg{s}{t}G_{u-s}\ast\left( F'(f_u)\varphi(u,t) \right)(x)\ld u.
\end{equation}
By Proposition~\ref{prop:average} and the bound on the spatial derivatives of $\varphi^N$,
\begin{equation*}
\norm[q]{\L\varphi^N(u,t)-\frac{1}{2}\Delta\varphi^N(u,t)} \lesssim r_N^2.
\end{equation*}
Still by Proposition~\ref{prop:average}, (omitting superscripts $N$ and time variables)
\begin{multline*}
\norm[q]{\avg{F'(\avg{f})\avg{\varphi}}-F'(f)\varphi} \\ \leq \frac{d}{2}r_{N}^{2}\left(\max_{\abs{\beta}=2}\norm[q]{\partial_{\beta}(F'(\avg{f})\avg{\varphi})} + \norm[\infty]{F'}\max_{\abs{\beta}=2}\norm[q]{\partial_{\beta}\varphi} + \norm[q]{\varphi}\norm[\infty]{F''}\max_{\abs{\beta}=2}\norm[\infty]{\partial_{\beta}f}\right).
\end{multline*}
The last term inside the brackets is uniformly bounded by Proposition~\ref{prop:convergence_centering_term} and the second to last is bounded as a consequence of the first part of the proof. 
Also, $\partial_{ij}(F'(\avg{f})\avg{\varphi})$ is dominated by a linear combination of (averages of) derivatives of both $f$ and $\varphi$. 
The latter are bounded in $L^{q}$ while the former are bounded in $L^{\infty}$, hence the first term within the brackets is also uniformly bounded. 
To sum up,
\begin{equation}\label{estimate_avg_F(phi)}
\norm[q]{\avg{F'(\avg{\ft{u}})\avg{\varphi^{N}(u,t)}}-F'(\ft{u})\varphi^{N}(u,t)} \lesssim r_{N}^{2}.
\end{equation}
Finally, by Proposition~\ref{prop:convergence_centering_term},
\begin{equation*}
\norm[q]{F'(\ft{u})-F'(f_u)} \lesssim r_N^2.
\end{equation*}
Hence, subtracting~\eqref{varphi_gronwall} from~\eqref{varphi_n_gronwall} and using Jensen's inequality as above with the $L^q$-contraction property of $G_t$, we have, for $t\in[0,T]$,
\begin{equation*}
\norm[q]{\varphi^N(s,t)-\varphi(s,t)}^q \lesssim r_N^{2q} + \intg{s}{t}\norm[q]{\varphi^N(u,t)-\varphi(u,t)}^q\ld u.
\end{equation*}
We conclude with Gronwall's inequality, yielding the first statement of Lemma~\ref{lemma:convergence_varphi}.
\end{proof}

\begin{proof}[Proof of Lemma~\ref{lemma:continuity_varphi}]
We can assume that $t'>t\geq s$ (if $t'\geq s\geq t$, then $\varphi^{N}(s,t)=\phi=\varphi^{N}(s,s)$ and the problem reduces to bounding $\varphi^{N}(s,t')-\varphi^{N}(s,s)$). Using \eqref{varphi_green} and recalling the way we extended $\varphi^{N}$ in~\eqref{varphi_extension}, we write
\begin{multline*}
\varphi^{N}(x,s,t')-\varphi^{N}(x,s,t) = \G_{t'-s}\ast\phi(x)-\G_{t-s}\ast\phi(x) - \intg{t}{t'}\G_{u-s}\ast\avg{F'(\avg{\ft{u}})\avg{\phi}}(x)\ld u \\ - \intg{s}{T}\G_{u-s}\ast\left(\avg{F'(\avg{\ft{u}})(\avg{\varphi^{N}(u,t')}-\avg{\varphi^{N}(u,t)})}\right)(x)\ld u.
\end{multline*}
Again, we use the $L^{q}$-contraction property of $\G_{t}$ to write
\begin{multline} \label{continuity_step1}
\norm[q]{\varphi^{N}(s,t')-\varphi^{N}(s,t)}^{q} \leq 3^{q-1}\norm[q]{\G_{t'-s}\ast\phi-\G_{t-s}\ast\phi}^{q} + 3^{q-1}\abs{t'-t}^{q}\norm[\infty]{F'}^{q}\norm[q]{\phi}^{q} \\ + (3(T-s))^{q-1}\norm[\infty]{F'}^{q}\intg{s}{T}\norm[q]{\varphi^{N}(u,t')-\varphi^{N}(u,t)}^{q}\ld u.
\end{multline}
We need a bound on the first term; recalling the definition of $\G$ in Subsection~\ref{subsec:brownian:regularity_estimate}, we have
\begin{equation*}
\G_{t'-s}\ast\phi(x)-\G_{t-s}\ast\phi(x) = \intg{t}{t'}\G_{u-s}\ast\L\phi(x)\ld u.
\end{equation*}
By Jensen's inequality,
\begin{align*}
\norm[q]{\G_{t'-s}\ast\phi-\G_{t-s}\ast\phi}^{q} &\leq \abs{t'-t}^{q-1}\intg{t}{t'}\norm[q]{\L\phi}^q\ld u \\
&\lesssim \abs{t'-t}^{q},
\end{align*}
by Proposition~\ref{prop:average}.
Hence, returning to \eqref{continuity_step1},
\begin{equation*}
\norm[q]{\varphi^{N}(s,t')-\varphi^{N}(s,t)}^{q} \lesssim \abs{t'-t}^{q} + \intg{s}{T}\norm[q]{\varphi^{N}(u,t')-\varphi^{N}(u,t)}^{q}\ld u.
\end{equation*}
Gronwall's inequality now yields the result.
\end{proof}

\begin{proof}[Proof of Lemma~\ref{lemma_varphi_uniform_bound}]
Using \eqref{varphi_green} and the definition of $\G$ we can write
\begin{equation*}
\varphi^{N}(x,s,t) = \phi(x) + \intg{s}{t}\G_{u-s}\ast\calL\phi(x)\ld u - \intg{s}{t}\G_{u-s}\ast\avg{F'(\avg{\ft{u}})\avg{\varphi^{N}(u,t)}}(x)\ld u.
\end{equation*}
Within the second integral, $u\leq t$, so we can write $\abs{\varphi^{N}(u,t)} \leq \sup_{t'\in[u,T]}\abs{\varphi^{N}(u,t')}$. Thus
\begin{equation*}
\sup_{t\in[s,T]}\abs{\varphi^{N}(x,s,t)} \leq \abs{\phi(x)} + \intg{s}{T}\G_{u-s}\ast\abs{\calL\phi}(x)\ld u + \norm[\infty]{F'}\intg{s}{T}\dbavg{\G_{u-s}}\ast\sup_{t\in[u,T]}\abs{\varphi^{N}(u,t)}(x)\ld u.
\end{equation*}
Integrating with respect to the variable $x\in\R^{d}$ yields
\begin{equation*}
\norm[1]{\sup_{t\in[s,T]}\abs{\varphi^{N}(s,t)}} \leq \norm[1]{\phi} + (T-s)\norm[1]{\calL\phi} + \norm[\infty]{F'}\intg{s}{T}\norm[1]{\sup_{t\in[u,T]}\abs{\varphi^{N}(u,t)}}\ld u.
\end{equation*}
By Gronwall's inequality,
\begin{equation*}
\norm[1]{\sup_{t\in[s,T]}\abs{\varphi^{N}(s,t)}} \leq \left(\norm[1]{\phi} + T\frac{d(d+2)}{2}\max_{\abs{\beta}=2}\norm[1]{\partial_{\beta}\phi}\right)e^{\norm{F'}(T-s)}.
\end{equation*}
(we have used Proposition~\ref{prop:average} to bound $\norm[1]{\calL\phi}$).
\end{proof}

\subsection{The stable case}

\begin{proof}[Proof of Lemma~\ref{lemma:stable:convergence_varphi}]
By the definition of $\varphi^N$ in~\eqref{defn_varphi_n_stable},
\begin{equation} \label{varphi_green_sr}
\varphi^{N}(x,s,t) = \G*_{t-s}\ast\phi(x) - \alpha\intg{s}{t}\intg{1}{\infty}\G*_{u-s}\ast\avg{F'(\avg{\ft{u}})\avg{\varphi^{N}(u,t)}}{\delta r}(x)\frac{\ld r}{r^{\alpha+1}}\ld u.
\end{equation}
The bound on the derivatives of $\varphi^N$ is proved following the same argument as in the proof of Lemma \ref{lemma:convergence_varphi} in the Brownian case (simply note that since we are only considering $q\in\braced{1,\infty}$, we can safely put the $L^q$ norm inside the integral with respect to $r$). 
By the definition of $\varphi$,
\begin{equation*}
\varphi(x,s,t) = \mathcal{G}^{(\alpha)}_{t-s} \ast \phi(x) - \intg{s}{t} \mathcal{G}^{(\alpha)}_{u-s} \ast \left( F'(f_u) \varphi(u,t) \right)(x) \ld u.
\end{equation*}
By Proposition~\ref{prop:stable_average} and by the bound on the spatial derivatives of $\varphi^N$,
\begin{equation*}
\norm[q]{ \L*\varphi^N(u,t) - \D\varphi^N(u,t) } \lesssim \delta_N^{2-\alpha}.
\end{equation*}
Using \eqref{estimate_avg_F(phi)} (which is still true in this case by the bound on the derivatives of $\varphi^N$), we have
\begin{align*}
\intg{1}{\infty}\norm[q]{\avg{F'(\avg{\ft{u}})\avg{\varphi^{N}(u,t)}}{\delta r}-F'(\ft{u})\varphi^{N}(u,t)}\frac{\ld r}{r^{\alpha+1}} &\lesssim \delta^2\intg{1}{\delta^{-1}} r^2 \frac{\ld r}{r^{\alpha+1}} +\intg{\delta^{-1}}{\infty}\frac{\ld r}{r^{\alpha+1}} \\
&\lesssim \delta^{\alpha}.
\end{align*}
Finally, by Proposition~\ref{prop:stable:convergence_centering_term},
\begin{equation*}
\norm[q]{F'(\ft{u})-F'(f_u)} \lesssim \delta_N^{\alpha\wedge(2-\alpha)}.
\end{equation*}
As a result, by the same argument as in the proof of Lemma~\ref{lemma:convergence_varphi}, by Gronwall's inequality,
\begin{equation*}
\norm[q]{\varphi^N(s,t)-\varphi(s,t)} \lesssim \delta_N^{\alpha\wedge(2-\alpha)}.
\end{equation*}
\end{proof}

\begin{proof}[Proof of Lemma~\ref{lemma:continuity_unif_bound_varphi_sr}]
The argument for the continuity estimate is the same as in the proof of Lemma \ref{lemma:continuity_varphi}, using Proposition \ref{prop:stable_average}. For the second bound, we use the same argument as in Lemma~\ref{lemma_varphi_uniform_bound}, again using Proposition \ref{prop:stable_average}.
\end{proof}

\begin{proof}[Proof of Lemma~\ref{lem:alpha_int_bound}]
Splitting the integral with respect to $z_2$, we have
\begin{multline*}
\intrd{2}\abs{f(z_{1})}\abs{g(z_{2})}|z_1-z_2|^{-\alpha}\ld z_{1}\ld z_{2} \leq \| g \|_\infty \intrd |f (z_1)|  \int_{B(z_1,1)} |z_1 -z_2|^{-\alpha}\ld z_2\ld z_1 \\ + \intrd\abs{f(z_1)}\int_{\R ^d \setminus B(z_1,1)} |g (z_2)| \ld z_2\ld z_1
\end{multline*}
But $\intbr{0}{1}\abs{y}^{-\alpha}\ld y = \frac{dV_1}{d-\alpha}$ and we have~:
\begin{equation*}
\intrd{2}\abs{f(z_{1})}\abs{g(z_{2})}|z_1-z_2|^{-\alpha}\ld z_{1}\ld z_{2} \leq \norm[\infty]{g}\frac{dV_1}{d-\alpha}\intrd\abs{f(z_1)}\ld z_1 + \norm[1]{g}\intrd\abs{f(z_1)}\ld z_1.
\end{equation*}
\end{proof}

\section{Estimates for drift load proofs}\label{append:driftload}

\begin{proof}[Proof of Lemma~\ref{lem:density_green_fct}]
For all $t>0$, $\xi^{(r)}_t$ can be written as $\xi^{(r)}_t = \sum_{k=1}^{N_t} Y_k$, where $\proc{N}$ is a Poisson process with intensity $\frac{(d+2)}{2 r^2}$ and $\proc{Y_k}{k \geq 1}$ is a sequence of independent and identically distributed random variables with density $\psi(y) = \frac{V_r(0,y)}{V_r^2}$.
As a result, the law of $\xi^{(r)}_t$ can be written
\begin{equation*}
\G_t (\ld x) = e^{- \frac{d+2}{2 r^2} t} \delta_0 (\ld x) + e^{- \frac{d+2}{2 r^2} t} \sum_{n \geq 1} \frac{\left( \frac{d+2}{2 r^2} t \right)^n}{n !} \psi^{\ast n} (x) \ld x.
\end{equation*}
Since $\psi$ is continuous on $\R^d$, so is $\psi^{\ast n}$ for any $n \geq 1$.
In addition, $\psi(y)$ is decreasing as a function of $\abs{y}$, and $\phi \ast \psi (x) = \dbavg{\phi} (x, r)$ so, by induction it follows that $\psi^{\ast n} (y)$ is also decreasing as a function of $\abs{y}$.
Since the above sum converges uniformly, we can conclude that $g^{(r)}_t$ is continuous on $\R^d$ and that $g^{(r)}_t (y)$ is a decreasing function of $\abs{y}$.
\end{proof}

\begin{proof}[Proof of Lemma \ref{lemma:magic}] 
By some elementary algebra,
\begin{multline*}
 \phi(y)^2-\phi(x)^2-2\phi(x)(\phi(y)-\phi(x)) + 2\frac{2r^2}{d+2}\phi(x)(\phi(y)-\phi(x))g(y) \\ 
 \begin{aligned} 
 &= \left(\phi(y)-\phi(x) + \frac{2r^2}{d+2}\phi(x)g(y)\right)^2 - \left(\frac{2r^2}{d+2}\right)^2\phi(x)^2 g(y)^2\\
 &\geq - \left(\frac{2r^2}{d+2}\right)^2\phi(x)^2,
 \end{aligned}
\end{multline*}
since $g(y)^2\leq 1$. Averaging the above inequality in $y$ twice around $x$ and multiplying by $\frac{d+2}{2r^2}$ yields
\begin{equation*}
\L\phi^2(x) - 2\phi(x)\L\phi(x)+2\phi(x)\dbavg{\phi\,g}(x,r) - 2\phi(x)^2\dbavg{g}(x,r) \geq -\frac{2r^2}{d+2}\phi(x)^2.
\end{equation*}
The first result then follows from the fact that $\gamma\leq g$. For the second inequality, set $a=\phi(y)$, $\epsilon = \frac{2r^2}{d+2}g(y)$ and $b=(1-\epsilon)^{1/3}\phi(x)$; then
\begin{equation*}
 \phi(y)^4-\phi(x)^4 - 4\left(1-\frac{2r^2}{d+2}g(y)\right)\phi(x)^3(\phi(y)-\phi(x)) = a^4-b^4-4b^3(a-b) + b^4-\phi(x)^4 - 4b^3(b-\phi(x)).
\end{equation*}
By convexity of the function $x\mapsto x^4$, $a^4-b^4-4b^3(a-b)\geq 0$, so the above expression is greater than
\begin{equation*}
 \phi(x)^4\left[ (1-\epsilon)^{4/3}-1 - 4(1-\epsilon)((1-\epsilon)^{1/3}-1) \right] \underset{\epsilon\to 0}{\sim} -\frac{2}{3}\phi(x)^4 \epsilon^2.
\end{equation*}
Hence there exists $c$ such that, for $r$ small enough,
\begin{equation*}
 \phi(y)^4-\phi(x)^4 - 4\left(1-\frac{2r^2}{d+2}g(y)\right)\phi(x)^3(\phi(y)-\phi(x)) \geq -4c\,r^4\phi(x)^4.
\end{equation*}
Averaging in $y$ twice around $x$ as above yields the second statement.
\end{proof}
\begin{proof}[Proof of Lemma \ref{lemma_differential_inequality}]
We define the following~:
\begin{equation*}
\calH(x,u,t) = e^{-\alpha(t-u)}G^{(r)}_{t-u}\ast h_{u}(x).
\end{equation*}
Differentiating with respect to $u$ yields
\begin{align} \label{diff_ineq_eqn}
\deriv*{\calH}{u}(x,u,t) &= e^{-\alpha(t-u)}G^{(r)}_{t-u}\ast\left(\partial_{u}h_{u}-\calL h_{u} + \alpha h_{u}\right)(x) \nonumber \\
&\leq e^{-\alpha(t-u)}G^{(r)}_{t-u}\ast g_{u}(x).
\end{align}
Integrating \eqref{diff_ineq_eqn} over $u\in[s,t]$, we have
\begin{align} \label{h_green_ineq}
h_{t}(x) &\leq e^{-\alpha(t-s)}G^{(r)}_{t-s}\ast h_{s}(x) + \intg{s}{t}e^{-\alpha(t-u)}G^{(r)}_{t-u}\ast g_{u}(x)\ld u.
\end{align}
By Jensen's inequality, 
\begin{align*}
\left(\intg{s}{t}e^{-\alpha(t-u)}G^{(r)}_{t-u}\ast g_{u}(x)\ld u\right)^{q} &\leq \left(\intg{s}{t}e^{-\alpha(t-u)}\ld u\right)^{q-1}\intg{s}{t}e^{-\alpha(t-u)}\left(G^{(r)}_{t-u}\ast g_{u}(x)\right)^{q}\ld u \\
&\leq \frac{1}{\alpha^{q-1}}\intg{s}{t}e^{-\alpha(t-u)}G^{(r)}_{t-u}\ast g_{u}^{q}(x)\ld u.
\end{align*}
The result follows by taking $\norm[q]{\cdot}$ norms on each side of \eqref{h_green_ineq}.
\end{proof}

\bibliography{biblio}

\begin{thebibliography}{DMFL86}

\bibitem[AK13]{achleitner_traveling_2013}
Franz Achleitner and Christian Kuehn.
\newblock Traveling waves for a bistable equation with nonlocal-diffusion.
\newblock {\em arXiv preprint arXiv:1312.6304}, 2013.

\bibitem[Ald78]{aldous_stopping_1978}
David Aldous.
\newblock Stopping times and tightness.
\newblock {\em The Annals of Probability}, 6(2):335--340, 1978.

\bibitem[BEK06]{bierme_about_2006}
Hermine Biermé, Anne Estrade, and Ingemar Kaj.
\newblock About scaling behavior of random balls models.
\newblock In {\em Proceed. 6th {Int}. {Conf}. on {Stereology}, {Spatial}
  {Statistics} and {Stochastic} {Geometry}, {Prague}}, 2006.

\bibitem[BEV10]{barton_new_2010}
Nick~H. Barton, Alison~M. Etheridge, and Amandine Véber.
\newblock A new model for evolution in a spatial continuum.
\newblock {\em Electron. J. Probab}, 15(7), 2010.

\bibitem[BEV13a]{barton_modeling_2013}
N.~H. Barton, A.~M. Etheridge, and A.~Véber.
\newblock Modeling evolution in a spatial continuum.
\newblock {\em Journal of Statistical Mechanics: Theory and Experiment},
  2013(01):P01002, 2013.

\bibitem[BEV13b]{berestycki_large_2013}
N.~Berestycki, A.~M. Etheridge, and A.~Véber.
\newblock Large scale behaviour of the spatial {Lambda}-{Fleming}–{Viot}
  process.
\newblock In {\em Annales de l'{Institut} {Henri} {Poincaré}, {Probabilités}
  et {Statistiques}}, volume~49, pages 374--401. Institut Henri Poincaré,
  2013.

\bibitem[BLG03]{bertoin_stochastic_2003}
Jean Bertoin and Jean-François Le~Gall.
\newblock Stochastic flows associated to coalescent processes.
\newblock {\em Probability Theory and Related Fields}, 126(2):261--288, 2003.

\bibitem[Chm13]{chmaj_existence_2013}
Adam Chmaj.
\newblock Existence of traveling waves in the fractional bistable equation.
\newblock {\em Archiv der Mathematik}, 100(5):473--480, 2013.

\bibitem[DMFL86]{de_masi_reaction-diffusion_1986}
A.~De~Masi, P.~A. Ferrari, and J.~L. Lebowitz.
\newblock Reaction-diffusion equations for interacting particle systems.
\newblock {\em Journal of statistical physics}, 44(3-4):589--644, 1986.

\bibitem[EFPS15]{etheridge_branching_2015}
Alison Etheridge, Nic Freeman, Sarah Penington, and Daniel Straulino.
\newblock Branching {Brownian} motion and {Selection} in the {Spatial}
  {Lambda}-{Fleming}-{Viot} {Process}.
\newblock {\em arXiv preprint arXiv:1512.03766}, 2015.

\bibitem[EFS15]{etheridge_brownian_2015}
Alison Etheridge, Nic Freeman, and Daniel Straulino.
\newblock The {Brownian} {Net} and {Selection} in the {Spatial}
  {Lambda}-{Fleming}-{Viot} {Process}.
\newblock {\em arXiv preprint arXiv:1506.01158}, 2015.

\bibitem[EK86]{ethier_markov_1986}
Stewart~N. Ethier and Thomas~G. Kurtz.
\newblock {\em Markov processes: characterization and convergence}, volume 282.
\newblock John Wiley \& Sons, 1986.

\bibitem[Eth08]{etheridge_drift_2008}
Alison~M. Etheridge.
\newblock Drift, draft and structure: some mathematical models of evolution.
\newblock {\em Banach Center Publ}, 80:121--144, 2008.

\bibitem[Eth09]{etheridge_mathematical_2009}
Alison Etheridge.
\newblock Some mathematical models from population genetics.
\newblock {\em Ecole d’Eté de Probabilités de Saint-Flour XXXIX—2009.},
  2009.

\bibitem[EVY14]{etheridge_rescaling_2014}
Alison Etheridge, Amandine Veber, and Feng Yu.
\newblock Rescaling limits of the spatial {Lambda}-{Fleming}-{Viot} process
  with selection.
\newblock {\em arXiv preprint arXiv:1406.5884}, 2014.

\bibitem[Fel51]{feller_diffusion_1951}
William Feller.
\newblock Diffusion processes in genetics.
\newblock In {\em Proc. {Second} {Berkeley} {Symp}. {Math}. {Statist}. {Prob}},
  volume 227, page 246, 1951.

\bibitem[Fis37]{fisher_wave_1937}
Ronald~Aylmer Fisher.
\newblock The wave of advance of advantageous genes.
\newblock {\em Annals of Eugenics}, 7(4):355--369, 1937.

\bibitem[JS87]{jacod_limit_1987}
Jean Jacod and Albert~N. Shiryaev.
\newblock {\em Limit theorems for stochastic processes}, volume 1943877.
\newblock Springer Berlin, 1987.

\bibitem[Kim53]{kimura_stepping-stone_1953}
M.~Kimura.
\newblock Stepping-stone model of population.
\newblock {\em Annual Report of the National Institute of Genetics}, 3:62--63,
  1953.

\bibitem[Kur71]{kurtz_limit_1971}
Thomas~G. Kurtz.
\newblock Limit theorems for sequences of jump {Markov} processes approximating
  ordinary differential processes.
\newblock {\em Journal of Applied Probability}, 8(2):344--356, 1971.

\bibitem[Mal48]{malecot_mathematiques_1948}
Gustave Malécot.
\newblock {\em Mathématiques de l'hérédité}.
\newblock Barnéoud frères, 1948.

\bibitem[MT95]{mueller_stochastic_1995}
Carl Mueller and Roger Tribe.
\newblock Stochastic pde's arising from the long range contact and long range
  voter processes.
\newblock {\em Probability theory and related fields}, 102(4):519--545, 1995.

\bibitem[Nor74a]{norman_central_1974}
M.~Frank Norman.
\newblock A central limit theorem for {Markov} processes that move by small
  steps.
\newblock {\em The Annals of Probability}, pages 1065--1074, 1974.

\bibitem[Nor74b]{norman_markovian_1974}
M.~Frank Norman.
\newblock Markovian learning processes.
\newblock {\em SIAM Review}, 16(2):143--162, 1974.

\bibitem[Nor75a]{norman_approximation_1975}
M.~Frank Norman.
\newblock Approximation of stochastic processes by {Gaussian} diffusions, and
  applications to {Wright}-{Fisher} genetic models.
\newblock {\em SIAM Journal on Applied Mathematics}, 29(2):225--242, 1975.

\bibitem[Nor75b]{norman_limit_1975}
M.~Frank Norman.
\newblock Limit theorems for stationary distributions.
\newblock {\em Advances in Applied Probability}, pages 561--575, 1975.

\bibitem[Nor77]{norman_ergodicity_1977}
M.~Frank Norman.
\newblock Ergodicity of diffusion and temporal uniformity of diffusion
  approximation.
\newblock {\em Journal of Applied Probability}, pages 399--404, 1977.

\bibitem[Rob70]{robertson_reduction_1970}
Alan Robertson.
\newblock The reduction in fitness from genetic drift at heterotic loci in
  small populations.
\newblock {\em Genetical research}, 15(02):257--259, 1970.

\bibitem[SKM93]{samko_fractional_1993}
Stefan~G. Samko, Anatoly~A. Kilbas, and Oleg~I. Marichev.
\newblock Fractional integrals and derivatives.
\newblock {\em Theory and Applications, Gordon and Breach, Yverdon}, 1993.

\bibitem[VW15]{veber_spatial_2015}
Amandine Veber and Anton Wakolbinger.
\newblock The spatial {Lambda}-{Fleming}–{Viot} process: {An} event-based
  construction and a lookdown representation.
\newblock In {\em Annales de l'{Institut} {Henri} {Poincaré}, {Probabilités}
  et {Statistiques}}, volume~51, pages 570--598. Institut Henri Poincaré,
  2015.

\bibitem[Wal86]{walsh_introduction_1986}
John~B. Walsh.
\newblock {\em An introduction to stochastic partial differential equations}.
\newblock Springer, 1986.

\bibitem[Wri43]{wright_isolation_1943}
Sewall Wright.
\newblock Isolation by distance.
\newblock {\em Genetics}, 28(2):114, 1943.

\end{thebibliography}

\end{document}